\documentclass[a4paper,11pt]{article}

\usepackage[dvips]{graphicx}
\usepackage{amssymb,amsmath,bm,enumerate}
\usepackage{amsfonts}
\usepackage{amssymb}
\usepackage{latexsym}
\usepackage{float}
\usepackage[FIGTOPCAP]{subfigure}
\usepackage{placeins}
\usepackage{geometry}

\usepackage{pstricks}

\usepackage{xspace}

\usepackage{color}

\newcommand{\note}[1]%
{\noindent\centerline{\fbox{\parbox{.9\textwidth}{\textbf{#1}}}}}
\newcommand{\snote}[1]%
{\fbox{\textbf{#1}}}

\def\proof{\par{\noindent\it Proof}. \ignorespaces}
\def\endproof{\vbox{\hrule height0.6pt\hbox{%
   \vrule height1.3ex width0.6pt\hskip0.8ex
   \vrule width0.6pt}\hrule height0.6pt}}

\input definitions.sty

\begin{document}

\title{Adaptive Fourier-Galerkin Methods}

\author{Claudio Canuto$^a$, Ricardo H. Nochetto$^b$ and Marco Verani$^c$}

\date{January 26, 2012}

\maketitle 

\begin{center}
{\small
$^a$ Dipartimento di Scienze Matematiche, Politecnico di Torino\\
        Corso Duca degli Abruzzi 24, 10129 Torino, Italy\\
	E-mail: {\tt claudio.canuto@polito.it}\\
\vskip 0.1cm
Department of Mathematics and Institute for Physical Science
  and Technology,\\ University of Maryland, College Park, MD 20742, USAy\\
E-mail: {\tt rhn@math.umd.edu} 
\vskip 0.1cm
$^c$ MOX, Dipartimento di Matematica, Politecnico di Milano\\
Piazza Leonardo da Vinci 32, I-20133 Milano, Italy\\
E-mail: {\tt marco.verani@polimi.it}\\
}
\end{center}

\begin{abstract}
\noindent We study the performance of adaptive Fourier-Galerkin methods in a
periodic box in $\mathbb{R}^d$ with dimension $d\ge1$. These methods offer
unlimited approximation power only restricted by solution and data 
regularity. They are of intrinsic interest but are also a first step
towards understanding adaptivity for the $hp$-FEM.
We examine two nonlinear approximation classes, one
classical corresponding to algebraic decay of Fourier coefficients
and another associated with
exponential decay. We study the sparsity classes of the residual and
show that they are the same as the solution for the algebraic class but
not for the exponential one. This possible sparsity degradation for
the exponential class can be
compensated with coarsening, which we discuss in detail. We present
several adaptive Fourier algorithms, and prove their contraction 
and optimal cardinality properties.
\end{abstract}


\section{Introduction}\label{S:intro}
%
Adaptivity is now a fundamental tool in scientific and engineering
computation. In contrast to the practice, which goes
back to the 70's, the mathematical theory for multidimensional problems is
rather recent. It started in 1996 with the convergence
results by D\"orfler \cite{dorfler:96} and Morin, Nochetto, and Siebert \cite{MNS:00}.
The first convergence rates were derived by Cohen, Dahmen, and DeVore
\cite{CDDV:1998} for wavelets in any dimensions $d$, and for finite element
methods (AFEM)
by Binev, Dahmen, and DeVore \cite{BDD:04} for
$d=2$ and Stevenson \cite{Stevenson:2007} for any $d$. 
The most comprehensive results
for AFEM are those of Casc\'on, Kreuzer, Nochetto,
and Siebert \cite{Nochetto-et-al:2008} for any $d$ and $L^2$ data,
and Cohen, DeVore, and Nochetto \cite{CDN:11} for $d=2$ and $H^{-1}$
data; we refer to the survey \cite{NSV:09} by Nochetto, Siebert and Veeser.
This theory is quite satisfactory in that it shows that 
AFEM delivers a convergence rate compatible with that 
of the approximation classes where the solution and data 
belong. The recent results in \cite{CDN:11} reveal that it is the
approximation class of the solution that really matters. In all cases
though the convergence rates are limited by the approximation power of
the method (both wavelets and FEM), which is finite and related to the
polynomial degree of the basis functions, and the regularity of the
solution and data. The latter is always measured in an {\it algebraic} 
approximation class.

In contrast very little is known for methods with infinite
approximation power, such as those based on Fourier analysis. We
mention here the results of DeVore and Temlyakov \cite{DeVore-Temlyakov:1995} for
trigonometric sums and those of Binev et al \cite{BCDDPW:10} for the reduced
basis method. A close relative to Fourier methods is the so-called
$p$-version of the FEM (see e.g.  \cite{Schwab:1998} and
\cite{CHQZ2:2006}), 
which uses Legendre polynomials instead of
exponentials as basis functions. The purpose of this paper is to
present {\it adaptive Fourier-Galerkin methods (ADFOUR)}, and discuss 
their convergence and optimality properties. We do so in the context
of both {\it algebraic} and {\it exponential} approximation classes,
and take advantage of the orthogonality inherent to complex
exponentials. We believe that this approach can be extended to
the $p$-FEM. We view this theory as a first step towards understanding 
adaptivity for the
$hp$-FEM, which combines mesh refinement ($h$-FEM) with polynomial
enrichment ($p$-FEM) and is much harder to analyze.

Our investigation reveals some striking differences between ADFOUR and
AFEM and wavelet methods. 
The basic assumption, underlying the success of adaptivity, is that
the information read in the residual is quasi-optimal for either mesh design
or choosing wavelet coefficients for the actual solution.
This entails that the sparsity classes of the residual and the
solution coincide. We briefly illustrate below, and fully discuss
later in Sect. \ref{sec:spars-res},
that this basic premise is false for exponential classes even though
it is true for algebraic classes.
Confronted with this unexpected fact, we have no alternative but to 
implement and study ADFOUR with {\it coarsening} for the exponential
case; see Sect. \ref{S:coarsening} and Sect. \ref{sec:adfour-coarse}. 
This was the original idea of Cohen et al \cite{CDDV:1998} 
and Binev et al \cite{BDD:04} for
the algebraic case, but it was subsequently removed by Stevenson
\cite{Stevenson:2007}. 

We give now a brief description of the essential issues we are
confronted with in designing and studying ADFOUR. To this end, we 
assume that we know the Fourier representation 
$\bv=\{v_k\}_{k\in\mathbb{Z}}$ of a periodic function $v$, 
and its non-increasing rearrangement $\bv^*=\{v_n^*\}_{n=1}^\infty$, namely,
$|v_{n+1}^*|\le |v_n^*|$ for all $n\ge 1$. 

\bigskip\noindent
{\bf D\"orfler marking and best $N$-term approximation.}
We recall the marking introduced by D\"orfler \cite{dorfler:96}, which is the 
only one for which there exist provable convergence rates. Given a 
parameter $\theta\in (0,1)$, and a current set of Fourier frequencies
or indices 
$\Lambda$, say the first $N$ ones according to the labeling of $\bv$, 
we choose the next set
$\partial\Lambda$ as the {\it minimal} set for which
\begin{equation}\label{dorfler}
\|P_{\partial\Lambda} \br\| \ge \theta \|\br\|,
\end{equation}
where $\br := \bv - P_\Lambda \bv$ is the {\it residual}
and $P_\Lambda$ is the orthogonal projection in the $\ell^2$-norm
$\|\cdot\|$ onto
$\Lambda$. Note that, if $\br_*:= \br - P_{\partial\Lambda}\br$ and
$\Lambda_* := \Lambda\cup\partial\Lambda$, then
\eqref{dorfler} can be equivalently written as
\begin{equation}\label{dorfler-2}
\|\br_*\| = \|\br - P_{\partial\Lambda}\br\| \le \sqrt{1-\theta^2} \|\br\|,
\end{equation}
and that $\br = \bv|_{\Lambda^c}$ 
where $\Lambda^c := \mathbb N \backslash \Lambda$ is the complement of 
$\Lambda$ and likewise for $\br_*$. This is the
simplest possible scenario because the information built in $\br$ is
exactly that of $\bv$. Moreover, $\bv-\br=\{v_n^*\}_{n=1}^N$ is the best
$N$-term approximation of $\bv$ in the $\ell^2$-norm and the corresponding
error $E_N(v)$ is given by
\begin{equation}\label{N-term}
E_N(v) = \Big( \sum_{n>N} |v_n^*|^2 \Big)^{-\frac12} = \|\br\|.
\end{equation}

\noindent
{\bf Algebraic vs exponential decay.}
Suppose now that $\bv$ has the precise {\it algebraic} decay\footnote{Throughout the paper, 
$A \lsim B$ means $A \le c \, B$ for some constant $c>0$ 
independent of the relevant parameters in the inequality; $A \simeq B$ means 
$B \lsim A \lsim B$.} 
\begin{equation}\label{algebraic}
|v_n^*| \simeq n^{-\frac{1}{\tau}} \quad\forall\, n\ge1.
\end{equation}
with 
\begin{equation}\label{tau}
\frac{1}{\tau} = \frac{s}{d} + \frac{1}{2}
\end{equation}
and $s>0$. We denote by $\|\bv\|_{\ell^s_B}$ the smallest constant in 
the upper bound in \eqref{algebraic}.
We thus have
\[
E_N(v)^2 \simeq \|\bv\|_{\ell^\tau_w}^2 \sum_{n>N} n^{-\frac{2}{\tau}}
=  \|\bv\|_{\ell^s_B}^2 \sum_{n>N} n^{-\frac{2s}{d}-1}
\simeq  \|\bv\|_{\ell^s_B}^2 N^{-\frac{2s}{d}}.
\]
This decay is related to certain {\it Besov} regularity of $v$ \cite{DeVore-Temlyakov:1995}.
Note that the effect of D\"orfler marking \eqref{dorfler-2} is to
reduce the residual from $\br$ to $\br_*$
by a factor $\alpha = \sqrt{1-\theta^2}$, or equivalently
\[
E_{N_*}(v) \le  \alpha E_N(v),
\]
with $N_* = |\Lambda_*|$. Since the set $\Lambda_*$ is minimal, we
deduce that $E_{N_*-1}(v) > \alpha E_N(v)$, whence
\begin{equation}\label{bulk-1}
\frac{N_*}{N} \simeq \alpha^{-\frac{d}{s}} 
\quad\Rightarrow\quad
N_* - N \simeq \alpha^{-\frac{d}{s}} N
\end{equation}
for $\alpha$ small enough. This means that the number of degrees of freedom to be added is
proportional to the current number. This simplifies considerably the complexity
analysis since every step adds as many degrees of freedom as we
have already accumulated.

The exponential case is quite different. Suppose that $\bv$
has a {\it genuinely exponential} decay
\begin{equation}\label{gen-exp}
|v_n^*| \simeq  e^{-\eta n}
\quad\forall\, n\ge 1,
\end{equation}
corresponding to analytic functions
\cite{Foias-Temam:1989}, and let $\|\bv\|_{\ell^{\eta}_G}$ be the
smallest constant appearing in the upper bound in
  \eqref{gen-exp}. These definitions are slight simplifications of the 
actual ones in Sect. \ref{sec:nlg} but enough to give insight on the main
issues at stake.
 We thus have
\begin{equation*}
E_N(v)^2 \simeq  \|\bv\|_{\ell^{\eta}_G}^2 \sum_{n>N} 
e^{-2\eta n} \simeq \|\bv\|_{\ell^{\eta}_G}^2 e^{-2\eta N};
\end{equation*}
this and similar decays are related to {\it Gevrey} classes of 
$C^\infty$ functions \cite{Foias-Temam:1989}. In contrast to
\eqref{bulk-1}, D\"orfler marking now yields\footnote{Throughout 
the paper, $A \sim B$ means $A=B+c$ for some quantity $c \simeq 1$.}
\begin{equation}\label{bulk-2}
N_* - N \sim \frac{1}{\eta} \log \frac{1}{\alpha}.
\end{equation}
This shows that the number of additional degrees of freedom per step
is fixed and independent of $N$, which makes their counting 
as well as their implementation a very delicate operation.

\bigskip\noindent
{\bf Plateaux.} We now consider a situation opposite to the ideal
decay examined above.
Suppose that the first $K>1$ Fourier coefficients of $v$ are constant and either
\begin{equation}\label{plateaux}
|v_n^*| = \|\bv\|_{\ell^s_B} n^{-\frac{1}{\tau}}\quad\text{\rm or}\quad
|v_n^*| = \|\bv\|_{\ell^{\eta}_G} e^{-\eta n}
\quad\forall\, n\ge K,
\end{equation}
for each approximation class. A simple calculation reveals
that either
\begin{equation}
\|\bv\| \simeq  \|\bv\|_{\ell^s_B} K^{-s/d}
\quad\text{\rm or}\quad
\|\bv\| \simeq \|\bv\|_{\ell^{\eta}_G} e^{-\eta K}.
\end{equation}
Repeating the argument leading to \eqref{bulk-1} and \eqref{bulk-2}
with $N=1$, we infer that either
\begin{equation}\label{bulk-3}
N_* \simeq K \alpha^{-\frac{d}{s}}
\quad\text{\rm or }\quad
N_* \sim K + \frac{1}{\eta}\log\frac{1}{\alpha}.
\end{equation}
For $K\gg1$ this is a much larger number than the optimal values 
\eqref{bulk-1} and \eqref{bulk-2}, and illustrates the fact that the
D\"orfler condition \eqref{dorfler} adds many more frequencies in the
presence of plateaux. We note that $K$ is a multiplicative constant
in the left of \eqref{bulk-3} and additive in the right of \eqref{bulk-3}.

\bigskip\noindent
{\bf Sparsity of the residual.}
In practice we do not have access to the Fourier decomposition of $v$ 
but rather of the residual $r(v)=f-Lv$, where $f$ is the forcing
function and $L$ the differential operator. Only an operator $L$ with constant
coefficients leads to a spectral representation with diagonal matrix $\bA$,
in which case the components of the residual $\br = \mathbf{f} - \bA\bv$ are
directly those of $\mathbf{f}$ and $\bv$. In general $\bA$ decays
away from the main diagonal with a law that depends on the regularity
of the coefficients of $L$; we will examine in Sect. \ref{S:properties-A}
either algebraic or exponential decay. In this much more intricate and
interesting endeavor, studied in this paper,
the components of $\bv$ interact with entries of
$\bA$ to give rise to $\br$.
The question whether $Lv$ belongs to the same approximation class of
$v$ thus becomes relevant because adaptivity decisions are made
  with $r(v)$, and thereby on the range of $L$ rather than its domain.

We now provide insight on the key issues at stake via a couple of
heuristic examples; we discuss this fully in
Sect. \ref{S:algebraic-case} and Sect. \ref{subsec:spars-res-exp}. We
start with the exponential case: let $\bv:=\{v_k\}_{k\in\mathbb{Z}}$ be defined
by
\[
v_k = e^{-\eta n} \quad\textrm{if}\quad k = 2p(n-1),
\qquad
v_k = 0 \quad\textrm{otherwise},
\]
for $p\ge2$ a given integer and $n\ge1$. This sequence exhibits gaps
of size $2p$ between consecutive nonzero entries for $k\ge0$. Its non-decreasing
rearrangement $\bv^*=\{v_n^*\}_{n=1}^\infty$ is thus given by
\[
v_n^* = e^{-\eta n} \quad n\ge 1,
\]
whence $\bv\in\ell^\eta_G$ with $\|\bv\|_{\ell^\eta_G}=1$. Let
$\bA:=(a_{ij})_{i,j=1}^\infty$ be the Toeplitz bi-infinite matrix given by
\[
a_{ij} = 1\quad\textrm{if } |i-j| \le q,
\qquad
a_{ij} = 0\quad\textrm{otherwise},
\]
with $1\le q<p$. This matrix $\bA$ has $2q+1$ main nontrivial diagonals and is
both of exponential and algebraic class according to the Definition
\ref{def:class.matrix} below. The product $\bA\bv$ is much less sparse
than $\bv$ but, because $q<p$, consecutive frequencies of $\bv$ do not
interact with each other: the $i$-th component reads
\[
(\bA\bv)_i = e^{-\eta n}\quad\textrm{if}\quad \big| i-2p(n-1)\big| \le q
\quad\textrm{for some}\quad n\ge 1,
\]
or $(\bA\bv)_i=0$ otherwise. The non-decreasing rearrangement
$(\bA\bv)^*$ of $\bA\bv$ becomes
\[
(\bA\bv)^*_m = e^{-\eta n} \quad\textrm{if}\quad
(2q+1)(n-1) + 1 \le m \le (2q+1) n.
\]
Consequently, writing $(\bA\bv)^*_m = e^{-\eta \frac{n}{m} m}$ and
observing that
\[
\frac{n}{m} \ge \frac{n}{(2q+1)n} = \frac{1}{2q+1}
\]
and the equality is attained for $m=(2q+1)n$, we deduce 
\[
\bA\bv \in \ell^{\bar\eta}_G \quad\textrm{with}\quad
\|\bA\bv\|_{\ell^{\bar\eta}_G} =1
\quad \bar\eta = \frac{\eta}{2q+1}.
\]
We thus conclude that the action of $\bA$ may
shift the exponential class, from the one characterized by the parameter $\eta$ for
$\bv$ to the one characterized by $\bar\eta<\eta$ for $\bA\bv$.
This uncovers the crucial feature that the image $\bA\bv$ of $\bv$ may be
substantially less sparse than $\bv$ itself.
In Sect. \ref{subsec:spars-res-exp} we present a rigorous construction
with $a_{ij}$ decreasing exponentially from the main diagonal and another,
rather sophisticated, construction that illustrates the fact that the exponent
$\tau = 1$  in the bound $|v_n^*| \lsim  e^{-\eta n} = e^{-\eta n^\tau}$ for $\bv$
may deteriorate to some $\bar\tau < 1$ in the corresponding bound for $\bA\bv$.

It is remarkable that a similar construction for the algebraic decay 
would not lead to a change of algebraic class. In fact, let
$\bv=\{v_k\}_{k\in\mathbb{Z}}$ be given by
\[
v_k = \frac{1}{n} \quad\textrm{if}\quad
k = 2p(n-1) \quad\textrm{for some}\quad
n\ge1,
\]
and $v_k=0$ otherwise. The non-decreasing rearrangement
$\bv^*=\{v_n^*\}_{n=1}^\infty$ of $\bv$ satisfies $v^*_n=\frac{1}{n}$
whence
\[
\bv\in \ell^s_B \quad\textrm{with}\quad
s = \frac{d}{2} \quad \|\bv\|_{\ell^s_B}=1.
\]
On the other hand, the $i$-th component of $\bA\bv$ reads
\[
(\bA\bv)_i = \frac{1}{n}\quad\textrm{if}\quad
\big|i-2p(n-1)\big| \le q 
\quad\textrm{for some}\quad
n\ge 1,
\]
or $(\bA\bv)_i=0$ otherwise. The non-decreasing rearrangement of
$(\bA\bv)^*$ in turn satisfies
\[
(\bA\bv)^*_m = \frac{1}{n} \quad\textrm{if}\quad
(2q+1)(n-1) + 1 \le m \le (2q+1) n,
\]
whence writing $(\bA\bv)^*_m = \frac{m}{n} \frac{1}{m}$ and arguing as
before we infer that
\[
\bA\bv\in \ell^s_B \quad\textrm{with}\quad
\|\bA\bv\|_{\ell^s_B} = 2q+1.
\]
Since  $\|\bA\bv\|_{\ell^s_B} > \|\bv\|_{\ell^s_B}$ we realize that 
$\bA\bv$ is less sparse than $\bv$ but, in contrast to the exponential
case, they belong to the same algebraic class $\ell^s_B$.
Moreover, we will
prove later in Sect. \ref{S:algebraic-case} that $\bA$ preserves the
class $\ell^s_B$ provided entries of $\bA$ possess a suitable algebraic decay
away from the main diagonal.

Since D\"orfler marking is applied to the residual $\br$, it is its
sparsity class that determines the degrees of freedom
$|\partial\Lambda|$ to be added. The same
argument leading to either \eqref{bulk-1} or \eqref{bulk-2} gives
\[
|\partial\Lambda| \le
\Big(\frac{\|\br\|_{\ell^s_B}}{\alpha\|\br\|}\Big)^{\frac{d}{s}} + 1
\qquad\text{\rm or }\qquad
|\partial\Lambda| \le \frac{1}{\eta}
\log 
\frac{\|\br\|_{\ell^{\eta}_G}}{\alpha\|\br\|} +1,
\]
for each class.
We thus see that the ratios $\|\br\|_{\ell^s_B}/\|\br\|$ and
$\|\br\|_{\ell^{\eta}_G}/\|\br\|$
control the behavior of the adaptive procedure. This has already been
observed and exploited by Cohen et al \cite{CDDV:1998} in the context
of wavelet methods for the class $\ell^s_B$. 
Our estimates, discussed in Sect. \ref{sec:spars-res},
are valid for both classes and use specific decay properties of
the entries of $\bA$. 

\bigskip\noindent
{\bf Coarsening.} Ever since its inception by Cohen et al
\cite{CDDV:1998} and Binev et al \cite{BDD:04}, 
this has been a controvertial issue for elliptic
PDE. It was originally due to the lack of control on the ratio
$\|\br\|_{\ell^s_B}/\|\br\|$ for large $s$ \cite{CDDV:1998}. It was removed by
Stevenson et al \cite{Gantumur-Stevenson:2007,Stevenson:2007}
for the algebraic class $\ell^s_B$
via a clever argument that exploits the minimality of D\"orfler
marking. This implicitly implies that the approximation classes for
both $v$ and $Lv$ coincide, which we prove explicitly in 
Sect. \ref{S:algebraic-case} for the algebraic case. This is not true though
for the exponential case and is discussed in Sect. \ref{subsec:spars-res-exp}.
For the latter, we need to resort to {\it coarsening} to keep the
cardinality of ADFOUR quasi-optimal. To this end, we 
{construct an insightful example 
in Sect. \ref{S:coarsening} and prove} a rather simple but
sharp coarsening estimate which improves upon \cite{CDDV:1998}.

\bigskip\noindent
{\bf Contraction constant.} It is well known that the contraction
constant $\rho(\theta) = \sqrt{1 - \frac{\alpha_*}{\alpha^*}\theta^2}$
cannot be arbitrarily close to $1$ for estimators whose upper and
lower constants, $\alpha^*\ge\alpha_*$, do not coincide. This is,
however, at odds with the philosophy of spectral methods which are
expected to converge superlinearly (typically exponentially). Assuming
that the decay properties of $\bA$ are known, we can enrich D\"orfler
marking in such a way that the contraction factor becomes
\[
\bar\rho(\theta) = \Big(\frac{\alpha^*}{\alpha_*}\Big)^{\frac12}
\sqrt{1-\theta^2}.
\]
This leads to $\bar\rho(\theta)$
as close to $1$ as desired and to {\it aggressive} versions
of ADFOUR discussed in Sect. \ref{sec:plain-adapt-alg}.

\medskip
This paper can be viewed as a first step towards understanding
adaptivity for the $hp$-FEM. However, the results we present are of 
intrinsic interest and of value for periodic problems with high degree
of regularity and rather complex structure. One such problem is
turbulence in a periodic box. Our techniques exploit periodicity and
orthogonality of the complex exponentials, but many of our assertions
and conclusions extend to the non-periodic case for which the natural
basis functions are Legendre
polynomials; this is the case of the $p$-FEM. In any event, the study of
adaptive Fourier-Galerkin methods seems to be a new paradigm in
adaptivity, with many intriguing questions and surprises, some
discussed in this paper. In contrast to the $h$-FEM,
they exhibit unlimited approximation power which is
only restricted by solution and data regularity.

We organize the paper as follows. In Sect. \ref{sec:gen} we introduce the
Fourier-Galerkin method, present a posteriori error estimators,
and discuss properties of the underlying matrix $\bA$ for both
algebraic and exponential
approximation classes. In Sect. \ref{sec:plain-adapt-alg} we deal with four algorithms,
two for each class, and prove their contraction properties. We devote
Sect. \ref{sec:nl} to nonlinear approximation theory with an emphasis on 
the exponential class. In Sect. \ref{sec:spars-res} we turn to the study
of the sparsity classes for the residual $\br$ along the lines
outlined above. We examine the role of coarsening and prove a sharp
coarsening estimate in Sect. \ref{S:coarsening}. We conclude with
optimality properties of ADFOUR for the algebraic class in Sect.
\ref{sec:complexity} and for the exponential class in Sect. \ref{sec:adfour-coarse}. 

\section{Fourier-Galerkin approximation}\label{sec:gen}
\subsection{Fourier basis and norm representation}

For $d \geq 1$, we consider $\Omega=(0,2\pi)^d$, and the trigonometric basis 
$$
\phi_k(x)=\frac1{(2\pi)^{d/2}} \, {\rm e}^{i k \cdot x} \;, 
\qquad k \in \mathbb{Z}^d \;, \quad x \in \mathbb{R}^d \;,
$$
which is orthonormal in $L^2(\Omega)$; let 
$$
v = \sum_k \hat{v}_k \phi_k \;, \quad \hat{v}_k=(v,\phi_k) \;, 
\qquad \text{with } \ \Vert v \Vert_{L^2(\Omega)}^2= \sum_k |\hat{v}_k|^2 \;,
$$
be the expansion of any $v \in L^2(\Omega)$ and the representation of its norm via the Parseval 
identity. Let $H^1_p(\Omega)=\{v \in H^1(\Omega) \, : \, v(x+2\pi e_j)=v(x) \ 1 \leq j \leq d\}$, 
and let $H^{-1}_p(\Omega)$ be its dual. Since the trigonometric basis is orthogonal in $H^1_p(\Omega)$
as well, one has for any $v \in H^1_p(\Omega)$
\begin{equation}\label{eq:four01}
\Vert v \Vert_{H^1_p(\Omega)}^2 = \sum_k (1+|k|^2)|\hat{v}_k|^2 = \sum_k |\hat{V}_k|^2 \;, 
\qquad (\text{setting }\hat{V}_k := \sqrt{(1+|k|^2)}\hat{v}_k)\;;
\end{equation}
here and in the sequel, $|k|$ denotes the Euclidean norm of the multi-index $k$. 
On the other hand, if $f \in H^{-1}_p(\Omega)$, we set 
$$
\hat{f}_k=\langle f, \phi_k \rangle \;, \qquad \text{so that } \ 
\langle f, v \rangle = \sum_k \hat{f}_k \hat{v}_k \quad \forall v \in H^1_p(\Omega) \;;
$$
the norm representation is
\begin{equation}\label{eq:four02}
\Vert f \Vert_{H^{-1}_p(\Omega)}^2 = \sum_k \frac1{(1+|k|^2)}|\hat{f}_k|^2 = \sum_k |\hat{F}_k|^2 \;, 
\qquad (\text{setting }\hat{F}_k := \frac1{\sqrt{(1+|k|^2)}}\hat{f}_k)\;.
\end{equation}
Throughout the paper, we will use the notation $\Vert \ . \ \Vert$ to
indicate both the $H^1_p(\Omega)$-norm 
of a function $v$, or the $H^{-1}_p(\Omega)$-norm of a linear
form $f$; the specific meaning will be clear from the context.

Given any finite index set $\Lambda \subset {\mathbb{Z}}^d$, we define
the subspace of $V:=H^1_p(\Omega)$
$$
V_{\Lambda} := {\rm span}\,\{\phi_k\, | \, k \in \Lambda \}\;;
$$
we set $|\Lambda|= \rm{card}\, \Lambda$, so that $\rm{dim}\, V_{\Lambda}=|\Lambda|$. If $g$ admits an
expansion $g = \sum_k \hat{g}_k \phi_k $ (converging in an appropriate norm), then we define its 
projection $P_\Lambda g$ upon $V_\Lambda$ by setting
$$
P_\Lambda g = \sum_{k \in \Lambda} \hat{g}_k \phi_k \;.
$$
\subsection{Galerkin discretization and residual}
We now consider the elliptic problem
\begin{equation}\label{eq:four03}
\begin{cases} 
Lu=-\nabla \cdot (\nu \nabla u)+ \sigma u = f & \text{in } \Omega \;, \\
u \ \ 2\pi\text{-periodic in each direction} \;, 
\end{cases} 
\end{equation}
where $\nu$ and $\sigma$ are sufficiently smooth real coefficients satisfying 
$0 < \nu_* \leq \nu(x) \leq \nu^* < \infty$ and $0 < \sigma_* \leq \sigma(x) \leq \sigma^* < \infty$
in $\Omega$; let us set
$$
\alpha_* = \min(\nu_*, \sigma_*) \qquad \text{and} \qquad \alpha^* = \max(\nu^*, \sigma^*) \;.
$$
We formulate this problem variationally as
\begin{equation}\label{eq:four.1}
u \in H^1_p(\Omega) \ \ : \quad a(u,v)= \langle f,v \rangle \qquad \forall v \in  H^1_p(\Omega) \;,
\end{equation}
where $a(u,v)=\int_\Omega \nu \nabla u \cdot \nabla \bar{v} + \int_\Omega \sigma u \bar{v}$ (bar 
indicating as usual complex conjugate). We denote by 
$\tvert v \tvert = \sqrt{a(v,v)}$
the energy norm of any $v \in H^1_p(\Omega)$, which satisfies 
\begin{equation}\label{eq:four.1bis}
\sqrt{\alpha_*}  \Vert v \Vert  \leq \tvert v \tvert \leq 
\sqrt{\alpha^*}  \Vert v \Vert \;.
\end{equation}

Given any finite set $\Lambda \subset \mathbb{Z}^d$, 
the Galerkin approximation is defined as
\begin{equation}\label{eq:four.2}
u_\Lambda \in V_\Lambda \ \ : \quad a(u_\Lambda,v_\Lambda)= 
\langle f,v_\Lambda \rangle \qquad \forall v_\Lambda \in V_\Lambda \;.
\end{equation}
For any $w \in V_\Lambda$, we define the residual
$$
r(w)=f-Lw = \sum_k \hat{r}_k(w) \phi_k \;, \qquad \text{where} \qquad 
\hat{r}_k(w) = \langle f - Lw, \phi_k \rangle = \langle f,\phi_k \rangle -a(w,\phi_k) \;.
$$
Then, the previous definition of $u_\Lambda$ is equivalent to the condition
\begin{equation}\label{eq:four.2.1ter}
P_\Lambda r(u_\Lambda) = 0 \;, \qquad \text{i.e., } 
\quad \hat{r}_k(u_\Lambda)=0 \qquad \forall k \in \Lambda \;.
\end{equation}
On the other hand, by the continuity and coercivity of the bilinear
form $a$, one has 
\begin{equation}\label{eq:four.2.1}
\frac1{\alpha^*} \Vert r(u_\Lambda) \Vert \leq
\Vert u - u_\Lambda \Vert \leq 
\frac1{\alpha_*} \Vert r(u_\Lambda) \Vert \;,
\end{equation}
or, equivalently,
\begin{equation}\label{eq:four.2.1bis}
\frac1{\sqrt{\alpha^*}} \Vert r(u_\Lambda) \Vert \leq
\tvert u - u_\Lambda \tvert \leq 
\frac1{\sqrt{\alpha_*}} \Vert r(u_\Lambda) \Vert \;.
\end{equation}

\subsection{Algebraic representations}\label{sec:algebraic_repres}
Let us identify the solution $u = \sum_k \hat{u}_k \phi_k$ of Problem (\ref{eq:four.1})
with the vector $\mathbf{u}=(\hat{U}_k)=(c_k \hat{u}_k) \in \mathbb{C}^{\mathbb{Z}^d}$ 
of its $H^1_p$-normalized Fourier coefficients, where we set for convenience 
$c_k=\sqrt{1+|k|^2}$. Similarly, let us identify
the right-hand side $f$ with the vector $\mathbf{f}=(\hat{F}_\ell)=(c_\ell^{-1}\hat{f}_\ell)
\in \mathbb{C}^{\mathbb{Z}^d}$ of its $H^{-1}_p$-normalized Fourier coefficients.
Finally, let us introduce the bi-infinite, Hermitian and positive-definite matrix 
\begin{equation}\label{eq:four100}
\mathbf{A}=(a_{\ell,k}) \qquad \text{with} \qquad 
a_{\ell,k}=\frac1{c_\ell c_k} a(\phi_k,\phi_\ell) \;.
\end{equation}
Then, Problem (\ref{eq:four.1}) can be equivalently written as
\begin{equation}\label{eq:four110}
\mathbf{A} \mathbf{u} = \mathbf{f} \;. 
\end{equation}
We observe that the orthogonality properties of the trigonometric basis implies that
the matrix $\mathbf{A}$ is diagonal if and only if the coefficients $\nu$ and $\sigma$
are constant in $\Omega$.

Next, consider the Galerkin problem (\ref{eq:four.2}) and let 
$\mathbf{u}_\Lambda \in \mathbb{C}^{|\Lambda|}$ be the vector collecting the coefficients
of $u_\Lambda$ indexed in $\Lambda$; 
let $\mathbf{f}_\Lambda \in \mathbb{C}^{|\Lambda|}$ be the analogous restriction for
the vector of the coefficients of $f$. Finally, denote by $\mathbf{R}_\Lambda$ the
matrix that restricts a bi-infinite vector to the portion indexed in $\Lambda$, so that
$\mathbf{E}_\Lambda=\mathbf{R}_\Lambda^H$ is the corresponding extension matrix.
Then, setting
\begin{equation}\label{eq:four120}
\mathbf{A}_\Lambda = \mathbf{R}_\Lambda \mathbf{A} \mathbf{R}_\Lambda^H \;,
\end{equation}
 Problem (\ref{eq:four.2}) can be equivalently written as
\begin{equation}\label{eq:four130}
\mathbf{A}_\Lambda \mathbf{u}_\Lambda = \mathbf{f}_\Lambda \;. 
\end{equation}

\subsection{Properties of the stiffness matrix}\label{S:properties-A}
It is  useful to express the elements of $\mathbf{A}$ in terms
of the Fourier coefficients of the operator coefficients $\nu$ and $\sigma$. Precisely,
writing 
$\nu=\sum_k \hat{\nu}_k \phi_k$ and  $\sigma=\sum_k \hat{\sigma}_k \phi_k$
and using the orthogonality of the Fourier basis, one easily gets
\begin{equation}\label{eq:four140}
a_{\ell,k}= \frac1{(2\pi)^{d/2}}\left( \frac{\ell \cdot k}{c_\ell c_k} \hat{\nu}_{\ell-k}
+ \frac1{c_\ell c_k} \hat{\sigma}_{\ell-k} \right) \;.
\end{equation} 
Note that the diagonal elements are uniformly bounded from below,
\begin{equation}\label{eq:four150}
a_{\ell,\ell} \geq \frac1{(2\pi)^{d/2}} \min({\hat{\nu}_0}, \hat{\sigma}_0) > 0 \;, 
\qquad \ell \in \mathbb{Z}^d \;,
\end{equation}
whereas all elements are bounded in modulus by the elements of a 
{{\it Toeplitz}} matrix,
\begin{equation}\label{eq:four160}
|a_{\ell,k}| \leq \frac1{(2\pi)^{d/2}} 
{\left( |\hat{\nu}_{\ell-k}|+ |\hat{\sigma}_{\ell-k}| \right)} \;,
\qquad \ell, k \in \mathbb{Z}^d  \;,
\end{equation}
which decay as $|\ell - k| \to \infty$ at a rate dictated by the smoothness of the operator coefficients.
Indeed, if $\nu$ and $\sigma$ are sufficiently smooth, their 
Fourier coefficients decay at a suitable rate and this property is inherited 
by the off-diagonal elements of the matrix ${\bf A}$, via 
\eqref{eq:four160}. To be precise, if the coefficients $\nu$ and $\sigma$ have 
a finite order of regularity, then the rate of decay of their Fourier coefficients 
is algebraic, i.e. 
\begin{equation}\label{eq:ass-coeff-alg}
|\hat{\nu}_k|,  |\hat{\sigma}_k| \lesssim  (1+\vert k \vert )^{-\eta} 
 \qquad \forall k \in \mathbb{Z}^d \;,
\end{equation}
for some $\eta>0$. On the other hand, if the operator coefficients are real analytic 
in a neighborhood of $\Omega$, then the rate of decay of their Fourier coefficients is exponential, i.e. 
\begin{equation}\label{eq:ass-coeff-exp}
|\hat{\nu}_k|,  |\hat{\sigma}_k| \lesssim  e^{-\eta\vert k \vert}
 \qquad \forall k \in \mathbb{Z}^d \;.
\end{equation}
Correspondingly, the matrix ${\bf A}$ belongs to one of the following
classes.

\begin{definition}[{regularity classes for $\bA$}]\label{def:class.matrix}
{A matrix ${\bf A}$ is said to belong to}
\begin{enumerate}[$\bullet$]
\item the {algebraic} class 
${\mathcal D}_a(\eta_L)$ if there exists a constant $c_L>0$ 
such that its elements satisfy
\begin{equation}
| a_{\ell,k} | \leq  c_L (1+ \vert \ell - k \vert )^{-\eta_L}\;
\qquad \ell, k \in \mathbb{Z}^d  \; ;
\end{equation}
\item the {exponential} class 
${\mathcal D}_e(\eta_L)$ if there exists a constant $c_L>0$ 
such that its elements satisfy
\begin{equation}\label{eq:four170}
| a_{\ell,k} | \leq  c_L e^{-\eta_L\vert \ell - k \vert }\;  \qquad \ell, k \in \mathbb{Z}^d  \;.
\end{equation}
\end{enumerate}
\end{definition}
The following properties hold.
\begin{property}[{continuity of $\bA$}]\label{prop:bounded}
If {either} ${\bf A}\in{\mathcal D}_a(\eta_L)$, with $\eta_L>d$, or 
${\bf A}\in{\mathcal D}_e(\eta_L)$, then ${\bf A}$ defines a bounded operator
on {$\ell^2(\mathbb{Z}^d)$}.
\end{property}
\begin{proof}
See e.g. \cite{Jaffard:1990, Dahlke-Fornasier-Groechenig:2010}. 
\end{proof}

\begin{property}[{inverse of $\bA$: algebraic case}] \label{prop:inverse.alg}
If ${\bf A}\in{\mathcal D}_a(\eta_L)$, with $\eta_L>d$ and ${\bf A}$ is invertible in 
$\ell^2(\mathbb{Z}^d)$, then ${\bf A}^{-1}\in{\mathcal D}_a(\eta_L)$. 
\end{property}
\begin{proof}
See e.g. \cite{Jaffard:1990}.
 \end{proof}

\begin{property}[{inverse of $\bA$: exponential case}] \label{prop:inverse.matrix-estimate}
If $\mathbf{A} \in {\mathcal D}_e(\eta_L)$ and there exists a constant $c_L$
satisfying \eqref{eq:four170} such that 
\begin{equation}\label{restriction-cL}
c_L < \frac12({\rm e}^{\eta_L} -1) \min_\ell a_{\ell,\ell}\;,
\end{equation}
then ${\bf A}$ is invertible in  $\ell^2(\mathbb{Z}^d)$ and 
${\bf A}^{-1}\in{\mathcal D}_e(\bar{\eta}_L)$ where 
$\bar{\eta}_L \in (0,\eta_L]$ 
is such that $\bar{z}={\rm e}^{-\bar{\eta}_L}$ is the unique zero in the 
interval $(0,1)$ of the polynomial
$$
z^2- \frac{{\rm e}^{2\eta_L}+2c_L+1}{{\rm e}^{\eta_L}(c_L+1)}z+1 \;.
$$
\end{property}
\begin{proof}
{We follow the suggestion by Bini \cite{Bini:xx},
 and thus exploit the one-to-one correspondence between Toeplitz
 matrices and formal Laurent series (see e.g. \cite{Toeplitz:book}):}
$$ f(z) = \sum_{k=-\infty}^\infty a_k z^k \longleftrightarrow \bT_f=(t_{i,j}),\quad t_{i,j}=a_{i-j}.$$ 
We refer to the function $f(z)$ as to the symbol associated to the
Toeplitz matrix $\bT_f$. {We recall now a few relations between
$f(z)$ and $\bT_f$. If}
$f(z)$ is analytic on 
${\mathcal A}_\alpha=\{z \in \mathbb{C}: \e^{-\alpha}<\vert z \vert <  \e^\alpha\}$ with $\alpha>0$, then there holds $f(z)=\sum_{k=-\infty}^{+\infty} a_k z^k$, where the coefficients $a_k$ have 
exponential decay with rate $\e^{-\alpha}$ {in the sense that}
for every $0<\rho<\e^{-\alpha}$ there exists 
a constant $\gamma >0$ such that $\vert a_k \vert \leq \gamma
\rho^{\vert k \vert}$. As a consequence, the symbol $f(z)$ of the
Toeplitz matrix $\bT_f$ is analytic on $\mathcal{A}_{\alpha}$ for some
$\alpha>0$  if and only if the elements of $\bT_f$ decay exponentially
with rate $\e^{-\alpha}$. {Moreover, it is known} 
that if $f(z)$ is analytic on $\mathcal{A}_\alpha$ and it is
non-zero  on $\mathcal{A}_\beta\subset\mathcal{A}_\alpha$, then the function $g(z)=1/f(z)$
is well defined and analytic on $\mathcal{A}_\beta$, the matrix $\bT_g$ is the inverse of 
$\bT_f$ and the elements of $\bT_g$ decay exponentially with rate
$\e^{-\beta}$. 

{We next} introduce the analytic functions in
$\mathcal{A}_\alpha$
\[
 h(z)=\sum_{k=1}^{\infty} \e^{-\alpha k}(z^k+z^{-k})
{ ~= \frac{z}{\e^\alpha-z}+\frac{z^{-1}}{\e^{\alpha}-z^{-1}}}, \qquad
f_c(z)=1-ch(z),
\]
with $c>0$. {For $|z|=1$ we deduce  
$|h(z)| \le 2\sum_{k=1}^\infty e^{-\alpha k} = 2/(e^\alpha-1)$, whence
$c|h(z)| <1$ provided that $c<\frac{1}{2}(\e^\alpha-1)$; moreover
$\|\bT_h\|\le\|\bT_h\|_\infty=2/(e^\alpha-1)$,
which is indeed a particular instance of Schur Lemma for symmetric matrices.
For this range of $c$'s,}
$f_c(z)\not= 0$ for $\vert z \vert =1$ and for continuity there exists 
$\mathcal{A}_\beta\subset \mathcal{A}_\alpha$ on which $f_c(z)$ in
non-zero. This implies that $g_c(z):=1/f_c(z)$ is analytic on
$\mathcal{A}_\beta$ and the elements of the associated Toeplitz matrix
$\bT_{g_c}$ decay exponentially with rate $\e^{-\beta}$. 
{The singularities of $g_c$ correspond to zeros of $f_c$, which
are in turn the roots $\zeta_1,\zeta_2$ of the polynomial
\[
z^2-\frac{\e^{2\alpha}+2c+1}{\e^\alpha(c+1)}z+1.
\]
These roots are real provided $c<\frac{1}{2}(\e^\alpha-1)$, in which
case $e^{-\beta}=\zeta_1=\zeta_2^{-1}<1$.}

Let $\bA\in\mathcal{D}_e(\alpha)$, i.e. there exists a constant $c$ such that 
$|a_{\ell,k} | \leq  c e^{-\alpha \vert \ell - k \vert }$  for $\ell, k \in \mathbb{Z}^d$.
{By rescaling of the rows of $\bA$, 
it is not restrictive to assume that the diagonal elements} $\bA$ are
equal to $1$. Then, it is possible to write $\bA=\bI-\bS$ with $\vert
\bS\vert \leq c\bT_h$, the inequality {being meant 
element by element,  and $\|\bS\|<1$. Since
$g_c(z)=1/(1-ch(z))=\sum_{k=0}^{\infty} c^k h(z)^k$ 
is well defined and analytic on $\mathcal{A}_\beta\subset\mathcal{A}_\alpha$,
it follows that}
\[
\left\vert \sum_{k=0}^{\infty} \bS^k\right\vert \leq \sum_{k=0}^{\infty}\vert \bS\vert^k 
\leq \sum_{k=0}^{\infty} c^k \bT_h^k = \bT_{g_c}.
\]
Hence, the elements of the matrix $\bT_{g_c}$ decay exponentially with
rate $\e^{-\beta}$. {Property $\|\bS\|<1$ yields
$\bA^{-1}=(\bI-\bS)^{-1}=\sum_{k=0}^{\infty} \bS^k$ and 
$ \vert\bA^{-1}\vert \leq \bT_{g_c}$, whence
the coefficients of $\bA^{-1}$ being bounded by those of $\bT_{g_c}$
decay exponentially with rate
$\e^{-\beta}$, 
i.e. $\bA^{-1}\in\mathcal{D}_e(\beta)$ for some $\beta<\alpha$. This
gives \eqref{restriction-cL} once the row scaling of $\bA$ is taken into account.}
\end{proof}

{
\begin{example}[sharpness of \eqref{restriction-cL}]\label{sharp-cL}
\rm 
The following example illustrates that \eqref{restriction-cL} is
sharp. Let $\bA$ be
\[
a_{ij} = - 2^{-1-|i-j|} \quad i\ne j,
\qquad
a_{ii} = 1,
\]
which is singular because the sum of the coefficients in every row
vanishes. This $\bA$ corresponds to $e^{\eta_L}=2$, $c_L=\frac12$ and $\frac12
(e^{\eta_L}-1)=\frac12$, which violates \eqref{restriction-cL}.
\end{example}
}

For any integer $J \geq 0$, let $\mathbf{A}_J$ denote {the following
symmetric truncation of the matrix $\mathbf{A}$}
\begin{equation}\label{eq:trunc-matr}
(\mathbf{A}_J)_{\ell,k}=
\begin{cases}
a_{\ell,k} & \text{if } |\ell-k| \leq J \;, \\
0 & \text{elsewhere.}
\end{cases}
\end{equation}
Then, we have the following well-known results, whose proof is reported for completeness.
\begin{property}[truncation]\label{prop:matrix-estimate}
{The truncated matrix $\bA_J$ has a number of non-vanishing entries
bounded by $\omega_d J^d$, where $\omega_d$ is the measure of the 
Euclidean unit ball in $\mathbb{R}^d$. Moreover,
under the assumption of Property \ref{prop:bounded},} there exists a constant 
$C_{\mathbf{A}} $ such that 
\[
\Vert \mathbf{A}-{\mathbf{A}}_J \Vert \leq
\psi_{\mathbf{A}}(J,\eta):=C_{\mathbf{A}} 
\begin{cases}
{(J+1)}^{-(\eta_L-d)} & \text{if }  \mathbf{A} \in {\mathcal D}_a(\eta_L)~\text{(algebraic case)} \;, \\
{(J+1)^{d-1}}{\rm e}^{-\eta_L J}  & \text{if}~\mathbf{A} \in {\mathcal D}_e(\eta_L)~\text{(exponential case)}\ ,
\end{cases}
\]
{for all $J\ge0$.}
Consequently, under the assumptions of Property \ref{prop:inverse.alg} or \ref{prop:inverse.matrix-estimate},
one has
\begin{equation}\label{eq:trunc-invmatr-err}
\Vert \mathbf{A}^{-1}-(\mathbf{A}^{-1})_J \Vert \leq 
\psi_{\mathbf{A}^{-1}} (J,\bar{\eta}_L)
\end{equation}
where we {let $\bar{\eta}_L=\eta_L$ in the algebraic case and
$\bar\eta_L$ be defined in Property \ref{prop:inverse.matrix-estimate}
for the exponential case.}
\end{property}

\begin{proof} 
We use the {Schur Lemma for symmetric matrices},
$\Vert \mathbf{B} \Vert \leq \Vert \mathbf{B} \Vert_\infty
=\sup_\ell \sum_k |b_{\ell,k}|$ for
$\mathbf{B}=\mathbf{A}-\mathbf{A}_J$.
Thus, in the algebraic case
\begin{eqnarray*}
\sup_\ell \sum_{k:|\ell-k|>J} |a_{\ell,k}| &\leq& C_L \sup_\ell \sum_{k:|\ell-k|>J}\frac1{(1+|\ell-k|)^{\eta_L}} \\
&\lsim & \sup_\ell \sum_{q=J+1}^\infty \sum_{\ {k:|\ell-k|=q}}\frac1{(1+q)^{\eta_L}} \lsim
\sup_\ell \sum_{q=J+1}^\infty \frac{q^{d-1}}{(1+q)^{\eta_L}} \lsim {(J+1)}^{d-\eta_L} \;.
\end{eqnarray*}
A similar argument yields the result in the exponential case.
\end{proof}

\subsection{An equivalent formulation of the Galerkin problem}
For future reference, herafter we rewrite the Galerkin problem \eqref{eq:four130} in an equivalent (infinite-dimensional) way.
Let  $$\mathbf{P}_\Lambda: \ell^2(\mathbb{Z}^d) \to \ell^2(\mathbb{Z}^d)$$ be the projector operator defined as 
\[
(\mathbf{P}_\Lambda \mathbf{v})_\lambda=
\begin{cases}
v_\lambda & \text{\rm if } \lambda\in\Lambda \;, \\
0 & \text{\rm if } \lambda\notin\Lambda \;.
\end{cases}
\]
Note that $\mathbf{P}_\Lambda$ can be represented  as a diagonal bi-infinite matrix whose diagonal elements 
are $1$ for indexes belonging to $\Lambda$,  zero otherwise.  
Let us set $\mathbf{Q}_\Lambda=\mathbf{I}-\mathbf{P}_\Lambda$ and 
we introduce the bi-infinite matrix $\widehat{\mathbf{A}}_\Lambda:=
\mathbf{P}_\Lambda \mathbf{A} \mathbf{P}_\Lambda + \mathbf{Q}_\Lambda$ which 
is equal to $\mathbf{A}_\Lambda$ for indexes in $\Lambda$ and to the identity matrix, otherwise. 
The definitions of the projectors $\mathbf{P}_\Lambda$ and $\mathbf{Q}_\Lambda$ yield the following result. 
\begin{property}[{invertibility of $\widehat\bA$}]\label{prop:inf-matrix}
If $\mathbf{A}$ is invertible with {either} $\mathbf{A}\in\mathcal{D}_a(\eta_L)$ or 
 $\mathbf{A}\in\mathcal{D}_e(\eta_L)$, then the same holds for $\widehat{\mathbf{A}}_\Lambda$. 
\end{property}

\noindent Now, let us consider the following extended Galerkin problem: 
find $\hat{\mathbf{u}}\in\ell^2(\mathbb{Z}^d)$ such that  
\begin{equation}\label{eq:inf-pb-galerkin}
\widehat{\mathbf{A}}_\Lambda \hat{\mathbf{u}}
= \mathbf{P}_\Lambda \mathbf{f}\ .
\end{equation}
Let  ${\mathbf{E}}_\Lambda: \mathbb{C}^{\vert \Lambda\vert} \to \ell^2(\mathbb{Z}^d)$ be the extension operator 
defined in Sect. \ref{sec:algebraic_repres} 
and let $\mathbf{u}_\Lambda\in \mathbb{C}^{\vert \Lambda\vert}$ be the Galerkin solution to  \eqref{eq:four130};
then, it is easy to check that $\hat{\mathbf{u}}={\mathbf{E}}_\Lambda \mathbf{u}_\Lambda$. 

In the following, with an abuse of notation, the solution of \eqref{eq:inf-pb-galerkin} will be denoted by 
$\mathbf{u}_\Lambda$. We will refer to it as to the (extended) Galerkin solution, meaning the infinite-dimensional 
representant of the finite-dimensional Galerkin solution. In case of possible confusion, we will make clear 
which version (infinite-dimensional or finite-dimensional) has to be considered.

\section{Adaptive algorithms with contraction properties}\label{sec:plain-adapt-alg}

Our first algorithm will be an {\sl ideal one}; it will serve as a reference to illustrate in the
simplest situation the contraction property which guarantees the convergence of the algorithm, 
and it will be subsequently modified to get more efficient versions. The ideal algorithm uses as
error estimator the ideal one, i.e., the norm of the residual in $H^{-1}_p(\Omega)$; 
we thus set, for any $v \in H^1_p(\Omega)$,
\begin{equation}\label{eq:four.2.2}
\eta^2(v)=\Vert r(v) \Vert^2 = \sum_{k \in \mathbb{Z}^d} |\hat{R}_k(v)|^2\;,
\end{equation}
so that (\ref{eq:four.2.1}) can be rephrased as
\begin{equation}\label{eq:four.2.3}
\frac1{\alpha^*} \eta(u_\Lambda) \leq
\Vert u - u_\Lambda \Vert \leq 
\frac1{\alpha_*} \eta(u_\Lambda) \; ;
\end{equation}
{recall that $\hat R_k(v) = (1 + |k|^2)^{-1/2}r_k(v)$ according
  to \eqref{eq:four02}.}
Obviously, this estimator is hardly computable in practice; in Sect. \ref{subsec:comput}
we will introduce a feasible version, but for the moment we go through the
ideal situation.
Given any subset $\Lambda \subseteq \mathbb{Z}^d$, we also define the quantity
$$
\eta^2(v;\Lambda) = \Vert P_\Lambda r(v) \Vert^2 
= \sum_{k \in \Lambda} |\hat{R}_k(v)|^2\;,
$$
so that $\eta(v)=\eta(v;\mathbb{Z}^d)$.
\subsection{ADFOUR: an ideal algorithm} \label{sec:defADFOUR}
We now introduce the following procedures, which will enter the definition of all our adaptive algorithms.

\begin{itemize}
\item $u_\Lambda := {\bf GAL}(\Lambda)$ \\
Given a finite subset $\Lambda \subset \mathbb{Z}^d$, the output
$u_\Lambda \in V_\Lambda$ is the solution of the Galerkin problem (\ref{eq:four.2}) relative to $\Lambda$.

\item $r := {\bf RES}(v_\Lambda)$ \\
Given a function $v_\Lambda \in V_\Lambda$ for some finite index set $\Lambda$, 
the output $r$ is the residual $r(v_\Lambda)=f-Lv_\Lambda$.

\item $\Lambda^* := \text{\bf D\"ORFLER}(r, \theta)$\\
Given $\theta \in (0,1)$ and an element $r \in H^{-1}_p(\Omega)$, 
the ouput $\Lambda^* \subset \mathbb{Z}^d$ is a finite set 
such that the inequality
\begin{equation}\label{eq:four.2.5.5}
\Vert P_{\Lambda^*} r \Vert \geq \theta  \Vert r \Vert 
\end{equation}
is satisfied.
\end{itemize}
Note that the latter inequality is equivalent to
\begin{equation}\label{eq:four.2.5.5bis}
\Vert r-P_{\Lambda^*} r \Vert \leq \sqrt{1-\theta^2}  \Vert r \Vert \;.
\end{equation}
If $r=r(u_\Lambda)$ is the residual of a Galerkin solution $u_\Lambda \in V_\Lambda$, 
then by (\ref{eq:four.2.1ter}) we can trivially assume 
that $\Lambda^*$ is contained in $\Lambda^c := \mathbb{Z}^d \setminus \Lambda$. 
For such a residual, inequality (\ref{eq:four.2.5.5}) can then be stated as
\begin{equation}\label{eq:four.2.5.5ter}
\eta(u_\Lambda;\Lambda^*) \geq \theta \eta(u_\Lambda) \;,
\end{equation}
a condition termed {\sl D\"orfler marking} in the finite element literature, or {\sl bulk chasing} in the
wavelet literature. {Writing} $\hat{R}_k = \hat{R}_k(u_{\Lambda})$,
the condition {\eqref{eq:four.2.5.5ter}} can be equivalently stated as
\begin{equation}\label{eq:four.2.4.bis}
\sum_{k \in \Lambda^*}  |\hat{R}_k|^2 
\geq \theta^2  \sum_{k \not \in \Lambda}  |\hat{R}_k|^2 \;. 
\end{equation}
Also note that a set $\Lambda^*$ of minimal cardinality can be immediately determined 
if the coefficients $\hat{R}_k$ are rearranged in non-increasing order of modulus; 
however, the subsequent convergence 
result does not require the property of minimal cardinality for the sets of active coefficients.

In the sequel, we will invariably make the following assumption:

\begin{assumption}[{D\"orfler marking}]\label{ass:minimality} 
The procedure {\bf D\"ORFLER} selects an index set $\Lambda^*$
of minimal cardinality among all those satisfying condition (\ref{eq:four.2.5.5}).
\end{assumption}

\medskip
Given two parameters $\theta \in (0,1)$ and $tol \in [0,1)$, 
we are ready to define  our ideal adaptive algorithm.

{\bf Algorithm ADFOUR}($\theta, \, tol$)
\begin{itemize}
\item[\ ] Set $r_0:=f$, $\Lambda_0:=\emptyset$, $n=-1$
\item[\ ] do
	\begin{itemize}
	\item[\ ] $n \leftarrow n+1$
	\item[\ ] $\partial\Lambda_{n}:= \text{\bf D\"ORFLER}(r_n, \theta)$
	\item[\ ] $\Lambda_{n+1}:=\Lambda_{n} \cup \partial\Lambda_{n}$
	\item[\ ] $u_{n+1}:= {\bf GAL}(\Lambda_{n+1})$
	\item[\ ] $r_{n+1}:= {\bf RES}(u_{n+1})$
	\end{itemize}
\item[\ ]  while $\Vert r_{n+1} \Vert > tol $
\end{itemize}

The following result states the
convergence of this algorithm, with a guaranteed error reduction rate.

\begin{theorem}[{convergence of {\bf ADFOUR}}]\label{teo:four1}
Let us set 
\begin{equation}\label{eq:def_rhotheta}
\rho=\rho(\theta)= \sqrt{1 - \frac{\alpha_*}{\alpha^*}\theta^2} \in (0,1) \;.
\end{equation}
Let $\{\Lambda_n, \, u_n \}_{n\geq 0}$ be the sequence generated by
the adaptive algorithm {\bf ADFOUR}.
Then, the following bound holds for any $n$:
$$
  \tvert u-u_{n+1} \tvert \leq \rho \tvert u-u_n \tvert \;.
$$ 
Thus, for any $tol>0$ the algorithm terminates in a finite number of iterations, whereas for $tol=0$
the sequence $u_n$ converges to $u$ in $H^1_p(\Omega)$ as $n \to \infty$.
\end{theorem}
\begin{proof}
For convenience, we use the notation $e_n:=\tvert u-u_n \tvert$ and $d_n:=\tvert u_{n+1} - u_n \tvert$.
As $V_{\Lambda_{n}} \subset V_{\Lambda_{n+1}}$, the following orthogonality property
holds
\begin{equation}\label{pytagora}
  e_{n+1}^2=e_n^2-d_n^2.
\end{equation}
On the other hand, for any $w \in H^{1}_p(\Omega)$, one has
{in light of \eqref{eq:four.1bis}}
\begin{eqnarray*}
\Vert Lw \Vert = \sup_{v \in H^{1}_p(\Omega)}
\frac{\langle Lw, v \rangle}{\Vert v \Vert} =
\sup_{v \in H^{1}_p(\Omega)}
\frac{a(w,v)}{\Vert v \Vert} 
\leq  \tvert w \tvert \sup_{v \in H^{1}_p(\Omega)}
\frac{\tvert v \tvert }{\Vert v \Vert} \leq \sqrt{\alpha^*} 
\tvert w \tvert \;.
\end{eqnarray*}
Thus, {using \eqref{eq:four.2.5.5},}
\begin{eqnarray*}
d_n^2 &\geq& \frac1{\alpha^*} \Vert L(u_{n+1}-u_n) \Vert^2 
=\frac1{\alpha^*} \Vert r_{n+1}-r_n \Vert^2 \\
&\geq& \frac1{\alpha^*} \Vert P_{\Lambda_{n+1}} (r_{n+1}-r_n) \Vert^2 =
\frac1{\alpha^*} \Vert P_{\Lambda_{n+1}} r_n \Vert^2
\geq  \frac{\theta^2}{\alpha^*} \Vert  r_n \Vert^2 \;.
\end{eqnarray*}
On the other hand, the {rightmost} inequality in (\ref{eq:four.2.1bis}) states that
$\Vert r_n \Vert^2 \geq \alpha_* e_n^2$,
whence the result. 
\end{proof}

\subsection{{\bf F-ADFOUR:} A feasible version of ADFOUR}\label{subsec:comput}
The error estimator $\eta(u_\Lambda)$ based on (\ref{eq:four.2.2}) is not computable in 
practice, since the residual $r(u_\Lambda)$ contains infinitely many coefficients.
We {thus} introduce a new estimator, defined from an approximation of such 
residual with finite Fourier expansion (i.e., a trigonometric polynomial). To this end,
let $\tilde{\nu}$, $\tilde{\sigma}$ and $\tilde{f}$  be suitable trigonometric polynomials,
which approximate $\nu$, $\sigma$ and $f$, respectively, to a given accuracy. Then, the
quantity
\begin{equation}\label{eq:four.2.5}
\tilde{r}(u_\Lambda)= \tilde{f} - \tilde{L}u_\Lambda =  \tilde{f} + \nabla \cdot (\tilde{\nu}
\nabla u_\Lambda) - \tilde{\sigma}  u_\Lambda
\end{equation}
belongs to $V_{\tilde{\Lambda}}$ for some finite subset $\tilde{\Lambda} \subset \mathbb{Z}^d$,
i.e., it has the finite (thus, computable) expansion
$$
\tilde{r}(u_\Lambda)= \sum_{k \in \tilde{\Lambda}} \hat{\tilde{r}}_k(u_\Lambda) \phi_k\;.
$$ 
The choice of the approximate coefficients has to be done in order to fulfil the following
condition: for a fixed parameter $\gamma\in (0,\theta)$, we require that
\begin{equation}\label{eq:four.2.6}  
\Vert r(u_\Lambda) - \tilde{r}(u_\Lambda)  \Vert \leq \gamma
\Vert \tilde{r}(u_\Lambda)  \Vert \;.
\end{equation}
Satisfying such a condition is possible, provided we have full access
to the data. Indeed, 
on the one hand, the left-hand side tends to $0$ as the approximation
of the coefficients 
gets better and better, since (we keep here the full norm indication
for a better clarity)
\begin{eqnarray*}
\Vert r(u_\Lambda) - \tilde{r}(u_\Lambda)  \Vert_{H^{-1}_p(\Omega)} &\leq&
\Vert f - \tilde{f} \Vert_{H^{-1}_p(\Omega)} + \Vert \nu - \tilde{\nu} \Vert_{L^\infty(\Omega)}
\Vert \nabla u_\Lambda \Vert_{L^2(\Omega)^d} + \Vert \sigma - \tilde{\sigma} \Vert_{L^\infty(\Omega)}
\Vert u_\Lambda \Vert_{L^2(\Omega)} \\
&\leq& 
\Vert f - \tilde{f} \Vert_{H^{-1}_p(\Omega)} +( \Vert \nu - \tilde{\nu} \Vert_{L^\infty(\Omega)}
+\Vert \sigma - \tilde{\sigma} \Vert_{L^\infty(\Omega)})
\frac1{\alpha_*}\Vert f  \Vert_{H^{-1}_p(\Omega)} \;,
\end{eqnarray*}
where we have used the bound on the solution of the Galerkin problem (\ref{eq:four.2})
 in terms of the data. On the other hand,
if $u_\Lambda \not = u$, then $ r(u_\Lambda) \not = 0$, {whence}
the right-hand side of 
(\ref{eq:four.2.6}) converges to a non-zero value {as
  $\tilde\Lambda$ increases}.

With this remark in mind, we define a new error estimator by setting
\begin{equation}\label{eq:four.2.7}
\tilde{\eta}^2(u_\Lambda)=\Vert \tilde{r}(u_\Lambda) \Vert^2 
= \sum_{k \in \tilde{\Lambda}} |\hat{\tilde{R}}_k(u_\Lambda)|^2\;, 
\end{equation}
which, {in view of \eqref{eq:four.2.6}, immediately yields}
\begin{equation}\label{eq:four.2.3bis}
\frac{1-\gamma}{\alpha^*} \tilde{\eta}(u_\Lambda) \leq
\Vert u - u_\Lambda \Vert \leq 
\frac{1+\gamma}{\alpha_*} \tilde{\eta}(u_\Lambda) \;.
\end{equation} 

\begin{lemma}[{feasible D\"orfler marking}]\label{lemma:four.1}
Let $\Lambda^*$ be any finite index set such that
$$
\tilde{\eta}(u_\Lambda;\Lambda^*) \geq \theta \tilde{\eta}(u_\Lambda) \;.
$$
Then, 
\begin{equation}\label{eq:four.2.300}
{\eta}(u_\Lambda;\Lambda^*) \geq \tilde{\theta} {\eta}(u_\Lambda)\;,
\qquad \text{with } \ \ \tilde{\theta} = \frac{\theta-\gamma}{1+\gamma} \in (0, \theta)\;.
\end{equation}
\end{lemma}
\begin{proof} One has
\begin{eqnarray*}
\Vert P_{\Lambda^*} {r}(u_\Lambda) \Vert &\geq&
\Vert P_{\Lambda^*} \tilde{r}(u_\Lambda) \Vert -
\Vert P_{\Lambda^*} \left( {r}(u_\Lambda) - \tilde{r}(u_\Lambda) \right) \Vert \\
&\geq& \theta \Vert \tilde{r}(u_\Lambda) \Vert - 
\Vert {r}(u_\Lambda) - \tilde{r}(u_\Lambda)  \Vert \\
&\geq& (\theta-\gamma) \Vert \tilde{r}(u_\Lambda) \Vert
\geq  \frac{\theta-\gamma}{1+\gamma} \Vert {r}(u_\Lambda) \Vert \; ,
\end{eqnarray*}
{which is the desired \eqref{eq:four.2.300}.}
\end{proof}

The previous result suggests introducing the following feasible variant of the procedure {\bf RES}:
\begin{itemize}
\item $r := \text{\bf F-RES}(v_\Lambda, \gamma)$ \\
Given $\gamma \in (0,\theta)$ and a function $v_\Lambda \in V_\Lambda$ 
for some finite index set $\Lambda$, the output $r$ is an approximate residual 
$\tilde{r}(v_\Lambda)=\tilde{f}+\nabla \cdot(\tilde{\nu}\nabla v_\Lambda)-\tilde{\sigma}v_\Lambda$,
defined on a finite set $\tilde{\Lambda}$ and satisfying
$$
\Vert r(v_\Lambda) - \tilde{r}(v_\Lambda)  \Vert \leq \gamma 
\Vert \tilde{r}(v_\Lambda)  \Vert \;.
$$
\end{itemize}

\begin{theorem}[{contraction property of {\bf F-AFOUR}}]\label{teo:four2}
Consider the feasible variant {\bf F-ADFOUR} of the adaptive algorithm {\bf ADFOUR}, 
where the step 
$r_{n+1}:= {\bf RES}(u_{n+1})$ is replaced by the step $r_{n+1}:= \text{\bf F-RES}(u_{n+1},\gamma)$
for some $\gamma \in (0,\theta)$. 
Then, the same conclusions of Theorem \ref{teo:four1} hold true for this variant, with the contraction
factor $\rho$ replaced by $\rho=\rho(\tilde{\theta})$, where $\tilde{\theta}$ is defined 
in (\ref{eq:four.2.300}).  \endproof
\end{theorem}

In the rest of the paper, we will develop our analysis considering Algorithm {\bf ADFOUR} rather
than {\bf F-ADFOUR}; this is just for the sake of simplicity, since all the conclusions
extend in a straightforward manner to the latter version as well.

\subsection{{\bf A-ADFOUR:} An aggressive version of ADFOUR}\label{subsec:aggress}
Theorem \ref{teo:four1} indicates that even if one {chooses}
$\theta$ very close to $1$, the predicted
error reduction rate $\rho=\rho(\theta)$ is always bounded from below by the quantity
$\sqrt{1 - \frac{\alpha_*}{\alpha^*}}$. Such a result looks overly pessimistic, particularly in the
case of smooth (analytic) solutions, since a Fourier method allows for an exponential decay of
the error as the number of (properly selected) active degrees of freedom is increased.
{Fig \ref{fig:theta-var} displays the influence of D\"orfler
  parameter on the decay rate and number of solves:
  choosing $\theta$ closer to 1
  does not significantly affect the rate of decay of the error versus
  the number of activated degrees of
freedom, but it significantly reduces the number of iterations.
This in turn reduces the computational cost measured in terms of
Galerkin solves.}  
\begin{figure}[t!]\label{fig:theta-var}
\begin{center}
\includegraphics[width=.5\textwidth]{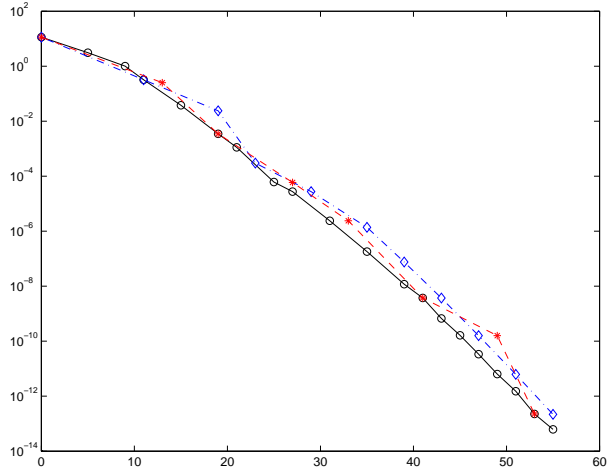}
\end{center}
\caption{Residual norm vs number of degrees of freedom activated by {\bf ADFOUR},
for different choices of D\"orfler parameter {$\theta$}; 
solid line: $\theta=1-10^{-1}$; dash-dotted line: 
$\theta=1-10^{-2}$; dashed line: $\theta=1-10^{-3}$. 
{The symbols (circles, diamonds, stars) identify the various
{\bf ADFOUR} iterations for the sample 1D problem (\ref{eq:four03}) with
analytic solution $u(x)={\rm exp}(\cos 2x +\sin x)$ and coefficients
with $\nu=1+\frac12 \sin 3x$ and $\sigma={\rm exp}(2\cos 3x)$.
}
}
\end{figure}

Motivated by this observation, hereafter we consider a variant of Algorithm {\bf ADFOUR}, which
-- under the assumptions of Property \ref{prop:inverse.alg} or \ref{prop:inverse.matrix-estimate}
 -- guarantees an arbitrarily large error reduction per iteration, 
provided the set of the new degrees of freedom detected by {\bf D\"ORFLER} is suitably enriched.

At the $n$-th iteration, let us define the {set} 
$\Lambda_{n+1}:=\Lambda_{n} \cup \partial\Lambda_{n}$ 
by setting
\begin{equation}\label{eq:aggr1}
\begin{split}
\widetilde{\partial\Lambda}_{n}:=& \text{\bf D\"ORFLER}(r_n, \theta)\\
\partial\Lambda_{n}:=& \text{\bf ENRICH}(\widetilde{\partial\Lambda}_{n},J) \;, 
\end{split}
\end{equation}
where the latter procedure and the value of the integer $J$ will be defined later on.
We recall that the set $\widetilde{\partial\Lambda}_n$ is such that 
$g_n= P_{\widetilde{\partial\Lambda}_n} r_n$ satisfies 
$$
\Vert r_n- g_n \Vert \leq \sqrt{1-\theta^2}  \Vert r_n \Vert 
$$
(see (\ref{eq:four.2.5.5bis})). Let $w_n \in V$ be the solution of $L w_n = g_n$, which in general
will have infinitely many components, and let us split it as
$$
w_n= P_{\Lambda_{n+1}} w_n + P_{\Lambda_{n+1}^c} w_n =: y_n + z_n \in V_{\Lambda_{n+1}} \oplus 
 V_{\Lambda_{n+1}^c} \;.
$$
Then, by the minimality property of the Galerkin solution in the energy norm and by (\ref{eq:four.1bis})
and (\ref{eq:four.2.1bis}), one has
\begin{eqnarray*}
\tvert u - u_{n+1} \tvert &\leq&  \tvert u - (u_{n}+y_{n}) \tvert \leq  
 \tvert u- u_n - w_{n} + z_{n} \tvert \\
&\leq& \frac1{\sqrt{\alpha_*}} \Vert L(u- u_n - w_{n}) \Vert + \sqrt{\alpha^*}\Vert z_{n} \Vert
= \frac1{\sqrt{\alpha_*}} \Vert r_n - g_{n} \Vert  + \sqrt{\alpha^*} \Vert z_{n} \Vert \;.
\end{eqnarray*}
Thus,
$$
\tvert u - u_{n+1} \tvert \leq \frac{1}{\sqrt{\alpha_*}}\sqrt{(1-\theta^2)} \, \Vert r_n \Vert 
+ \sqrt{\alpha^*} \Vert z_{n}\Vert \;.
$$
Now we can write $z_n= \big( P_{\Lambda_{n+1}^c} L^{-1} P_{\widetilde{\partial\Lambda}_n} \big) r_n $; hence,
if $\Lambda_{n+1}$ is defined in such a way that
$$
 k \in \Lambda_{n+1}^c \quad \text{and} \quad \ell \in \widetilde{\partial\Lambda}_n\qquad \Rightarrow \qquad
|k - \ell | > J \;,
$$
then we have 
$$
\Vert P_{\Lambda_{n+1}^c} L^{-1} P_{\widetilde{\partial\Lambda}_n} \Vert \leq
\Vert \mathbf{A}^{-1}-(\mathbf{A}^{-1})_J \Vert \leq 
\psi_{\mathbf{A}^{-1}} (J,\bar{\eta}_L) \;,
$$ 
where we have used (\ref{eq:trunc-invmatr-err}). Now, $J>0$ can be chosen to satisfy
\begin{equation}\label{eq:aggr2}
\psi_{\mathbf{A}^{-1}} (J,\bar{\eta}_L) \leq \sqrt{\frac{1-\theta^2}{\alpha_* \alpha^*}} \;,
\end{equation}
in such a way that
\begin{equation}\label{eq:aggr_error_reduct}
\tvert u - u_{n+1} \tvert \leq \frac1{\sqrt{\alpha_*}} \sqrt{1-\theta^2} \, \Vert r_n \Vert 
\leq \left(\frac{\alpha^*}{\alpha_*}\right)^{1/2} \!\!\! \sqrt{1-\theta^2} \, \tvert u - u_{n} \tvert \;.
\end{equation}
Note that, as desired, the new error reduction rate
\begin{equation}\label{eq:aggr3}
\bar{\rho}=\left(\frac{\alpha^*}{\alpha_*}\right)^{1/2}\! \!  \sqrt{1-\theta^2}
\end{equation}
can be made arbitrarily small by choosing $\theta$ arbitrarily close to $1$.
The procedure {\bf ENRICH} is thus defined as follows: 
\begin{itemize}
\item $\Lambda^* := \text{\bf ENRICH}(\Lambda,J)$ \\
Given an integer $J \geq 0$ and a finite set $\Lambda \subset \mathbb{Z}^d$, the output is the set
$$
\Lambda^* := 
\{ k \in \mathbb{Z}^d\ : \ \text{ there exists } \ell \in \Lambda \text{  such that } |k - \ell| \leq J \} \;.
$$
\end{itemize}
Note that since the procedure adds a $d$-dimensional ball of radius $J$ around each point of $\Lambda$, the cardinality
of the new set $\Lambda^*$ can be estimated as
\begin{equation}\label{eq:estim-enrich}
|\Lambda^*| \leq |\overline{B_d(0,J)}\cap \mathbb{Z}^d|\, |\Lambda| \sim \omega_d J^d |\Lambda| \;,
\end{equation}
where $\omega_d$ {is the measure of the $d$-dimensional Euclidean}
unit ball $B_d(0,1)$ centered at the origin.

It is convenient for future reference to denote by ${\partial\Lambda}_{n}:= \text{\bf E-D\"ORFLER}(r_n, \theta,J)$
the procedure described in (\ref{eq:aggr1}).
We summarize our results in the following theorem.
\begin{theorem}[{contraction property of {\bf A-ADFOUR}}]\label{teo:four3}
Consider the aggressive variant {\bf A-ADFOUR} of the adaptive algorithm {\bf ADFOUR}, 
in which the step ${\partial\Lambda}_{n}:= \text{\bf D\"ORFLER}(r_n, \theta)$
 is replaced by 
$$
{\partial\Lambda}_{n}:= \text{\bf E-D\"ORFLER}(r_n, \theta, J) \;,
$$ 
where $\theta$ is such that $\bar{\rho}$ defined in (\ref{eq:aggr3}) is smaller than $1$,
and $J$ is the smallest integer for which (\ref{eq:aggr2})
is fulfilled. Let the assumptions of Property \ref{prop:inverse.alg} 
or \ref{prop:inverse.matrix-estimate} be satisfied.
Then, the same conclusions of Theorem \ref{teo:four1} hold true for this variant, with the contraction
factor $\rho$ replaced by $\bar{\rho}$.  \endproof
\end{theorem}

\subsection{{C-ADFOUR and PC-ADFOUR:} ADFOUR with coarsening}\label{subsec:coarse-adfour} 

The adaptive algorithm {\bf ADFOUR} and its variants introduced above are not guaranteed to be 
optimal in terms of complexity. 
Indeed, the discussion in the forthcoming Sect. \ref{sec:spars-res} 
{for the exponential case} will indicate 
that the residual $r(u_\Lambda)$ may be
significantly less sparse than the corresponding Galerkin solution $u_\Lambda$; in particular, we will see
 that many indices in $\Lambda$, activated in an early stage of the adaptive
process, could be lately discarded since the corresponding components of $u_\Lambda$ are zero.
For these reasons, we propose here a new variant of algorithm {\bf ADFOUR}, which
incorporates a recursive coarsening step.

The algorithm is constructed through the procedures {\bf GAL}, {\bf RES}, {\bf D\"ORFLER} already introduced 
in Sect. \ref{sec:defADFOUR}, together with the new procedure {\bf COARSE} defined as follows: 
\begin{itemize}
\item $\Lambda := {\bf COARSE}(w, \epsilon)$\\
Given a function $w \in V_{\Lambda^*}$ for some finite index set $\Lambda^*$, and an accuracy $\epsilon$
which is known to satisfy $\Vert u -  w \Vert \leq  \epsilon$,
 the output $\Lambda \subseteq \Lambda^*$ is a set of minimal cardinality such that
\begin{equation}\label{eq:def-coarse}
\Vert w - P_\Lambda w \Vert \leq 2 \epsilon \;.
\end{equation}
\end{itemize}

We will subsequently show (see Theorem \ref{T:coarsening}) that the cardinality $|\Lambda|$ 
is optimally related to the sparsity class of $u$. The following result will be used several times in the paper.
\begin{property}[{coarsening}]\label{prop:cons-coarse}
The procedure {\bf COARSE} guarantees the bounds 
\begin{equation}\label{eq:def-coarse-bis}
\Vert u - P_\Lambda w \Vert \leq 3 \epsilon 
\end{equation}
and, for the Galerkin solution $u_\Lambda \in V_\Lambda$,
\begin{equation}\label{eq:def-coarse-ter}
\tvert u - u_\Lambda  \tvert \leq 3 \sqrt{\alpha^*}\epsilon \;.
\end{equation}
\end{property}
\proof
The first bound is trivial, the second one follows from the minimality property of the Galerkin solution 
in the energy norm and from (\ref{eq:four.1bis}):
$$
\tvert u - u_\Lambda \tvert \leq \tvert u - P_{\Lambda}w \tvert \leq \sqrt{\alpha^*}
\Vert u-P_{\Lambda} w \Vert \leq 3 \sqrt{\alpha^*} \epsilon \;. \qquad \quad \endproof
$$

Given two parameters $\theta \in (0,1)$ and $tol \in [0,1)$, we define the following adaptive algorithm
with coarsening. 

{\bf Algorithm C-ADFOUR}($\theta, \ tol$)
\begin{itemize}
\item[\ ] Set $r_0:=f$, $\Lambda_0:=\emptyset$, $n=-1$
\item[\ ] do
	\begin{itemize}
	\item[\ ] $n \leftarrow n+1$
	\item[\ ] set $\Lambda_{n,0}=\Lambda_n$, $r_{n,0}=r_n$
	\item[\ ] $k=-1$
	\item[\ ] do
	\begin{itemize}
		\item[\ ] $k \leftarrow k+1$
		\item[\ ] $\partial\Lambda_{n,k}:= \text{\bf D\"ORFLER}(r_{n,k}, \theta)$
		\item[\ ] $\Lambda_{n,k+1}:=\Lambda_{n,k} \cup \partial\Lambda_{n,k}$
		\item[\ ] $u_{n,k+1}:= {\bf GAL}(\Lambda_{n,k+1})$
		\item[\ ] $r_{n,k+1}:= {\bf RES}(u_{n,k+1})$
	\end{itemize}
	\item[\ ] while $\Vert r_{n,k+1} \Vert > \sqrt{1-\theta^2} \Vert r_{n} \Vert$
	\item[\ ] $\Lambda_{n+1}:={\bf COARSE}\left(u_{n,k+1}, \tfrac1{\sqrt\alpha_*} \Vert r_{n,k+1} \Vert \right)$
	\item[\ ] $u_{n+1}:={\bf GAL}(\Lambda_{n+1})$
	\item[\ ] $r_{n+1}:={\bf RES}(u_{n+1})$
	\end{itemize} 
\item[\ ]  while $\Vert r_{n+1} \Vert > tol $
\end{itemize}

We observe that the {specific} choice of accuracy
$\epsilon=\epsilon_n {= \frac{1}{\sqrt{\alpha_*}} \|r_{n,k+1}\|}$ 
in each call of {\bf COARSE} in the algorithm above is
motivated by the wish of guaranteeing a fixed reduction of the 
residual {and error} at each outer iteration. This is
made precise in the following theorem.

\begin{theorem}[{contraction property of {\bf C-ADFOUR}}]\label{teo:adfour-coarse-convergence}
The algorithm {\bf C-ADFOUR} {satisfies} 
\begin{enumerate}[\rm (i)]
\item
The number of iterations of each inner loop is finite and bounded independently of $n$;
\item
The sequence of residuals $r_n$ and errors $u-u_n$ generated for $n \geq 0$ by 
the algorithm satisfies the inequalities
\begin{equation}\label{eq:adgev.1}
\Vert r_{n+1} \Vert \leq \rho  \Vert r_{n} \Vert 
\end{equation}
and 
\begin{equation}\label{eq:adgev.2}
\tvert u-u_{n+1} \tvert \leq \rho \tvert u-u_{n} \tvert 
\end{equation}
for
\begin{equation}\label{eq:adgev.2bis}
\rho=3 \frac{\alpha^*}{\alpha_*}\sqrt{1-\theta^2} \;.
\end{equation}
In particular, if $\theta$ is chosen in such a way that $\rho<1$, 
for any $tol>0$ the algorithm terminates in a finite number of iterations, whereas for $tol=0$
the sequence $u_n$ converges to $u$ in $H^1_p(\Omega)$ as $n \to \infty$.
\end{enumerate}
\end{theorem}
\begin{proof}
{\rm (i)} 
For any fixed $n$, each inner iteration behaves as the algorithm {\bf ADFOUR} considered
in {Sect. \ref{sec:defADFOUR}}. Hence, setting again 
$\rho=\sqrt{1-\frac{\alpha_*}{\alpha^*} \theta^2}$, we have as in Theorem \ref{teo:four1}
$$
  \tvert u-u_{n,k+1} \tvert \leq \rho^{k+1} \tvert u-u_n \tvert \;,
$$ 
which implies, by (\ref{eq:four.2.1bis}),
$$
\Vert r_{n,k+1} \Vert \leq \sqrt{\alpha^*}   \tvert u-u_{n,k+1} \tvert 
\leq \sqrt{\alpha^*} \rho^{k+1} \tvert u-u_n \tvert 
\leq \sqrt{\frac{\alpha^*}{\alpha_*}} \rho^{k+1} \Vert r_{n} \Vert \;.
$$
This shows that the termination criterion
\begin{equation}\label{eq:adgev.3}
\Vert r_{n,k+1} \Vert \leq   \sqrt{1-\theta^2} \, \Vert r_{n} \Vert
\end{equation}
is certainly satisfied if
$$
\sqrt{\frac{\alpha^*}{\alpha_*}} \rho^{k+1} \leq  \sqrt{1-\theta^2} \;,
$$
i.e., as soon as
$$
k+1 \geq \frac{\log \left( \tfrac{\alpha_*}{\alpha^*}(1-\theta^2) 
\right)}{2 \log \rho} {> k}\;.
$$
We conclude that the number $K_n { = k+1}$ of inner iterations is
bounded by {$1 + \frac{\log \left( \tfrac{\alpha_*}{\alpha^*}(1-\theta^2) 
\right)}{2 \log \rho}$, which is independent of $n$.}

{\rm (ii)} By (\ref{eq:four.2.1}), we have 
$$
\Vert u-u_{n,k+1} \Vert \leq \frac1{\alpha_*} \Vert r_{n,k+1} \Vert \;.
$$
At the exit of the inner loop, the quantity on the right-hand side 
is precisely the parameter $\epsilon_n$ {fed} to
the procedure {\bf COARSE}; then, Property \ref{prop:cons-coarse} yields
$$ 
\tvert u - u_{n+1} \tvert \leq 3 \sqrt{\alpha^*} \epsilon_n \;.
$$
On the other hand, the termination criterion (\ref{eq:adgev.3}) yields
$$
\epsilon_n \leq \frac1{\alpha_*}  \sqrt{1-\theta^2} \Vert r_{n} \Vert \;,
$$
so that
$$
\tvert u - u_{n+1} \tvert
\leq 3 \frac{\sqrt{\alpha^*}}{\alpha_*}  \sqrt{1-\theta^2} \Vert r_{n} \Vert \;.
$$
This bound together with the left-hand inequality in (\ref{eq:four.2.1bis}) applied to $r_{n+1}$ yields
 (\ref{eq:adgev.1}), whereas the same inequality applied to 
$r_{n}$ yields (\ref{eq:adgev.2}). 
\end{proof}

A coarsening step can also be inserted in the aggressive algorithm {\bf A-ADFOUR} considered
in Sect. \ref{subsec:aggress}; indeed, the enrichment step {\bf ENRICH} could activate a larger
number of degrees of freedom than really needed, endangering optimality. The algorithm we now propose can
be viewed as a variant of {\bf C-ADFOUR}, in which the use of {\bf E-D\"ORFLER} instead of {\bf D\"ORFLER}
allows one to take a single inner iteration; in this respect, one can consider the enrichment step  as a
``prediction'', and the coarsening step as a ``correction'', of the new set of active degrees of freedom.
For this reason, we call this variant the {\bf Predictor/Corrector-ADFOUR}, or simply {\bf PC-ADFOUR}. 

Given two parameters $\theta \in (0,1)$ and $tol \in [0,1)$, we choose $J\geq 1$ as the smallest integer
for which (\ref{eq:aggr2}) is fulfilled, and we define the following adaptive algorithm. 

\medskip
{\bf Algorithm PC-ADFOUR}($\theta, \ tol, \ J$)
\begin{itemize}
\item[\ ] Set $r_0:=f$, $\Lambda_0:=\emptyset$, $n=-1$
\item[\ ] do
	\begin{itemize}
	\item[\ ] $n \leftarrow n+1$
	\item[\ ] $\widehat{\partial\Lambda}_{n}:= \text{\bf E-D\"ORFLER}(r_{n}, \theta, J)$		
        \item[\ ] $\widehat\Lambda_{n+1}:=
			\Lambda_{n} \cup \widehat{\partial\Lambda}_{n}$
	\item[\ ] $\widehat{u}_{n+1}:= {\bf GAL}(\widehat\Lambda_{n})$
	\item[\ ] $\Lambda_{n+1}:=
		{\bf COARSE}\left(\widehat{u}_{n+1}, \frac{1}{\alpha_*}
		\sqrt{1-\theta^2}\|r_n\|\right)$
	\item[\ ] $u_{n+1}:={\bf GAL}(\Lambda_{n+1})$
	\item[\ ] $r_{n+1}:={\bf RES}(u_{n+1})$
	\end{itemize} 
\item[\ ]  while $\Vert r_{n+1} \Vert > tol $
\end{itemize}

\begin{theorem}[{contraction property of {\bf PC-ADFOUR}}]\label{teo:pc-adfour-convergence}
{If} the assumptions of Property \ref{prop:inverse.alg} 
or Property \ref{prop:inverse.matrix-estimate} be satisfied, {then} 
the statement (ii) of Theorem \ref{teo:adfour-coarse-convergence} applies to
Algorithm {\bf PC-ADFOUR} as well.
\end{theorem}
\begin{proof}
The first inequalities in both (\ref{eq:aggr_error_reduct}) and (\ref{eq:four.1bis}) yield
$$
\Vert u - \widehat{u}_{n+1} \Vert \leq \frac1{\alpha_*}
\sqrt{1-\theta^2} \, \Vert r_n \Vert \; .
$$
{Since} the right-hand side is precisely the parameter $\epsilon_n$ {fed} to
the procedure {\bf COARSE}, one proceeds as in the proof of Theorem \ref{teo:adfour-coarse-convergence}.
\end{proof}

\section{Nonlinear approximation in Fourier spaces}\label{sec:nl}

\subsection{Best $N$-term approximation and rearrangement}\label{sec:abstract-nl}
Given any nonempty finite index set $\Lambda \subset \mathbb{Z}^d$ and the corresponding subspace
$V_\Lambda \subset V=H^1_p(\Omega)$ of dimension $|\Lambda|=\text{card}\, \Lambda$, 
the best approximation of $v$ in $V_{\Lambda}$ is the orthogonal projection
of $v$ upon $V_{\Lambda}$, i.e. the function $P_{\Lambda} v = \sum_{k \in \Lambda} \hat v_k \phi_k$,
which satisfies
$$
\Vert v - P_{\Lambda} v \Vert= 
\left(\sum_{k \not \in \Lambda} |\hat{V}_k|^2\right)^{1/2} 
$$
(we set $P_\Lambda v = 0$ if $\Lambda=\emptyset$).
For any integer $N \geq 1$, we minimize this error over all possible choices of $\Lambda$
with cardinality $N$, {thereby} leading to the {\sl best $N$-term approximation error}
$$
E_N(v)= \inf_{\Lambda \subset\mathbb{Z}^d , 
\ |\Lambda|=N} \Vert v - P_{\Lambda} v \Vert \;.
$$
A way to construct a {\sl best $N$-term approximation} $v_N$ of $v$ consists of
rearranging the coefficients of $v$ in decreasing order of modulus
$$
\vert \hat{V}_{k_1}\vert \geq \ldots \geq \vert \hat{V}_{k_n} \vert 
\geq \vert \hat{V}_{k_{n+1}} \vert \geq \dots 
$$
and setting $v_N=P_{\Lambda_N}v$ with $\Lambda_N = \{ k_n \ : \ 1 \leq
n \leq N \}$. 
{As already mentioned in the Introduction, 
let us denote from now on $v_n^*=\hat{V}_{k_n}$
the rearranged and rescaled Fourier coefficients of $v$}. Then,
$$
E_N(v)= \left(\sum_{n>N} |v_n^*|^2 \right)^{1/2} \;.
$$

Next, given a strictly decreasing function $\phi : \mathbb{N} \to \mathbb{R}_+$ such that
$\phi(0)=\phi_0$ for some $\phi_0 >0$
and $\phi(N) \to 0$ when $N \to \infty$, we introduce the corresponding
 {\sl sparsity class} ${\mathcal A}_\phi$ by setting
\begin{equation}\label{eq:nl.gen.001}
{\mathcal A}_\phi = {\Big\{ v \in V \ : \ \Vert v \Vert_{{\mathcal A}_\phi}:= \sup_{N \geq 0} 
\, \frac{E_N(v)}{\phi(N)} < +\infty \Big\} \; .}
\end{equation}
We point out that in applications $\Vert v\Vert_{{\mathcal A}_\phi}$
need not be a (quasi-)norm {since}
${\mathcal A}_\phi$ need not be a linear space. Note however that $\Vert v \Vert_{{\mathcal A}_\phi}$ always controls
the $V$-norm of $v$, since $\Vert v \Vert = E_0(v) \leq {\phi_0} \Vert v \Vert_{{\mathcal A}_\phi}$.
Observe that $v \in {\mathcal A}_\phi$ iff there exists a constant $c>0$ such that
\begin{equation}\label{eq:nl.gen.1}
E_N(v) \leq c \phi(N)\;, \qquad \forall N \geq 0 \;.
\end{equation}
The quantity $\Vert v \Vert_{{\mathcal A}_\phi}$ {dictates}
the minimal number $N_\varepsilon$ of basis functions
needed to approximate $v$ with accuracy $\varepsilon$. {In fact}, from the relations
$$
E_{N_\varepsilon}(v) \leq \varepsilon  < E_{N_\varepsilon-1}(v) 
\leq \phi(N_\varepsilon-1) \Vert v \Vert_{{\mathcal A}_\phi} \;,
$$
{and the monotonicity of $\phi$, we obtain}
\begin{equation}\label{eq:nl.gen.2o}
N_\varepsilon \leq \phi^{-1}\left(\frac{\varepsilon}{\Vert v \Vert_{{\mathcal A}_\phi}}\right) +1 \;.
\end{equation}
The second addend on the right-hand side can be absorbed by a multiple of the first one, provided $\varepsilon$
is sufficiently small; in other words, it is not restrictive to assume that there exists a constant
$\kappa$ slightly larger than $1$ such that
\begin{equation}\label{eq:nl.gen.2}
N_\varepsilon \leq \kappa \, \phi^{-1}\left(\frac{\varepsilon}{\Vert v \Vert_{{\mathcal A}_\phi}}\right) \;.
\end{equation}

\begin{remark}[{sparsity class for $V'$}]\label{rem:nl1}{\rm
Replacing $V$ by $V'$ in (\ref{eq:nl.gen.001}) leads to the definition of a sparsity class, still denoted
by ${\mathcal A}_\phi$, in the space of linear continuous forms 
{$f$} on $H^1_p(\Omega)$. This observation applies to the
subsequent definitions as well (e.g., for the class ${\mathcal A}^{\eta,t}_G$).  In essence,
we will treat in a unified way the nonlinear approximation of a function $v \in H^1_p(\Omega)$ and of
a form ${f} \in H^{-1}_p(\Omega)$. \endproof
}
\end{remark}

Throughout the paper, we shall consider two main families of sparsity classes, identified by specific
choices of the function $\phi$ depending upon one or more parameters. The first family 
is related to the best approximation in {\em Besov} spaces of periodic functions, thus accounting for a
finite-order regularity in $\Omega$; the corresponding functions $\phi$ exhibit an algebraic decay as
$N \to \infty$, {which motivates our terminology of
{\em algebraic classes}}.
The second family is related to the best approximation in {\em Gevrey} spaces of periodic functions,
which are formed by infinitely-differentiable functions in $\Omega$; the associated $\phi$'s exhibit an
exponential decay, and for this reason such classes will be referred to as {\em exponential classes}.
Properties of both families are collected hereafter.

\subsection{Algebraic classes}\label{sec:nlb}
The following is the counterpart for Fourier approximations of by now well-known 
nonlinear approximation settings {\cite{DeVore:1998}}, e.g. for wavelets or nested finite elements. For this reason, we just state
definitions and properties without proofs.

For $s>0$, let us introduce the function 
\begin{equation}\label{eq:nlb.200}
\phi(N)=N^{-s/d} \qquad \qquad \text{for } N \geq 1\;,
\end{equation} 
and $\phi(0)=\phi_0>1$ arbitrary, with inverse 
\begin{equation}\label{eq:nlb.201}
\phi^{-1}(\lambda) = \lambda^{-d/s} \qquad \qquad \text{for } \lambda \leq 1\;,
\end{equation}
and let us consider the corresponding class ${\mathcal A}_\phi$ defined in (\ref{eq:nl.gen.001}).
\begin{definition}[{algebraic class of functions}]\label{def:ABes} 
We denote by ${\mathcal A}^{s}_B$ the subset of $V$ defined as
$$
{\mathcal A}^{s}_B {:= \Big\{ v \in V \ : \ \Vert v \Vert_{{\mathcal A}^{s}_B}:= 
\Vert v \Vert + \sup_{N \geq 1} \, E_N(v) \, N^{s/d}  < +\infty \Big\} \;.}
$$
\end{definition}
It is immediately seen that ${\mathcal A}^{s}_B$ contains the Sobolev space of periodic functions $H^{s+1}_p(\Omega)$.
On the other hand, it is proven in \cite{DeVore-Temlyakov:1995}, as a part of a more general result, that 
for $0 < \sigma, \tau \leq \infty$, the Besov space 
${B^{s+1}_{\tau,\sigma}(\Omega)}=B^{s+1}_\sigma(L^\tau(\Omega))$ 
is contained in ${\mathcal A}^{s^*}_B$ {provided $s^* := s-d(1/\tau-1/2)_+>0$.} 

Let us associate the quantity $\tau>0$ to the parameter $s$, via the relation
$$
\frac1\tau=\frac{s}{d} + \frac12 \;.
$$
The condition for a function $v$ to belong to some class ${\mathcal A}^{s}_B$ can be equivalently 
stated as a condition on the vector 
${\bm v}=(\hat{V}_k)_{k \in \mathbb{Z}^d}$ of its Fourier coefficients, precisely, on the rate of decay of
the non-increasing rearrangement ${\bm v}^* =(v^*_n)_{n \geq 1}$ of ${\bm v}$. 
\begin{definition}[{algebraic class of sequences}]\label{def:elpicBes}
{Let $\ell_B^{s}(\mathbb{Z}^d)$ be} the subset of sequences 
${\bm v} \in \ell^{2}(\mathbb{Z}^d)$ so that
$$
\Vert {\bm v} \Vert_{\ell_B^{s}(\mathbb{Z}^d)} := \sup_{n \geq 1} 
n^{1/\tau} |v_n^*| < +\infty \;.
$$
\end{definition}
\noindent Note that this space is often denoted by $\ell^\tau_w(\mathbb{Z}^d)$ in the literature, being an example
of Lorentz space.

The relationship between  ${\mathcal A}^{s}_B$ and $\ell_B^{s}(\mathbb{Z}^d)$ is stated
in the following Proposition.

\begin{proposition}[{equivalence of algebraic classes}]\label{prop:nlg1-alg}
Given a function $v \in V$ and the sequence ${\bm v}$ of its Fourier coefficients,
 one has  $v \in {\mathcal A}^{s}_B$ if and only if
${\bm v} \in \ell_B^{s}(\mathbb{Z}^d)$, with
$$
\Vert v \Vert_{{\mathcal A}^{s}_B} \lsim 
\Vert {\bm v} \Vert_{\ell_B^{s}(\mathbb{Z}^d)}
\lsim \Vert v \Vert_{{\mathcal A}^{s}_B}\,.  
$$
\end{proposition}

At last, we note that the quasi-Minkowski inequality
$$
\Vert {\bm u}+{\bm v} \Vert_{\ell_B^{s}(\mathbb{Z}^d)} \leq C_s \left(
\Vert {\bm u} \Vert_{\ell_B^{s}(\mathbb{Z}^d)} + \Vert {\bm v} \Vert_{\ell_B^{s}(\mathbb{Z}^d)} \right) 
$$
holds in $\ell_B^{s}(\mathbb{Z}^d)$, yet the constant $C_s$ blows up exponentially 
as $s \to \infty$.

\subsection{Exponential classes}\label{sec:nlg}
We first recall the definition of Gevrey spaces of periodic functions in $\Omega=(0,2\pi)^d$ 
(see \cite{Foias-Temam:1989}). Given reals $\eta > 0 $, {$0<t\le d$} 
and $s \geq 0$, we set
$$
G^{\eta,t, s}_p(\Omega) { ~:=\Big\{ v \in L^2(\Omega) \ : 
\ \Vert v \Vert_{G,\eta,t,s}^2 = 
\sum_{k\in\mathbb{Z}} \e^{2\eta |k|^t}(1+|k|^{2s}) |\hat{v}_k|^2 < +\infty \Big\} \;.}
$$ 
Note that $G^{\eta,t, s}_p(\Omega)$ is contained in all Sobolev spaces of periodic functions
$H^r_p(\Omega)$, $r \geq 0$. Furthermore, if $t \geq 1$, $G^{\eta,t, s}_p(\Omega)$ is made of 
analytic functions.

Gevrey spaces have been introduced to study the $C^\infty$ and analytical regularity of the
solutions of partial differential equations. For our elliptic problem 
(\ref{eq:four03}), {the following statement is an example of shift theorem in Gevrey spaces}.
\begin{theorem}[{shift theorem}]
{If the assumptions of Property
  \ref{prop:inverse.matrix-estimate} are satisfied,
then for any $\eta<\bar\eta_L$, ${0<t\leq 1}$ and $s \geq -1$, $L$ is an isomorphism between 
$G^{\eta,t, s+2}_p(\Omega)$ and $G^{\eta,t, s}_p(\Omega)$.}
\end{theorem}
\begin{proof}
{Proceeding as in Sect. \ref{sec:algebraic_repres},
it is immediate to see} that the problem $Lu=f$ can be equivalently formulated as $\bA \bu = \bff$, where the vectors $\bff$ and $\bu$ contain the Fourier coefficients of functions $f$ and $u$ normalized in $H_p^s(\Omega)$ and $H_p^{s+2}(\Omega)$, respectively. 
{If $\bW=\text{diag}(\e^{\eta\vert k \vert^t})$ is a bi-infinite
  diagonal exponential matrix, then we can write 
$\bW\bu=\bW\bA^{-1}\bff=(\bW\bA^{-1} \bW^{-1})\bW\bff$. 
We observe that property
$\|\bW\bu\|_{\ell^2}\lesssim\|\bW\bff\|_{\ell^2}$,
which implies the thesis, is a consequence of
$\| \bW\bA^{-1}\bW^{-1} \|_{\ell^2}\lesssim 1 $.

To show the latter inequality, we let ${\bf x, \bf
  y}\in\ell^2(\mathbb{Z}^d)$ and notice that
\begin{equation*}
|{\bf y}^T \bW \bA^{-1}\bW^{-1} {\bf x}| \le c_L
\sum_{m\in\mathbb{Z}^d}  e^{-\bar\eta_L |m|} \sum_{k\in\mathbb{Z}^d} |y_{m+k}|
e^{\eta |m+k|^t} e^{-\eta |k|^t} |x_k|.
\end{equation*}
Since ${0<t\le 1}$, we deduce $|m+k|^t \le |m|^t + |k|^t$ and 
$e^{\eta (|m+k|^t- |k|^t)} \le e^{\eta |m|^t}$, whence
\[
\|\bW\bA^{-1}\bW^{-1}{\bf x}\| = \sup_{{\bf y}\in\mathbb{Z}^d}
\frac{|{\bf y}^T\bW\bA^{-1}\bW^{-1}{\bf x}|}{\|\bf y\|}
\le c_L
\sum_{m\in\mathbb{Z}^d} e^{(-\bar\eta_L + \eta) |m|^t} \|\bf x\|
\lsim \|\bf  x\|
\]
because $\bar\eta_L > \eta$ and the series converges. This implies
the desired estimate.
}
\end{proof}

From now on, we fix $s=1$ and we normalize again the Fourier coefficients of a function $v$
with respect to the $H^1_p(\Omega)$-norm. Thus, we set
\begin{equation}\label{eq:gevrey-class}
G^{\eta,t}_p(\Omega)= G^{\eta,t, 1}_p(\Omega)= 
\{ v \in V \ : \ \Vert v \Vert_{G,\eta,t}^2 = 
\sum_k \e^{2\eta |k|^t} |\hat{V}_k|^2 < +\infty \} \;.
\end{equation}
Functions in $G^{\eta,t}_p(\Omega)$ can be approximated by the linear orthogonal projection
$$
P_M v = \sum_{|k| \leq M} \hat{V}_k \phi_k \;,
$$
for which we have
\begin{eqnarray*}
\Vert v -P_M v \Vert^2 &=& \sum_{|k| > M} |\hat{V}_k|^2 
= \sum_{|k| > M} \e^{-2\eta |k|^t} \e^{2\eta |k|^t} |\hat{V}_k|^2  \\
&\leq & \e^{-2\eta M^t} \sum_{|k| > M} \e^{2\eta |k|^t} |\hat{V}_k|^2 \leq 
\e^{-2\eta M^t} \Vert v \Vert_{G,\eta,t}^2 \;.
\end{eqnarray*}
As already observed {in Property \ref{prop:matrix-estimate}}, 
setting $N={\rm card}\{k \, : \, |k|\leq M \}$, one has $N \sim \omega_d M^d$,
so that
\begin{equation}\label{eq:nlg.3}
E_N(v) \leq \Vert v -P_M v \Vert \lsim 
{\rm exp}\left(- \eta \omega_d^{-t/d} N^{t/d} \right) \Vert v \Vert_{G,\eta,t} \;.
\end{equation}
Hence, we are led to introduce the function 
\begin{equation}\label{eq:nlg.300}
\phi(N)={\rm exp}\left(- \eta \omega_d^{-t/d} N^{t/d} \right) \quad \qquad (N \geq 0) \;,
\end{equation} 
whose inverse is given by
\begin{equation}\label{eq:nlg.301}
\phi^{-1}(\lambda) = \frac{\omega_d}{\eta^{d/t}}\left( \log \frac1\lambda \right)^{d/t}
\quad \qquad (\lambda \leq 1) \;,
\end{equation}
and to consider the corresponding class ${\mathcal A}_\phi$ defined in (\ref{eq:nl.gen.001}), 
which therefore contains $G^{\eta,t}_p(\Omega)$.
\begin{definition}[{exponential class of functions}]\label{def:AGev} 
We denote by ${\mathcal A}^{\eta,t}_G$ the subset of $G^{\eta,t}_p(\Omega)$ defined as
$$
{\mathcal A}^{\eta,t}_G { := \Big\{ v \in V \ : 
\ \Vert v \Vert_{{\mathcal A}^{\eta,t}_G}:= 
\sup_{N \geq 0} \, E_N(v) \, {\rm exp}\left(\eta \omega_d^{-t/d}
  N^{t/d} \right) < +\infty \Big\} \;.}
$$
\end{definition}

At this point, we make the subsequent notation easier by introducing the $t$-dependent function
$$
\tau= \frac{t}{d} \le 1 \;.
$$
As in the algebraic case, the class ${\mathcal A}^{\eta,t}_G$ can be equivalently characterized in terms of 
behavior of rearranged sequences of Fourier coefficients. 
\begin{definition}[{exponential class of sequences}]\label{def:elpicGev}
{Let $\ell_G^{\eta,t}(\mathbb{Z}^d)$ be the} subset of sequences 
${\bv} \in \ell^{2}(\mathbb{Z}^d)$ so that \looseness=-1
$$
\Vert {\bv} \Vert_{\ell_G^{\eta,t}(\mathbb{Z}^d)} := \sup_{n \geq 1} 
n^{(1-\tau)/2} {\rm exp}\left(\eta \omega_d^{-\tau} n^{\tau} \right)
|v_n^*| < +\infty \;,
$$
where {${\bv}^*=(v_n^*)_{n=1}^\infty$} is the non-increasing rearrangement of ${\bv}$.
\end{definition}


The relationship between  ${\mathcal A}^{\eta,t}_G$ and $\ell_G^{\eta,t}(\mathbb{Z}^d)$ is stated
in the following Proposition.
\begin{proposition}[{equivalence of exponential classes}]\label{prop:nlg1}
Given a function $v \in V$ and the sequence ${\bm v}=(\hat{V}_k)_{k \in \mathbb{Z}^d}$ of its Fourier coefficients,
 one has  $v \in {\mathcal A}^{\eta,t}_G$ if and only if
${\bv} \in \ell_G^{\eta,t}(\mathbb{Z}^d)$, with
$$
\Vert v \Vert_{{\mathcal A}^{\eta,t}_G} \lsim 
\Vert {\bv} \Vert_{\ell_G^{\eta,t}(\mathbb{Z}^d)}
\lsim \Vert v \Vert_{{\mathcal A}^{\eta,t}_G}\,.
$$
\end{proposition}
\begin{proof}
Assume first that ${\bv} \in \ell_G^{\eta,t}(\mathbb{Z}^d)$. Then, 
$$
E_N(v)^2 = \Vert v - P_N(v) \Vert^2 = \sum_{n>N} |v_n^*|^2
\lsim \sum_{n>N} n^{\tau-1} {\rm exp}\left(-2\eta \omega_d^{-\tau} n^{\tau} \right) 
\Vert{\bv}\Vert_{\ell_G^{\eta,t}(\mathbb{Z}^d)}^2 \;.
$$
Now, setting for simplicity $\alpha=2\eta \omega_d^{-\tau}$, one has
$$
S:=\sum_{n>N} n^{\tau-1} \e^{- \alpha n^{\tau}} \sim 
\int_N^\infty x^{\tau-1} \e^{- \alpha x^{\tau}} dx \;.
$$
The substitution $z=x^{\tau}$ yields
$$
S \sim \frac{d}{t} \int_{N^{\tau}}^\infty \e^{-\alpha z} dz =
\frac{d}{\alpha t} \e^{-\alpha N^{\tau}}
$$
{whence $\|v\|_{\mathcal{A}_G^{\eta,t}} \lsim 
\Vert {\bv} \Vert_{\ell_G^{\eta,t}(\mathbb{Z}^d)}$.}
Conversely, {let} $v \in {\mathcal A}^{\eta,t}_G$. We have to prove that for any $n \geq 1$, one 
has
$$
n^{1-\tau} |v_n^*|^2 \lsim \e^{-\alpha n^{\tau}} \Vert v \Vert_{{\mathcal A}^{\eta,t}_G} \;.
$$
Let $m<n$ be the largest integer such that $n-m \geq n^{1-\tau}$ (note that $0 \leq 1-\tau < 1$),
i.e., $m \sim n(1-n^{-\tau})$. Then,
$$
n^{1-\tau} |v_n^*|^2 \leq (n-m)|v_n^*|^2 \leq \sum_{j=m+1}^n |v_j^*|^2 \leq 
\Vert v - P_m(v) \Vert^2 \leq \e^{-\alpha m^{\tau}} \Vert v \Vert_{{\mathcal A}^{\eta,t}_G}^2 \;.
$$ 
Now, by Taylor expansion,
$$
m^{\tau} \sim n^{\tau}(1-n^{-\tau})^{\tau} = n^{\tau}\left(1- \tau n^{-\tau} + o(n^{-\tau})\right) =
n^{\tau} -\tau + o(1) \;,
$$ 
so that $ \e^{-\alpha m^{\tau}} \lsim  \e^{-\alpha n^{\tau}}$, and 
{$\Vert {\bv} \Vert_{\ell_G^{\eta,t}(\mathbb{Z}^d)} \lsim
\|v\|_{\mathcal{A}_G^{\eta,t}}$ is proven.}
\end{proof}

\medskip
Next, we briefly comment on the structure of the set $\ell_G^{\eta,t}(\mathbb{Z}^d)$. This is not
a vector space, since it may happen that ${\bu},\,{\bv}$ belong to this set, whereas
${\bu}+{\bv}$ does not. Assume for simplicity that $\tau=1$ and consider for instance 
the sequences in $\ell_G^{\eta,t}(\mathbb{Z}^d)$
\begin{eqnarray*}
{\bu} &=& \left(\e^{-\eta}, 0 , \e^{-2\eta}, 0 , \e^{-3 \eta}, 0 , \e^{-4 \eta}, 0 ,\dots \right) \;, \\
{\bv} &=& \left(0, \e^{-\eta}, 0 , \e^{-2\eta}, 0 , \e^{-3 \eta}, 0, \e^{-4 \eta}, \dots \right) \;,
\end{eqnarray*}
Then,
$$
{\bu}+{\bv} = ({\bu}+{\bv})^*= \left(\e^{-\eta},\e^{-\eta} , \e^{-2\eta},\e^{-2\eta} , \e^{-3 \eta},\e^{-3\eta} , \e^{-4 \eta}, \e^{-4\eta} ,\dots \right) \;;
$$
thus, $({\bu}+{\bv})^*_{2j}= \e^{-\eta j}$, so that $\e^{\eta 2j} ({\bu}+{\bv})^*_{2j}
\to \infty$ as $j \to +\infty$, i.e., 
{${\bu}+{\bv} \notin\ell_G^{\eta,t}(\mathbb{Z}^d)$}.
On the other hand, we have the following property.

\begin{lemma}[{quasi-triangle inequality}]\label{L:nlg2}
If ${\bu}_i\in {\ell_G^{\eta_i,t}(\mathbb{Z}^d)}$ for $i=1,2$, then 
${\bu}_1+{\bu}_2 \in {\ell_G^{\eta,t}(\mathbb{Z}^d)}$ with
\[
{\|{\bu}_1+{\bu}_2\|_{\ell_{G}^{\eta,t}} \le
\|{\bu}_1\|_{\ell_{G}^{\eta_1,t}} + \|{\bu}_2\|_{\ell_{G}^{\eta_2,t}},}
\qquad
\eta^{-\frac{1}{\tau}} = \eta_1^{-\frac{1}{\tau}} +
\eta_2^{-\frac{1}{\tau}}.
\]
\end{lemma}
\begin{proof}
We use the characterization given by Proposition \ref{prop:nlg1}, so that
$$
\Vert u_i - P_{N_i}(u_i) \Vert \le \|u_i\|_{\mathcal{A}_G^{\eta,t}} {\rm exp}\left(- \eta
  \omega_d^{-\tau} N_i^{\tau} \right)\quad i=1,2\;.
$$
Given $N\ge1$, we seek $N_1,N_2$ so that
\[
N = N_1+N_2, \qquad
\eta_1 N_1^\tau = \eta_2 N_2^\tau.
\]
This implies
\[
N = N_1 \eta_1^{\frac{1}{\tau}} {\Big(\eta_1^{-\frac{1}{\tau}}+
\eta_2^{-\frac{1}{\tau}} \Big)} = N_1 \eta_1^{\frac{1}{\tau}} \eta^{-\frac{1}{\tau}},
\] 
and
{
\begin{align*}
\|(u_1+u_2)-P_N(u_1+u_2)\| 
& \le \|u_1 - P_{N_1}(u_1)\| + \|u_2 - P_{N_2}(u_2)\| \\
& \le \|u_1\|_{{\mathcal A}_G^{\eta_1,t}} {\rm exp} (-\eta_1
\omega_d^{-\tau} N_1^\tau)
+ |u_2|_{{\mathcal A}_G^{\eta_2,t}} {\rm exp} (-\eta_2 \omega_d^{-\tau}N_2^\tau)
\big)
\\
&
\le \big(\|u_1\|_{{\mathcal A}_G^{\eta_1,t}} + 
\|u_2\|_{{\mathcal A}_G^{\eta_2,t}}\big) {\rm exp} (-\eta \omega_d^{-\tau} N^\tau).
\end{align*}
}
whence the {assertion}.
\end{proof}

{Note that when $\eta_1=\eta_2$ we obtain $\eta = 2^{-\tau}\eta_1\le
  2^{-1}\eta_1$
  thereby extending the previous counterexample.}

\section{Sparsity classes of the residual}\label{sec:spars-res}

For any finite index set $\Lambda$, let $r=r(u_\Lambda)$ be the residual produced by the 
Galerkin solution $u_\Lambda$. Under Assumption \ref{ass:minimality}, the step
$$
\partial\Lambda := \text{\bf D\"ORFLER}(r, \theta)
$$
selects a set {$\partial\Lambda$}
of minimal cardinality in $\Lambda^c$ for which $\Vert r-P_{\partial\Lambda}r \Vert
\leq \sqrt{1-\theta^2} \Vert r \Vert$. Thus, 
if $r$ belongs to a certain sparsity class ${\mathcal A}_{\bar{\phi}}$, identified by a function $\bar{\phi}$,
then (\ref{eq:nl.gen.2o}) yields
\begin{equation}\label{eq:boundDorfler} 
|\partial\Lambda| \leq {\bar{\phi}}^{-1}\left(\sqrt{1-\theta^2} \, 
\frac{\Vert r \Vert}{\ \Vert r \Vert_{{\mathcal A}_{\bar{\phi}}}}\right) +1 \;.
\end{equation}
 Explicitly, if $r \in {\mathcal A}_B^{\bar{s}}$ for some $\bar{s}>0$, we have by (\ref{eq:nlb.201})
$$
|\partial\Lambda| \leq (1-\theta^2)^{-d/{2\bar{s}}}
\left(\frac{\ \Vert r \Vert_{{\mathcal A}_B^{\bar{s}}}}{\Vert r \Vert}\right)^{d/\bar{s}} +1 \;,
$$
whereas if $r \in {\mathcal A}^{\bar{\eta},\bar{t}}_G$ for some $\bar{\eta}>0$ and $\bar{t}>0$, we have by (\ref{eq:nlg.301})
$$
|\partial\Lambda| \leq \frac{\omega_d}{\eta^{d/\bar{t}}} \left(
 \log \frac{\Vert r \Vert_{{\mathcal A}^{\bar{\eta},\bar{t}}_G}}{\Vert r \Vert} + |\log \sqrt{1-\theta^2}|
 \right)^{d/\bar{t}} + 1 \;.
$$
We stress the fact that the cardinality of $\partial\Lambda$ is related to the 
{\sl sparsity class of the residual}. {We will see in the rest
of this section that} such a class does coincide with the sparsity class of the solution in the algebraic case, whereas it is different
(indeed, worse)  in the exponential case. This is a crucial point to be kept in mind in the forthcoming optimality 
analysis of our algorithms. 

The cardinality of $\partial\Lambda$ depends indeed on how much the sparsity measure  
$\Vert r \Vert_{{\mathcal A}_{\bar{\phi}}}$ deviates from the {Hilbert} norm $\Vert r \Vert$. So, before embarking ourselves {on} the
study of the relationship between the sparsity classes of the residual and of the solution, we make some brief comments
on the ratio between these two quantities. For shortness, we only consider the exponential case, although similar
considerations apply to the algebraic case as well. The size of the ratio
$$
Q := \frac{\  \Vert r \Vert_{{\mathcal A}^{\bar{\eta},\bar{t}}_G}}{\Vert r \Vert} \; 
$$
depends on the relative behavior of the rearranged
coefficients $r_n^*$ of $r$, which by Definition \ref{def:elpicGev} and Proposition \ref{prop:nlg1}  satisfy
\begin{equation}\label{rearranged-res}
|r_n^*| \leq \lambda^* n^{(\bar{\tau}-1)/2} {\rm e}^{-\bar{\eta} \omega_d^{-\bar{\tau}} n^{\bar{\tau}}}
\Vert r \Vert_{{\mathcal A}^{\bar{\eta},\bar{t}}_G} 
\end{equation}
for some constant $\lambda^*>0$, with $\bar{\tau}=\bar{t}/d$. Let us consider two representative situations.

\begin{example}[{\em genuinely decaying functions}]\label{E:genuinely-decaying}
\rm
The most ``favorable'' situation is the one
in which the sequence of rearranged coefficients decays precisely at
the rate given by the right-hand side of {\eqref{rearranged-res}}; 
in other words, suppose that there exists a constant $\lambda_*>0$ such that for all $n \geq 1$
\begin{equation}\label{eq:spars20}
\lambda_* n^{(\bar{\tau}-1)/2} {\rm e}^{-\bar{\eta} \omega_d^{-\bar{\tau}} n^{\bar{\tau}}}
\Vert r \Vert_{{\mathcal A}^{\bar{\eta},\bar{t}}_G} \, \leq \,
|r_n^*| \, \leq \, \lambda^* n^{(\bar{\tau}-1)/2} {\rm e}^{-\bar{\eta} \omega_d^{-\bar{\tau}} n^{\bar{\tau}}}
\Vert r \Vert_{{\mathcal A}^{\bar{\eta},\bar{t}}_G} \;.
\end{equation}
Then,
$$
(\lambda_*)^2 \sum_{n \geq 1} n^{(\bar{\tau}-1)} {\rm e}^{-2 \bar{\eta} \omega_d^{-\bar{\tau}} n^{\bar{\tau}}}
\Vert r \Vert_{{\mathcal A}^{\bar{\eta},\bar{t}}_G}^2 \, \leq \, \Vert r \Vert^2 \, \leq \, 
(\lambda^*)^2 \sum_{n \geq 1} n^{(\bar{\tau}-1)} {\rm e}^{-2 \bar{\eta} \omega_d^{-\bar{\tau}} n^{\bar{\tau}}}
\Vert r \Vert_{{\mathcal A}^{\bar{\eta},\bar{t}}_G}^2 \;,
$$
and since
$$
\sum_{n \geq 1} n^{(\bar{\tau}-1)} {\rm e}^{-2 \bar{\eta} \omega_d^{-\bar{\tau}} n^{\bar{\tau}}}
\sim \int_{1}^{+\infty} x^{\bar{\tau}-1} {\rm e}^{-2 \bar{\eta} \omega_d^{-\bar{\tau}} x^{\bar{\tau}}} \, dx =
\int_{1}^{+\infty} {\rm e}^{-2 \bar{\eta} \omega_d^{-\bar{\tau}} y} \, dy =C \;,
$$
we obtain 
$$
\frac1{C \lambda^*} \leq Q \leq \frac1{C \lambda_*} \;.
$$
Thus, if (\ref{eq:spars20}) is a ``tight'' bound, the ratio $Q$ is ``small'', and the procedure {\bf D\"ORFLER} 
activates a moderate number of degrees of freedom at the current
iteration. \endproof
\end{example}

\begin{example}[{\em plateaux}]\label{E:plateaux}
\rm
The opposite situation, i.e., the worst scenario, occurs when the sequence of rearranged coefficients of $r$ exhibits
large ``plateaux'' consisting of equal (or nearly equal) elements in modulus. 
Fix an integer $K$ arbitrarily large, and suppose that the $K$ largest coefficients of $r$ satisfy
$$
|r^*_1|=|r^*_2|= \cdots =|r^*_{K-1}|=|r^*_{K}| 
= \lambda^* K^{(\bar{\tau}-1)/2} {\rm e}^{-\bar{\eta} \omega_d^{-\bar{\tau}} K^{\bar{\tau}}}
\Vert r \Vert_{{\mathcal A}^{\bar{\eta},\bar{t}}_G} \;.
$$
Since 
$$
\sum_{n > K} n^{(\bar{\tau}-1)} {\rm e}^{-2 \bar{\eta} \omega_d^{-\bar{\tau}} n^{\bar{\tau}}} \sim
\int_{(K+1)^{\bar{\tau}}}^{+\infty} {\rm e}^{-2 \bar{\eta} \omega_d^{-\bar{\tau}} y} \, dy 
= {\rm e}^{-2 \bar{\eta} \omega_d^{-\bar{\tau}}(K+1)^{\bar{\tau}} } < 
{\rm e}^{-2 \bar{\eta} \omega_d^{-\bar{\tau}} K^{\bar{\tau}} } \;,
$$
there exists $\delta \in (0,1)$ such that 
$$
\Vert r \Vert^2 = (\lambda^*)^2 (K+\delta)^{\bar{\tau}}{\rm e}^{-2 \bar{\eta} \omega_d^{-\bar{\tau}} K^{\bar{\tau}} }
\Vert r \Vert_{{\mathcal A}^{\bar{\eta},\bar{t}}_G}^2 \;.
$$
We conclude that the ratio
$$
Q=\frac{{\rm e}^{\bar{\eta} \omega_d^{-\bar{\tau}} K^{\bar{\tau}}}}{\lambda^* (K+\delta)^{\bar{\tau}/2}}
$$
turns out to be arbitrarily large, and indeed for such a residual it is easily seen that D\"orfler's condition 
$\Vert P_{\partial\Lambda}r \Vert \geq \theta \Vert r \Vert $ requires $|\partial\Lambda|$ 
to be of the order of $\theta K$. \endproof
\end{example}

Let us now investigate the sparsity classes of the residual, treating the algebraic and exponential cases separately.
Note that, in view of Propositions \ref{prop:nlg1-alg} or \ref{prop:nlg1}, for studying the sparsity classes of
certain functions $v$ and $Lv$ we are entitled to study, equivalently, 
the sparsity classes of the related 
vectors $\mathbf{v}$ and $\mathbf{A} \mathbf{v}$, where $\mathbf{A}$ is the stiffness matrix (\ref{eq:four100}).

\subsection{Algebraic case}\label{S:algebraic-case}
We first recall the notion of matrix compressibility 
(see \cite{CDDV:1998} where the concept has been used in the wavelet context).
\begin{definition}[{matrix compressibility}]\label{def:compress}
For $s^*>0$, a bounded matrix $\mathbf{A}:\ell^2(\mathbb{Z}^d)\to\ell^2(\mathbb{Z}^d)$ 
is called $s^*$-compressible if for any $j\in\mathbb{N}$ there exist constants 
$\alpha_j$ and $C_j$ and a matrix $\mathbf{A}_j$ having at most $\alpha_j 2^j$ non-zero entries per column, such that 
\[
\| \mathbf{A}-\mathbf{A}_j\| \leq C_j
\]
where $\{\alpha_j\}_{j\in\mathbb{N}}$ is summable, and for any $s<s^*$, 
$\{ C_j 2^{sj/d}\}$ is summable. 
\end{definition}
Concerning the compressibility of the matrices belonging to the 
class ${\mathcal D}_a(\eta_L)$ {of Definition \ref{def:class.matrix}}, 
the following result can be found in \cite[Lemma 3.6]{Dahlke-Fornasier-Groechenig:2010}.  We report here the proof for completeness.
\begin{lemma}[{compressibility}]\label{lm:compr}
{If
 $s^*:=\eta_L-d>0$, then any matrix
 $\mathbf{A}\in\mathcal{D}_a(\eta_L)$ 
is $s^*$-compressible.}
 \end{lemma} 
\begin{proof}
Let us take {$N_j=\lceil\frac{2^{j/d}}{(j+1)^2}\rceil$, where
  $\lceil\cdot\rceil$ denotes the integer part plus $1$}. Then  
by Property \ref{prop:matrix-estimate} (algebraic case) there holds 
$\|\mathbf{A}-\mathbf{A}_{N_j}\|\lesssim 2^{-j(\eta_L-d)/d}\ (j+1)^{2(\eta_L-d)}=:C_j$ and 
$\mathbf{A}_{N_j}$ has $\alpha_j 2^j$ {non-vanishing}
entries per column with {$\alpha_j \approx 2^d (j+1)^{-2d}$}.  
It is immediate to verify that $\sum_j \alpha_j<\infty$. Moreover, for 
$s<s^*$ and setting $\delta=s^*-s$, we clearly have 
$\sum_j C_j 2^{js/d}=\sum_j 2^{-j\delta/d} (j+1)^{2s^*} < \infty$.
\end{proof}

We now consider the continuity properties of the operator $L$ between sparsity spaces.
The following result is well known (see
e.g. \cite{Dahlke-Fornasier-Raasch:2007}) and {its} proof is here
reported for completeness.
\begin{proposition}[{continuity  of $L$ in $\mathcal{A}_B^s$}]\label{prop:A-continuity-alg}
Let $\mathbf{A}\in\mathcal{D}_a(\eta_L)$, $\eta_L>d$ and $s^*={\eta_L-d}$. 
For any $s<s^*$, if $v \in {\mathcal A}_B^s$ then  $Lv \in {\mathcal A}_B^s$, with
$$
\Vert Lv \Vert_{{\mathcal A}_B^s} \lsim \Vert v \Vert_{{\mathcal A}_B^s} \;.
$$
The constants appearing in the bounds go to infinity as $s$ approaches $s^*$.
\end{proposition}
\begin{proof}
Let us choose {$N_j=\lceil\frac{2^{j/d}}{(j+1)^2}\rceil$} as in
the proof of Lemma \ref{lm:compr}. {If we} set $\mathbf{A}_j:=\mathbf{A}_{N_j}$,  then 
by Property \ref{prop:matrix-estimate} (algebraic case) we have
$$
\Vert \mathbf{A}-\mathbf{A}_{j} \Vert  \lesssim 2^{-j(\eta_L-d)/d}\ (j+1)^{2(\eta_L-d)}=2^{-j s^*/d}\ (j+1)^{2s^*} . 
$$
On the other hand, for any $j \geq 0$, let ${\bf v}_j=P_j({\bf v})$ be a best $2^j$-term approximation of ${\bf v}\in\ell_B^s$, 
which therefore satisfies $\Vert {\bf v}-{\bf v}_j \Vert \leq 2^{-js/d}
\Vert {\bf v } \Vert_{\ell_B^s}$.
Note that the difference ${\bf v}_j - {\bf v}_{j-1}$ 
satisfies as well
$$
\Vert {\bf v}_j - {\bf v}_{j-1} \Vert \lsim  2^{-js/d}
\Vert {\bf v } \Vert_{\ell_B^s} \;. 
$$
Let 
$$
{\bf w}_J = \sum_{j=0}^J \mathbf{A}_{J-j}({\bf v}_j-{\bf v}_{j-1}) \;,
$$
where we set ${\bf v}_{-1}={\bf 0}$.
Writing ${\bf v}={\bf v}-{\bf v}_J + \sum_{j=0}^J ({\bf v}_j-{\bf v}_{j-1})$, we obtain 
$$
\mathbf{A}{\bf v}-{\bf w}_J = \mathbf{A}({\bf v}-{\bf v}_J)
+ \sum_{j=0}^J (\mathbf{A}-\mathbf{A}_{J-j})({\bf v}_j-{\bf v}_{j-1})\;.
$$ 
The last equation yields 
\begin{eqnarray}
\|\mathbf{A}{\bf v}-{\bf w}_J\| &\leq& \|\mathbf{A}\| \|{\bf v}-{\bf v}_J\|
+ \sum_{j=0}^J {\|\mathbf{A}-\mathbf{A}_{J-j}\|} \|{\bf v}_j-{\bf v}_{j-1}\|\nonumber\\
&\lesssim& \left(2^{-Js/d} + \sum_{j=0}^J  2^{-(J-j) s^*/d}\ (J-j+1)^{2s^*}  2^{-js/d}\right)
\|\bv\|_{\ell^s_B}\nonumber\\
&\lesssim& 2^{-Js/d}\left( 1 + \sum_{j=0}^J  2^{-(J-j) (s^*-s)/d}\ (J-j+1)^{2s^*} \right)
\|\bv\|_{\ell^s_B} \nonumber\\
 &\lesssim& 2^{-Js/d}\|\bv\|_{\ell^s_B},\nonumber
\end{eqnarray}
where the series $ \sum_{k} 2^{-k(s^*-s)/d}\ (k+1)^{2s^*}$ is
convergent but degenerates as $s$ approaches $s^*$. 
Finally, by construction {$\bw_J$} belongs to a finite dimensional space $V_{\Lambda_J}$, where
\[
\vert {\Lambda_J} \vert {\lsim \omega_d} \sum_{j=0}^J N_{J-j}^d\lesssim 2^J\sum_{j=0}^J (J-j+1)^{-2d}\lesssim  2^J\ .
\]
This implies $\| {\mathbf{A}}{\bf{v}}\|_{\ell_B^s} \lesssim \|{\bf{v}}\|_{\ell_B^s} $ for any $s<s^*$.
\end{proof}

At last, we  discuss the sparsity class of the residual
$r=r(u_\Lambda)$ for some Galerkin solution $u_\Lambda$.
\begin{proposition}[{sparsity class of the residual}]\label{prop:unif-bound-res-alg}
Let the assumptions of Property \ref{prop:inverse.alg} be satisfied, 
and set {$s^*=\eta_L-d$}. 
For any $s<s^*$, if $u \in {\mathcal A}_B^s$ then  $r(u_\Lambda) \in {\mathcal A}_B^s$ for any index set $\Lambda$, with
$$
\Vert r(u_\Lambda) \Vert_{{\mathcal A}_B^s} \lsim \Vert u \Vert_{{\mathcal A}_B^s} \;.
$$
\end{proposition}
\begin{proof}
Denoting by $ {\bf r}_\Lambda$ the vector representing
$r(u_\Lambda)$ and using Proposition \ref{prop:A-continuity-alg}, we get
\begin{equation}\label{eq:aux-1}
\| {\bf r}_\Lambda\|_{\ell_B^s} = \| \mathbf{A}({\bf u} - {\bf u}_\Lambda )\|_{\ell_B^s} 
\lesssim \| {\bf u} - {\bf u}_\Lambda \|_{\ell_B^s} 
\lesssim \| {\bf u} \|_{\ell_B^s} + \| {\bf u}_\Lambda \|_{\ell_B^s} .
 \end{equation}
At this point, we invoke the equivalent formulation of the Galerkin problem given by \eqref{eq:inf-pb-galerkin},
which yields $ \hat{{\bf u}} = ({\widehat{\mathbf{A}}_\Lambda})^{-1} (\mathbf{P}_\Lambda{\bf f})$. 
Using $\mathbf{A}\in\mathcal{D}_a(\eta_L)$ and combining 
Property \ref{prop:inf-matrix} together with Property \ref{prop:inverse.alg},
we obtain  $(\widehat{\mathbf{A}}_\Lambda)^{-1}\in \mathcal{D}_a(\eta_L)$. 
Hence, applying Proposition \ref{prop:A-continuity-alg} to $(\widehat{\mathbf{A}}_\Lambda)^{-1}$ we get 
\[
 \|{\bf u}_\Lambda\|_{\ell_B^s}  = \|\hat{{\bf u}} \|_{\ell_B^s} =
\|({\widehat{\mathbf{A}}_\Lambda})^{-1} (\mathbf{P}_\Lambda{\bf f})\|_{\ell_B^s} \lesssim \| \mathbf{P}_\Lambda{\bf f} \|_{\ell_B^s} \leq \| {\bf f} \|_{\ell_B^s} \;,
\]
where the last step is an easy consequence of the definition of the projector $ \mathbf{P}_\Lambda$. 
By substituting the above inequality  into \eqref{eq:aux-1}, we finally obtain
\begin{equation}
\| {\bf r}_\Lambda\|_{\ell_B^s} 
\lesssim \| {\bf u} \|_{\ell_B^s} + \| {\bf f} \|_{\ell_B^s} = \| {\bf u} \|_{\ell_B^s} + \| \mathbf{A}{\bf u} \|_{\ell_B^s} \lesssim  \| {\bf u} \|_{\ell_B^s} \;,
 \end{equation}
where in the last inequality we used again Proposition \ref{prop:A-continuity-alg}.
\end{proof}

We observe that the previous bound is tailored {to} the ``worst-scenario'':
one expects indeed that for $\Lambda$ large enough the residual becomes
progressively smaller than the solution.

\subsection{Exponential case}\label{subsec:spars-res-exp}

{As already alluded to in the Introduction, and in striking
  contrast to} the previous algebraic case,  
the implication $v \in {\mathcal A}^{{\eta},{t}}_G \Rightarrow Lv \in {\mathcal A}^{{\eta},{t}}_G$
is false. The following counter-examples prove this fact, and shed light on which could be the correct implication.

\begin{example}[{Banded matrices}]\label{ex:spars1}{\rm
Fix $d=1$ and $t=1$ (hence, $\tau { = \frac{t}{d}} =1$). Recalling the expression (\ref{eq:four140}) for the entries of $\mathbf{A}$,
let us choose $\hat{\nu}_0=\hat{\sigma}_0= {\sqrt{2 \pi}}$, which gives
$$
{a_{\ell,\ell} = 1 \qquad\forall \; \ell\in\mathbb{Z}.}
$$
Next, let us choose $\hat{\sigma}_h=0$ for all $h \not= 0$, which
implies {(because $d=1$)}
$$
|a_{\ell,k}|=\frac1{\sqrt{2\pi}}\, \frac{|\ell|\, |k|}{c_\ell \, c_k} |\hat{\nu}_{\ell-k}|\;, \qquad \ell \not= k \;,
$$
i.e., 
$$
\frac1{2\sqrt{2\pi}}\, |\hat{\nu}_{\ell-k}| \leq |a_{\ell,k}| \leq \frac1{\sqrt{2\pi}}\,|\hat{\nu}_{\ell-k}| \;, 
\qquad \ell \not= k \;, \ \ |\ell|, |k| \geq 1 \;.
$$
At this point, let us fix a real $\eta_L >0$ and an integer $p \geq 0$,
 and let us choose the coefficients $\hat{\nu}_h$ for $h \not= 0$ to satisfy
$$
|\hat{\nu}_h| =
\begin{cases}
\sqrt{2\pi} {\rm e}^{-\eta_L |h|}& \text{if } 0 < |h| \leq p \;, \\
0 & \text{if } |h| > p\;.
\end{cases}
$$ 
In summary, the coefficient $\nu$ of the elliptic operator $L$ is a trigonometric polynomial of degree $p$, whereas
the coefficient $\sigma$ is a constant. The corresponding stiffness matrix $\mathbf{A}$ is banded with $2p+1$ 
non-zero diagonals, and satisfies
\begin{equation}\label{eq:spars3}
\tfrac12  {\rm e}^{-\eta_L |\ell-k|} \leq |a_{\ell,k}| \leq {\rm e}^{-\eta_L |\ell-k|} \;, 
\qquad 0 \leq |\ell-k|\leq p \;, \ \ |\ell|, |k| \geq 1 \;.
\end{equation}

In order to define the vector $\mathbf{v}$, let us introduce the function $\iota \, : \, \mathbb{N}_* \to \mathbb{N}_*$,
{$\iota(n)=2(p+1)n$}. Let us fix a real $\eta>0$ and let us define the components $(\mathbf{v})_k = \hat{v}_k$
of the vector in such a way that
$$
|(\mathbf{v})_k|=
\begin{cases}
{\rm e}^{-\frac{\eta}2 n}& \text{if } k=\iota(n) \text{ for some } n \geq 1\;, \\
0 & \text{otherwise}\;.
\end{cases}
$$ 
Thus, the rearranged components $(\mathbf{v})_n^*$ satisfy $|(\mathbf{v})_n^*|={\rm e}^{-\frac{\eta}2 n}$, 
$n \geq 1$, {whence}
$\mathbf{v} \in \ell_G^{\eta,1}(\mathbb{Z})$ (or, equivalently, $v \in {\mathcal A}^{\eta,1}_G$), with
$\Vert \mathbf{v} \Vert_{\ell_G^{\eta,1}(\mathbb{Z})}=1$,
{according to Definition \ref{def:elpicGev}}.

The definition of the mapping $\iota$ and the banded structure of $\mathbf{A}$ imply that the only non-zero components
of $\mathbf{A}\mathbf{v}$ are those of indices $\iota(n)+q$ for some $n \geq 1$ and $q \in [-p,p]$. For these components
one has
$$
(\mathbf{A}\mathbf{v})_{\iota(n)+q}= a_{\iota(n)+q,\iota(n)} (\mathbf{v})_{\iota(n)} \;,
$$ 
thus, recalling (\ref{eq:spars3}), we easily obtain
\begin{equation}\label{eq:spars4}
\tfrac12 {\rm e}^{-\eta_L p} {\rm e}^{-\frac{\eta}2 n} \leq |(\mathbf{A}\mathbf{v})_{\iota(n)+q}| \leq
{\rm e}^{-\frac{\eta}2 n}\;, \qquad q \in [-p,p]\;.
\end{equation}
This shows that, for any integer $N \geq 1$,
$$
\# \{ \ell \, : \, |(\mathbf{A}\mathbf{v})_{\ell}| \geq \tfrac12 {\rm e}^{-\eta_L p} {\rm e}^{-\frac{\eta}2 N} \, \}
\geq (2p+1)N \;,
$$
hence
$$
|(\mathbf{A}\mathbf{v})^*_{(2p+1)N}| \, {\rm e}^{\frac{\eta}2 (2p+1)N} \geq \tfrac12 {\rm e}^{-\eta_L p}
{\rm e}^{\eta pN} \to +\infty \qquad \text{as } N \to +\infty \;,
$$
i.e., $\mathbf{A}\mathbf{v} \not \in \ell_G^{\eta,1}(\mathbb{Z})$ (or,
equivalently, $Lv \not\in {\mathcal A}^{\eta,1}_G$) {regardless
  of the relative values of $\eta_L$ and $\eta$.}

On the other hand, let ${m}_p$ be the smallest integer such that $\frac12 {\rm e}^{-\eta_L p} > 
{\rm e}^{-\frac{\eta}2 {m}_p}$. Given any $m \geq 1$, let $N \geq 1$ and $Q \in [-p,p]$ be such that
$(\mathbf{A}\mathbf{v})^*_m = (\mathbf{A}\mathbf{v})_{\iota(N)+Q}$, 
{which combined with (\ref{eq:spars4}) yields}
$$
{\rm e}^{-\frac{\eta}2 (N+{m}_p)} < |(\mathbf{A}\mathbf{v})^*_m| \leq {\rm e}^{-\frac{\eta}2 N} \;.
$$  
The {rightmost inequality in (\ref{eq:spars4}), 
namely $|(\bA\bv)_{\iota (N+m_p) +q}| \le e^{-\frac{\eta}{2}(N+m_p)}$},
 shows that there are at most $(2p+1)(N+{m}_p)$ components
of $\mathbf{A}\mathbf{v}$ that are larger than ${\rm e}^{-\frac{\eta}2
  (N+m_p)}$ in modulus. {This implies $m \leq  (2p+1)(N+{m}_p)$, whence}
$$
{\rm e}^{-\frac{\eta}2 N} \leq {\rm e}^{\frac{\eta}2 {m}_p} {\rm e}^{-\frac{\eta}{2(2p+1)} m} \;.
$$
Setting $\bar{\eta}=\frac{\eta}{2p+1}$, we conclude that $\mathbf{A}\mathbf{v} \in \ell_G^{\bar{\eta},1}(\mathbb{Z})$ 
(or, equivalently, $Lv \in {\mathcal A}^{\bar{\eta},1}_G$), with 
$$
\Vert \mathbf{A}\mathbf{v} \Vert_{\ell_G^{\bar{\eta},1}(\mathbb{Z})} \leq 
{\rm e}^{\frac{\eta}2 {m}_p} \Vert \mathbf{v} \Vert_{\ell_G^{\eta,1}(\mathbb{Z})} \;.
$$
{Therefore,
the sparsity class of $\bA\bv$ deteriorates  from $\ell^{\eta,1}_G(\mathbb{Z})$
for $\bv$ to $\ell^{\bar\eta,1}_G(\mathbb{Z})$ with
$\bar\eta=\frac{\eta}{2p+1}$.}
\endproof
}
\end{example}

Next counter-example shows that, when the stiffness matrix $\bA$ is not
banded, in order to have 
{$\bA\bv \in \ell^{\bar{\eta},\bar{t}}_G(\mathbb{Z})$} it is not enough to choose some
$\bar{\eta} < \eta$ as above, but a choice of $\bar{t} < t$ is mandatory.

\begin{example}[{Dense matrices}]\label{ex:spars2}{\rm
Let us {take again $d=t=1$ and}
modify the setting of the previous example, by assuming now that the coefficients $\hat{\nu}_h$
satisfy 
$$
|\hat{\nu}_h| = \sqrt{2\pi} {\rm e}^{-\eta_L |h|} \qquad \text{for all } |h|>0 \;,
$$
so that $\mathbf{A}$ is no longer banded, and its elements satisfy
\begin{equation}\label{eq:spars30}
\tfrac12  {\rm e}^{-\eta_L |\ell-k|} \leq |a_{\ell,k}| \leq {\rm e}^{-\eta_L |\ell-k|} 
\qquad \text{for all } |\ell|, |k| \geq 1 \;.
\end{equation}

{If $M>0$ is an arbitrary integer, we now construct 
a vector $\bv^M = \sum_{n \geq 1} \mathbf{v}^{M,n}$ with gaps of size $\lambda(M)\ge M$ 
between consecutive non-vanishing entries. To this end, we
introduce the function $\iota_M \, : \, \mathbb{N}_* \to \mathbb{N}_*$ defined as
$\iota_M(n):=\lambda(M)n$ and the vectors $\mathbf{v}^{M,n}$ with components}
$$
|(\mathbf{v}^{M,n})_k|={\rm e}^{-\frac{\eta}2 n} \delta_{k,\iota_M(n)} \;, \qquad k \in \mathbb{Z} \;.
$$
{From (\ref{eq:spars30}) and the fact that only the
  $\iota_M(n)$-th entry of $\bv^{M,n}$ does not vanish, we obtain
\begin{equation}\label{eq:spars31}
\tfrac12 {\rm e}^{-\eta_L |\ell - \iota_M(n)|} {\rm e}^{-\frac{\eta}2 n} 
\leq |(\mathbf{A}\mathbf{v}^{M,n})_\ell| \leq
{\rm e}^{-\eta_L |\ell - \iota_M(n)|} {\rm e}^{-\frac{\eta}2 n} \;.
\end{equation}
As in Example \ref{ex:spars1}, it is obvious that 
$\mathbf{v}^M \in \ell_G^{\eta,1}(\mathbb{Z})$ with 
$\Vert \mathbf{v}^M \Vert_{\ell_G^{\eta,1}(\mathbb{Z})}=1$. However, we will
prove below that 
$\|\bA\bv^M\|_{\ell^{\bar\eta,\bar t}_G} \lsim \|\bv^M\|_{\ell^{\eta,1}_G}$ 
cannot hold uniformly in $M$ for any
$\bar\eta>0$ and $\bar t>1/2$. 

We start by examining the cardinality $\#\mathcal{F}_n$ of the set
\[
\mathcal{F}_n:=\{ \ell\in\mathbb{Z} \, : \, |(\mathbf{A}\mathbf{v}^{M,n})_{\ell}|
> {\rm e}^{-\frac{\eta}2 M} \, \}
\]
In view of \eqref{eq:spars31}, the condition 
$|(\mathbf{A}\mathbf{v}^{M,n})_{\ell}| >  {\rm e}^{-\frac{\eta}2 M}$ 
is satisfied by those
$\ell=\iota_M(n)+m$ such that
\begin{equation*}
0 \leq |m| \leq \frac{\eta}{2\eta_L}(M-n) \; ,
\end{equation*}
whence $n\le M$ and $\#\mathcal{F}_n\ge
\frac{\eta}{\eta_L}(M-n)+1$. We now claim that
\begin{equation}\label{eq:C_M}
C_M := \# \{ \ell \, : \, |(\mathbf{A}\mathbf{v}^M)_{\ell}| 
\geq {\rm e}^{-\frac{\eta}2 M} \, \} \geq
\sum_{n = 1}^M
\# \mathcal{F}_n  \; ,
\end{equation}
whose proof we postpone. Assuming \eqref{eq:C_M} we see that
\[
C_M \geq \sum_{n=1}^{M}\left(\frac{\eta}{\eta_L}(M-n)+1
\right) \sim \frac{\eta}{2\eta_L}M^2 \; ,
\]
or equivalently there are about 
$N_M=\left\lceil\frac{\eta}{2\eta_L}M^2\right\rceil$ coefficients
of $\bv^M$ with values at least ${\rm e}^{-\frac{\eta}2 M}$. This
implies that the $N_M$-th rearranged coefficient of 
$\mathbf{A}\mathbf{v}^M$ satisfies
\[
 |(\mathbf{A}\mathbf{v}^{M})^*_{N_M}| \ge {\rm e}^{-\frac{\eta}{2}M}
\ge {\rm e}^{-\frac12 (2\eta_L\eta)^{1/2}N_M^{1/2}} \qquad
\text{for all } M \geq 1 \;.
\]
This proves that for any $\bar{\eta}>0$ and $\bar t> \frac12$, one has
$$
\Vert \mathbf{A}\mathbf{v}^M \Vert_{\ell_G^{\bar{\eta},\bar t}(\mathbb{Z})} \geq 
|(\mathbf{A}\mathbf{v}^{M})^*_{N_M}|\, {\rm e}^{\frac{\bar{\eta}}2
  N_M^{\bar t}} \geq
{\rm e}^{\frac{\bar{\eta}}2 N_M^{\bar t} - \frac12 (2\eta_L\eta)^{1/2}N_M^{1/2}} \to +\infty \qquad
\text{as  } M \to \infty \; ,
$$
whence the following bound cannot be valid}
$$
\Vert \mathbf{A}\mathbf{v} \Vert_{\ell_G^{\bar{\eta},t}(\mathbb{Z})} \lsim 
\Vert \mathbf{v} \Vert_{\ell_G^{\eta,1}(\mathbb{Z})} \;, 
\qquad \text{for all } \mathbf{v} \in \ell_G^{\eta,1}(\mathbb{Z}) \;.
$$

{It remains to prove \eqref{eq:C_M}. We first note that the sets
$\mathcal{F}_n$ are disjoint provided $\iota_M(n+1)-\iota_M(n) =
\lambda(M) \ge \frac{\eta}{\eta_L}M$. We next set
\[
\varepsilon_M:=\min_{1\le n \le M}\min_{\ell\in\mathcal{F}_n}   
|(\mathbf{A}\mathbf{v}^{M,n})_{\ell}| -{\rm e}^{-\frac{\eta}2 M} > 0
\]
which is a constant only dependent on $M$. We}
observe that for every $\ell\in\mathcal{F}_n$, there holds
\begin{eqnarray}
|(\mathbf{A}\mathbf{v}^M)_{\ell}| &\geq& |(\mathbf{A}\mathbf{v}^{M,n})_{\ell}| - 
| \sum_{p\not= n} (\mathbf{A}\mathbf{v}^{M,p})_{\ell} |\geq{\rm
  e}^{-\frac{\eta}2 M} + \varepsilon_M - \sum_{p\not= n} | (\mathbf{A}\mathbf{v}^{M,p})_{\ell} |. \label{example:aux1}
\end{eqnarray}
{We write $\ell\in\mathcal{F}_n$ as $\ell=\iota_M(n)+m$,
make use of \eqref{eq:spars31} and the definition of
$\iota_M(n)=\lambda(M)n$ to deduce
\[
\sum_{p\not= n} |(\mathbf{A}\mathbf{v}^{M,p})_{\ell} |\leq  
\sum_{p\not= n} e^{-\eta_L |\ell -\iota_M(p)|}e^{-\frac{\eta}{2}p}
\leq \sum_{p\not= n} e^{-\eta_L |m+\lambda(M)(n-p)|}\leq 
 \sum_{p\not= n} e^{-\eta_L (\lambda(M)|n-p| - |m|)}.
\]
Since $|m|\le \frac{\eta}{2\eta_L}M$, the above inequality gives}
\begin{equation}\label{example:aux2}
\sum_{p\not= n} |(\mathbf{A}\mathbf{v}^{M,p})_{\ell} |\leq 
2 e^{\eta_L |m|} \sum_{q\geq 1} e^{-\eta_L \lambda(M)q} \leq 
2e^{\frac{\eta}{2}M} \sum_{q\geq 1} e^{-\eta_L \lambda(M) q}\;. 
\end{equation}
Combining \eqref{example:aux1} and \eqref{example:aux2} yields 
$$
|(\mathbf{A}\mathbf{v}^M)_{\ell}|\geq {\rm e}^{-\frac{\eta}2 M} + \varepsilon_M - 
2e^{\frac{\eta}{2}M} \sum_{q\geq 1} e^{-\eta_L \lambda(M) q}\;.
$$
By choosing $\lambda(M)$ sufficiently large, the last term on the
right-hand side of the above inequality can be made arbitrarily small,
in particular {$\le \varepsilon_M$. We thus get 
$|(\mathbf{A}\mathbf{v}^M)_{\ell}|\geq {\rm e}^{-\frac{\eta}2 M}$ and
prove \eqref{eq:C_M}}.
\endproof
}
\end{example}

{Guided by Examples \ref{ex:spars1} and \ref{ex:spars2}}, 
we are ready to state the main
result of this section. {We define
\begin{equation}\label{aux-funct}
\zeta(t) := \Big( \frac{1+t}{\omega_d^{1+t}} \Big)^{\frac{t}{d(1+t)}}
\qquad\forall\; 0<t\le d.
\end{equation}
}

\begin{proposition}[{continuity of $L$ in $\mathcal{A}^{\eta,t}_G$}]\label{propos:spars-res}
Let the differential operator $L$ be such that the corresponding stiffness matrix satisfies 
$\mathbf{A} \in {\mathcal D}_e(\eta_L)$ for some constant $\eta_L>0$.
Assume that $v \in {\mathcal A}^{\eta,t}_G$ for some $\eta>0$ and $t \in (0,d]$. Let one of the two following
set of conditions be satisfied.
\begin{enumerate}[\rm (a)]
\item
{If the} matrix $\mathbf{A}$ is banded with $2p+1$ non-zero
diagonals, {let us set}
$$
\bar{\eta}= \frac{\eta}{(2p+1)^\tau} \;, \qquad \bar{t}= t \;.
$$
\item
{If the matrix $\mathbf{A}$ is dense}, but the coefficients $\eta_L$ and $\eta$ satisfy
the inequality $\eta< \eta_L  \omega_d^{\tau}$, {let us set} 
$$
\bar{\eta}= \zeta(t)\eta \;, \qquad \bar{t}= \frac{t}{1+t} \;.
$$
\end{enumerate}
Then, one has $Lv \in {\mathcal A}^{\bar{\eta},\bar{t}}_G$, with
\begin{equation}\label{eq:spars11bis}
\Vert Lv \Vert_{{\mathcal A}_G^{\bar{\eta},\bar{t}}} \lsim 
\Vert v \Vert_{{\mathcal A}_G^{\eta,t}} \;.
\end{equation}
\end{proposition}
\proof
We adapt to our situation the technique introduced in \cite{CDDV:1998}. Let $L_J$ ($J \geq 0$) be the differential
operator obtained by truncating the Fourier expansion of the coefficients of $L$ to the modes $k$ satisfying
$|k|\leq J$. Equivalently, $L_J$ is the operator whose stiffness matrix $\mathbf{A}_J$ is defined in 
(\ref{eq:trunc-matr}); thus, by Property \ref{prop:matrix-estimate} 
{(exponential case)} we have
$$
\Vert L-L_J \Vert = \Vert \mathbf{A}-\mathbf{A}_J \Vert \leq
C_{\mathbf{A}} {(J+1)}^{d-1}{\rm e}^{-\eta_L J} \;. 
$$
On the other hand, for any $j \geq 1$, let $v_j=P_j(v)$ be a best $j$-term approximation of $v$ (with $v_{0}=0$), 
which therefore satisfies $\Vert v-v_j \Vert \leq {\rm e}^{-\eta \omega_d^{-\tau} j^\tau} 
\Vert v \Vert_{{\mathcal A}^{{\eta},{t}}_G}$, with $\tau=t/d$. Note that the difference $v_j -v_{j-1}$ 
consists of a single {Fourier} mode and satisfies as well
$$
\Vert v_j-v_{j-1} \Vert \lsim {\rm e}^{-\eta \omega_d^{-\tau} j^\tau} 
\Vert v \Vert_{{\mathcal A}^{{\eta},{t}}_G} \;. 
$$
Finally, let us introduce the function $\chi \, : \, \mathbb{N} \to \mathbb{N}$ defined 
as $\chi(j)=\lceil j^\tau \rceil$, the smallest integer larger than or equal to $j^\tau$.

For any $J \geq 1$, let $w_J$ be the approximation of $Lv$ defined as
$$
w_J = \sum_{j=1}^J L_{\chi(J-j)}(v_j-v_{j-1}) \;.
$$
Writing $v=v-v_J + \sum_{j=1}^J (v_j-v_{j-1})$, we obtain 
$$
Lv-w_J = L(v-v_J)
+ \sum_{j=1}^J (L-L_{\chi(J-j)})(v_j-v_{j-1})\;.
$$ 
{We now assume to be in Case (b). Since
  $L:\ell^2(\mathbb{Z}^d)\to\ell^2(\mathbb{Z}^d)$ is continuous}, the last equation yields
\begin{equation}\label{eq:spars11}
\Vert Lv-w_J \Vert \lsim \left( {\rm e}^{-\eta \omega_d^{-\tau} J^\tau} 
+ \sum_{j=1}^J {\big(\lceil (J-j)^\tau \rceil + 1\big)^{d-1}}
 {\rm e}^{-(\eta_L \lceil (J-j)^\tau \rceil +\eta \omega_d^{-\tau} j^\tau )}
\right)\Vert v \Vert_{{\mathcal A}^{{\eta},{t}}_G} \;.
\end{equation}
The exponents of the addends can be bounded from below as follows
{because $\tau\le 1$}
\begin{eqnarray*}
\eta_L \lceil (J-j)^\tau \rceil +\eta \omega_d^{-\tau} j^\tau &=&
\eta_L \lceil (J-j)^\tau \rceil - \eta \omega_d^{-\tau} (J-j)^\tau + \eta \omega_d^{-\tau}( (J-j)^\tau + j^\tau) \\
&\geq& \eta_L (J-j)^\tau  - \eta \omega_d^{-\tau} (J-j)^\tau + \eta \omega_d^{-\tau}((J-j) + j)^\tau \\
&=& \beta (J-j)^\tau + \eta \omega_d^{-\tau} J^\tau \;, 
\end{eqnarray*}
with $\beta= \eta_L -\eta \omega_d^{-\tau} >0$ by assumption. Then,
(\ref{eq:spars11}) yields
\begin{equation}\label{eq:spars12}
\Vert Lv-w_J \Vert \lsim \left( 1 + \sum_{j=0}^{J-1} 
{\big(\lceil j^\tau\rceil+1\big)^{d-1}} {\rm e}^{-\beta j^\tau} \right) 
{\rm e}^{-\eta \omega_d^{-\tau} J^\tau} \Vert v \Vert_{{\mathcal A}^{{\eta},{t}}_G} \lsim
{\rm e}^{-\eta \omega_d^{-\tau} J^\tau} \Vert v \Vert_{{\mathcal A}^{{\eta},{t}}_G} \;.
\end{equation}

On the other hand, by construction $w_J$ belongs to a finite
dimensional space $V_{\Lambda_J}$, where
{
\begin{equation}\label{eq:spars13}
|\Lambda_J| \leq \omega_d\sum_{j=1}^J \chi(J-j)^d = 
\omega_d \sum_{j=0}^{J-1} \lceil j^\tau \rceil^d 
\sim \frac{\omega_d}{1+t} J^{1+t} \qquad \text{as } J \to \infty \;. 
\end{equation}
This implies
$$
\Vert Lv-w_J \Vert \lsim {\rm e}^{-\bar{\eta} \omega_d^{-\bar\tau} |\Lambda_J|^{\bar\tau}}
\Vert v \Vert_{{\mathcal A}^{{\eta},{t}}_G} \;,
$$
with $\bar\tau = \frac{\tau}{1+d\tau} = \frac{t}{d(1+t)}$ and 
$\bar\eta = \left(\frac{1+d\tau}{\omega_d^{1+d\tau}}\right)^{\bar\tau}\eta=\zeta(t)\eta$
as asserted.} 

{We last consider Case (a)}. One has $L_{\chi(J-j)}=L$ if
$\chi(J-j) \geq p$, {whence if 
$j \leq J-p^{1/\tau}$, then the} summation in (\ref{eq:spars11}) can be limited to those $j$
satisfying $j_p \leq j \leq J$, where $j_p= \lceil J- p^{1/\tau} \rceil$. Therefore 
$$
\Vert Lv-w_J \Vert \lsim \left( {\rm e}^{-\eta \omega_d^{-\tau} J^\tau} 
+ \max_{j_p\leq j \leq J} \lceil {(J-j)}^\tau\rceil^{d-1} 
\sum_{j=j_p}^J {\rm e}^{- \eta \omega_d^{-\tau} j^\tau }
\right)\Vert v \Vert_{{\mathcal A}^{{\eta},{t}}_G} \;.
$$
Now, $J-j \leq p^{1/\tau}$ if $j_p\leq j \leq J$ and $j^\tau \geq j_p^\tau \geq (J-p^{1/\tau})^\tau \geq J^\tau - p$, whence
$$
\Vert Lv-w_J \Vert \lsim \left(1+ p^{d-1+1/\tau}{\rm e}^{\eta \omega_d^{-\tau} p}
\right) {\rm e}^{-\eta \omega_d^{-\tau} J^\tau} \Vert v \Vert_{{\mathcal A}^{{\eta},{t}}_G} \;.
$$
We conclude by observing that $|\Lambda_J|\leq (2p+1)J$, since any matrix $\mathbf{A}_J$ has 
at most $2p+1$ diagonals.  \endproof

\bigskip

Finally, we discuss the sparsity class of the residual
$r=r(u_\Lambda)$ for any Galerkin solution $u_\Lambda$.
\begin{proposition}[{sparsity class of the residual}]\label{prop:unif-bound-res-exp}
{Let $\mathbf{A}\in\mathcal{D}_e(\eta_L)$ and
$\bA^{-1} \in\mathcal{D}_e(\bar\eta_L)$, for constants $\eta_L>0$
and $\bar\eta_L\in(0,\eta_L]$ according to Property
\ref{prop:inverse.matrix-estimate}{, and let}
$1\leq d \leq 10$.
If $u \in {\mathcal A}^{\eta,t}_G$ for some $\eta>0$ and $t \in
(0,d]$, such that {$\eta < \omega_d^{t/d(1+2t)}{\bar\eta_L}$}, then
there exist suitable positive constants $\bar{\eta} \leq \eta$ and 
$\bar{t} \leq t$ such that
$r(u_\Lambda) \in {\mathcal A}_G^{\bar{\eta},\bar{t}}$ for any index
set $\Lambda$, with
}
$$
\Vert r(u_\Lambda) \Vert_{{\mathcal A}_G^{\bar{\eta},\bar{t}}} \lsim
\Vert u \Vert_{{\mathcal A}^{\eta,t}_G} \;.
$$
\end{proposition}
\begin{proof}
We first remark that the hypothesis $1\leq d \leq 10$ guarantees $\omega_d\geq2$ 
(see e.g. \cite[Corollary 2.55]{Folland:real-analysis}); this implies $r < \omega_d^r$ for any $r>0$, 
{whence} the function $\zeta$ introduced in (\ref{aux-funct}) satisfies $\zeta(t)<1$ for any $t>0$.
Assume for the moment we are given $\bar{\eta}$ and $\bar{t}$. 
By using Proposition \ref{propos:spars-res} and Lemma \ref{L:nlg2}, we get
\begin{equation}\label{eq:aux-2}
\| {\bf r}_\Lambda\|_{\ell_G^{\bar{\eta},\bar{t}}} = \| \mathbf{A}({\bf u} - {\bf u}_\Lambda )\|_{\ell_G^{\bar{\eta},\bar{t}}} 
\lesssim \| {\bf u} - {\bf u}_\Lambda \|_{\ell_G^{{\eta_1},{t_1}}} 
\lesssim {\| {\bf u} \|_{\ell_G^{2^{\tau_1}{\eta_1},{t_1}}} + 
\| {\bf u}_\Lambda \|_{\ell_G^{2^{\tau_1}{\eta_1},{t_1}}}},
 \end{equation}
 where, $\bar{\tau}=\bar{t}/d,\tau_1=t_1/d$ and
 the following relations hold
 \[
 \bar\eta=\zeta(t_1)\eta_1 , \qquad \bar t =\frac{t_1}{1+t_1}<t_1 \ .
 \]
 From \eqref{eq:inf-pb-galerkin} we have $ {\bf u}_\Lambda = ({\widehat{\mathbf{A}}_\Lambda})^{-1} (\mathbf{P}_\Lambda{\bf f})$. 
Using Property \ref{prop:inf-matrix} and applying Proposition \ref{propos:spars-res} to $(\widehat{\mathbf{A}}_\Lambda)^{-1}$ we get 
\[
\|{\bf u}_\Lambda\|_{\ell_G^{2^{\tau_1}{\eta_1},{t_1}}}   = 
\|({\widehat{\mathbf{A}}_\Lambda})^{-1} (\mathbf{P}_\Lambda{\bf f})\|_{\ell_G^{2^{\tau_1}{\eta_1},{t_1}}}  \lesssim 
\| \mathbf{P}_\Lambda{\bf f} \|_{\ell_G^{{\eta_2},{t_2}}}  \leq \| {\bf f} \|_{\ell_G^{{\eta_2},{t_2}}}  \;, 
\]
with 
\[
 2^{\tau_1}\eta_1={\zeta(t_2)\eta_2 < \eta_2} \ , \qquad 
 {t_1}=\frac{t_2}{1+t_2}<t_2 \ .
 \]
By substituting the above inequality  into \eqref{eq:aux-2} and  using again Proposition \ref{propos:spars-res} we get 
\begin{equation}
\| {\bf r}_\Lambda\|_{\ell_G^{\bar{\eta},\bar{t}}} 
\lesssim 
\| {\bf u} \|_{\ell_G^{2^{\tau_1}{\eta_1},{t_1}}} +
\| {\bf f} \|_{\ell_G^{{\eta_2},{t_2}}}  = 
\| {\bf u} \|_{\ell_G^{2^{\tau_1}{\eta_1},{t_1}}} +
\| \mathbf{A}{\bf u} \|_{\ell_G^{{\eta_2},{t_2}}}
\lesssim  
\| {\bf u} \|_{\ell_G^{\eta,t}}
\end{equation}
where
\[
\eta_2={\zeta(t)\eta  <\eta}\ , \qquad 
 {t_2}=\frac{t}{1+t}<t \ .
 \]
This shows that the thesis holds true for the choice
\[
 \bar{\eta}=\Big(\frac12\Big)^{\frac{t}{d(1+2t)}}
 \zeta \Big(\frac{t}{1+2t}\Big)
 \zeta \Big(\frac{t}{1+t}\Big)
 \zeta (t)
 \eta,
\qquad
\bar t = \frac{t}{1+3t}.
\]
It remains to verify the assumptions of Proposition
\ref{propos:spars-res} when $\bA$ is dense. Since $\omega_d\geq 2$ and  
\[
t_1 = \frac{t}{1+2t} < t_2 = \frac{t}{1+t} < t,
\]
we have {$\omega_d^{\tau_1}<\omega_d^{\tau_2}<\omega_d^{\tau}$}.
Moreover, using $\eta_1<2^{\tau_1}\eta_1<\eta_2<\eta$ 
{ and $\eta_L \ge \bar\eta_L > \omega_d^{-\tau_1}\eta$ yields
%
%
\[
\eta<\omega_d^\tau \eta_L, 
\qquad
\eta_1 < \omega_d^{\tau_1}\eta_L,
\qquad
\eta_2 < \omega_d^{\tau_2}\bar\eta_L,
\]
which are the required conditions to apply Proposition 
\ref{propos:spars-res} when $\bA$ is dense. This concludes the proof.}
\end{proof}

{
\begin{remark}[{definition of $\omega_d$}]{\rm
The limitation $1\leq d \leq 10$ stems from the fact that the measure
of the unit {Euclidean ball $\omega_d$} in $\mathbb{R}^d$
monotonically decreases to $0$ as $d \to \infty$. To avoid such a restriction, one could modify
the definition of the Gevrey classes $G^{\eta,t}_p(\Omega)$ given in (\ref{eq:gevrey-class}), by replacing
the Euclidean norm $|k|=\Vert k \Vert_2$ appearing in the exponential by the maximum norm $\Vert k \Vert_\infty$. 
Consequently, throughout the rest of the paper $\omega_d$ would be replaced by the quantity $2^d$, strictly
larger than 1 for any $d$. \endproof
 }
\end{remark}
}

\section{Coarsening}\label{S:coarsening}

\newcommand{\wg}{{\ell^{\frac{\eta}2,t}_G({\mathbb{Z}}}}
\newcommand{\uu}{{\bf{u}}}
\newcommand{\vv}{{\bf{v}}}
\newcommand{\ww}{{\bf{w}}}
\newcommand{\zz}{{\bf{z}}}
\renewcommand{\aa}{{\bf{a}}}
\newcommand{\bb}{{\bf{b}}}
\newcommand{\cc}{{\bf{c}}}

We start by considering an example that sheds light on the role of
coarsening for the exponential case. We then state and prove a
seemingly new coarsening result, which is valid for both classes.

\subsection{Example of coarsening}\label{S:example-coarse}
Let $\aa,\bb\in\mathbb{R}^p$ for $p\ge1$ be the vectors
\[
\aa := (1,0,\cdots,0),
\quad
\bb := \frac{1}{p} (1,1,\cdots,1). 
\]
Let $\vv,\zz$ be the sequences defined by
\[
\vv := \big( {\rm e}^{-\eta k} \aa \big)_{k=0}^\infty,
\quad
\zz := \big( {\rm e}^{-\eta k} \bb \big)_{k=0}^\infty.
\]
We first observe that
\[
\|\vv\|^2 = {p\|\zz\|^2} = \frac{1}{1-{\rm e}^{2\eta}}\;,
\qquad
\Vert \vv \Vert_{\ell^{2\eta,1}_G(\mathbb{Z})} = 
{p \Vert \zz \Vert_{\ell^{2\eta/p,1}_G(\mathbb{Z})} = 1}
\]
(recall that $\omega_d=2$ for $d=1$). Given a parameter $\vare<1$,
we now construct a perturbation $\ww$ of $\vv$ which is much less
sparse {than $\bv$} by simply scaling $\zz$ and adding it to
$\vv$ {(see Fig. \ref{fig:coarsening} (a))}:
\[
\ww := \vv + \vare \zz = \big({\rm e}^{-\eta k}
(\aa+\vare\bb)\big)_{k=1}^\infty\; .
\]
%

\begin{figure}[!htbp]
\centering
\subfigure[\label{fig:1L}]{
\includegraphics[width=0.45\textwidth]{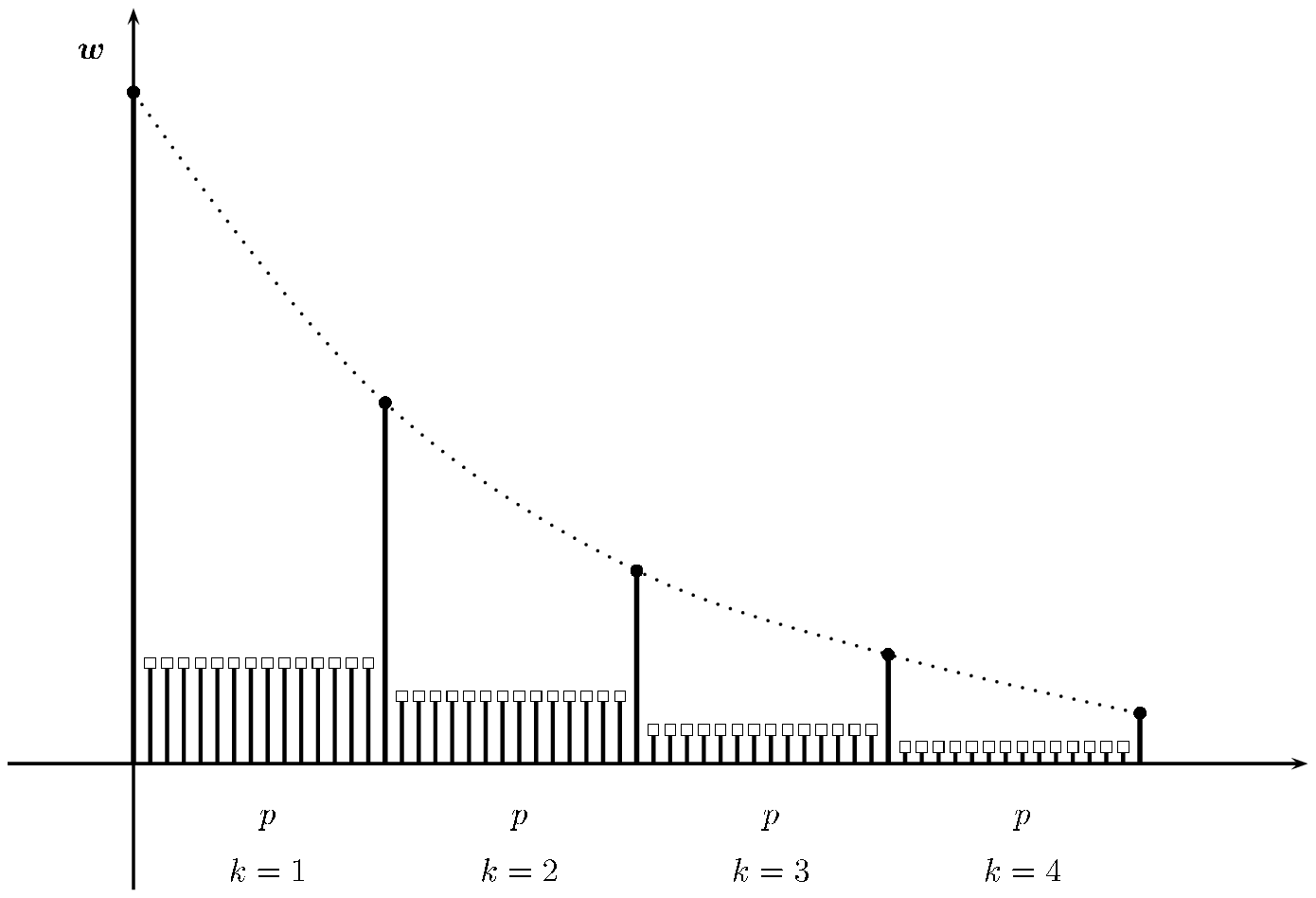}
}
\hspace{0.1cm}
\subfigure[\label{fig:1R}]{
\includegraphics[width=0.45\textwidth]{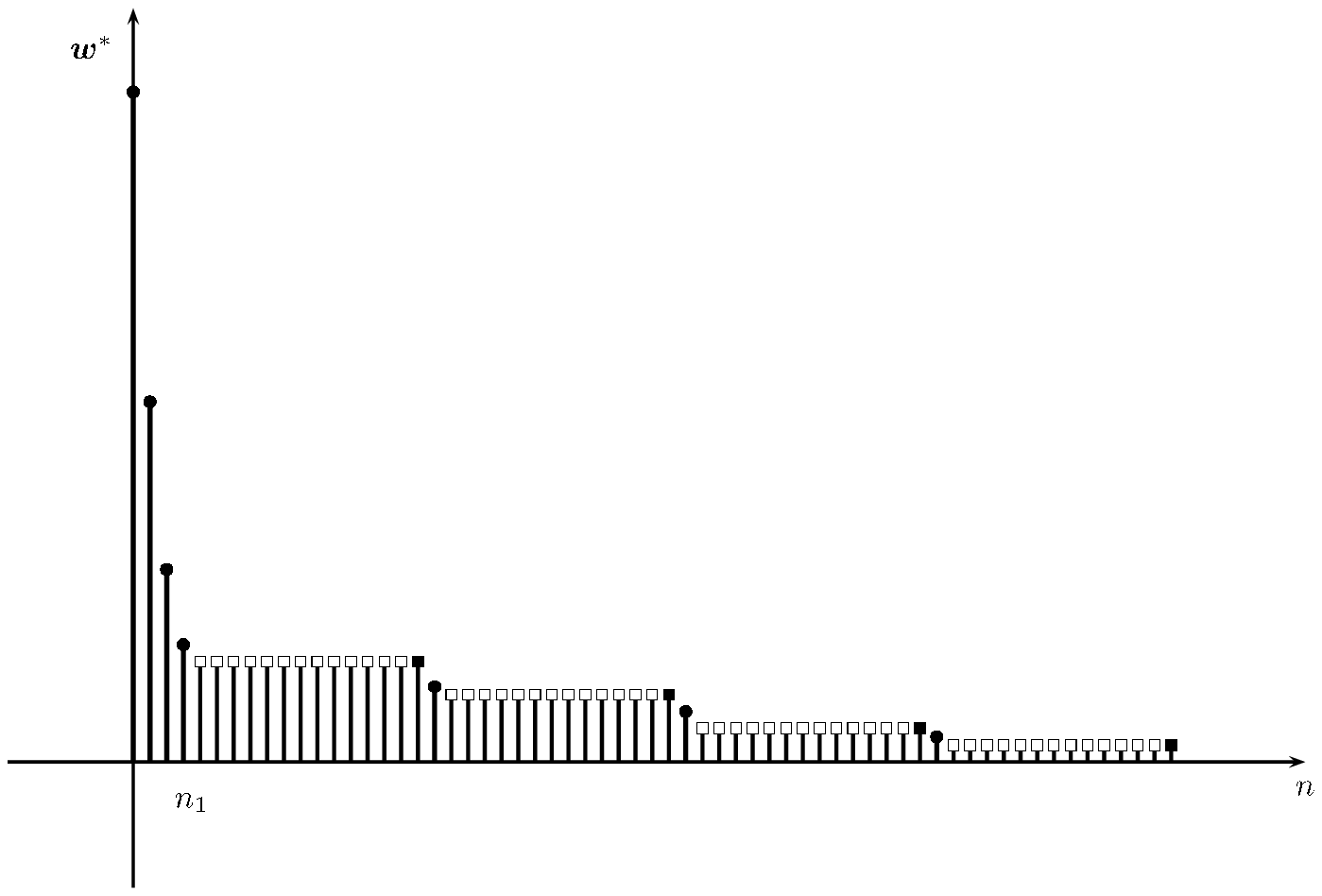}
}
\caption{Pictorial representation of (a) the components of the vector 
${\ww=\vv+\vare\zz}$
and (b) its rearrangement ${\ww}^*$. {It turns out that $\ww^*$ exhibits
the decay rate ${\rm e}^{-k\eta}$ 
of $\vv$ up to a level of accuracy $\|\ww-\vv\|$ in
$\ell^2(\mathbb{Z})$ but a worse decay rate 
${\rm e}^{-k\frac{\eta}{p}}$ of $\zz$ for smaller
tolerances. Therefore, truncating $\ww^*$ with a threshold
$\delta\ge\|\ww-\vv\|$ captures the behavior of $\vv$.}
}
\label{fig:coarsening} 
\end{figure}

The first task is to compute the norms of $\ww$. We obviously have
$\|\ww\| \simeq  \|\vv\|$. To determine the weak quasi-norm of $\ww$
we need to find the rearrangement $\ww^*$ (see Fig. \ref{fig:coarsening} (b)). 
Let $n_1$ be the smallest integer such
that
\[
\Big(1+\frac{\vare}{p}\Big) {\rm e}^{-\eta n_1} \ge \frac{\vare}{p}
{\rm e}^{-\eta} > \Big(1+\frac{\vare}{p}\Big) {\rm e}^{-\eta (n_1+1)}\;,
\]
namely the index corresponding to the first crossing of the
exponential curve ${\rm e}^{-\eta n}$ dictating the behavior of the first portion of the
rearranged sequence $\ww^*$ (which coincides with the behavior of $\vv^*$), 
and the first plateaux of $\zz$. This implies
\[
\frac{1}{\eta} \log\Big(1+\frac{p}{\vare}\Big)
< n_1 
\le 1+ \frac{1}{\eta} \log\Big(1+\frac{p}{\vare}\Big)\;.
\]
Next, let $n_2$ be the smallest integer such that
\[
\Big(1+\frac{\vare}{p}\Big) {\rm e}^{-\eta n_2} \ge \frac{\vare}{p}
{\rm e}^{-2\eta} > \Big(1+\frac{\vare}{p}\Big) {\rm e}^{-\eta (n_2+1)}\;,
\]
which corresponds to the beginning of a number of decreasing
exponentials preceeding the second plateaux of $\ww^*$. This implies
\[
1+ \frac{1}{\eta} \log\Big(1+\frac{p}{\vare}\Big)
< n_2 
\le 2+ \frac{1}{\eta} \log\Big(1+\frac{p}{\vare}\Big)
\]
and shows that $n_2-n_1=1$, and that there is exactly one exponential between
the first and second plateaux. Iterating this argument, we see that the difference between two
consecutive {$n_j$'s} is just $1$, and that there is exactly one exponential between
two consecutive plateaux {(see Fig \ref{fig:coarsening} (b))}.

We are now ready to compute the weak quasi-norm of $\ww$. Let $\nu_k$
denote the index corresponding to the end of the $k$-th plateaux
{of $\ww$},
which in turn corresponds to the value $w^*_{\nu_k} =
{\rm e}^{-\eta k}$. Then
\[
\nu_k = pk + n_1 \sim pk + \frac{1}{\eta} \log\Big(1+\frac{p}{\vare}\Big) \;.
\]
To determine the class of $\ww$, we seek $\lambda$ so that
$\ww\in \ell^{\lambda,1}_G(\mathbb{Z})$, namely
\[
\sup_{k\geq 0} \Big({\rm e}^{\lambda \nu_k /2} {\rm e}^{-\eta k}\Big) < \infty
\quad
\Leftrightarrow
\quad
\frac12 \lambda pk - \eta k \le 0
\quad
\Leftrightarrow
\quad
\lambda\le \frac{2\eta}{p} \;.
\]
We thus realize that $\ww\in\ell^{2\eta/p,1}_G(\mathbb{Z})$ belongs to a
sparsity class much worse than that of $\vv$, that deteriorates as
the size $p$ of the plateaux tends to $\infty$. On the other hand,
we note that {the restrictions
  $\ww^*_{|[1,n_1]}=\vv^*_{|[1,n_1]}$ coincide, thereby} showing that the decay rate of
the first part of $\ww^*$ is the same as that of $\vv^*$ 
{(see Fig \ref{fig:coarsening}(b))}.
This example explains the need to coarsen the vector $\ww$ starting at latest 
at $n_1$, to eliminate the tail of $\ww^*$
which decays with rate $2\eta/p$ instead of the optimal rate $2\eta$ of $\vv$.

{In addition, we} observe that the best $n_1$-term approximation of $\ww$ satisfies
\[
\|\ww-\ww_{n_1}\|^2 = \sum_{k=0}^\infty p
\frac{\vare^2}{p^2} {\rm e}^{-2k\eta} = \frac{\vare^2}{p}
\frac{1}{1-{\rm e}^{-2\eta}} = { \|\vv-\ww\|^2 = \vare^2
  \|\zz\|^2}\; ,
\]
{which is precisely the size of the perturbation error of $\vv$.
Given an error tolerance $\delta\ge\vare\|\zz\|$, the best $N$-term 
approximation $\ww_N$ of $\ww$ satisfying $\|\ww-\ww_N\|\le\delta$ 
would require
$$
 N \sim \frac1\eta \log\frac{1}{\delta} = 
\frac2{2\eta} \log\frac{\Vert \vv \Vert_{\ell^{2\eta,1}_G(\mathbb{Z})}}{\delta}\;.
$$
}

\subsection{New coarsening Result}\label{S:new-coarsening}
%
We extract the following lesson from the example of Sect.
\ref{S:example-coarse}: for
as long as we deal with the first part of {$\ww^*$}, which has a decay
rate ${\rm e}^{-k\eta}$ dictated by that of {$\vv^*$}, we could coarsen $\ww$
and obtain an approximation of both $\ww$ and $\vv$ with {the} decay
rate ${\rm e}^{-k\eta}$ of $\vv$. This requires limiting the accuracy to
size {$\|\vv-\ww\|$} since a smaller accuracy utilizes the tail of $\ww$ which
has a slower decay ${\rm e}^{-k\frac{\eta}{p}}$. 

We express this heuristics in the following theorem, which goes back 
to Cohen, Dahmen, and DeVore {\cite{CDDV:1998}}.
However, our proof is much more elementary and the statement much more precise.
Although the result holds for the general setting of Sect. \ref{sec:abstract-nl}, we just present it for the
exponential case, since it will be used only in this situation. 

\begin{theorem}[coarsening]\label{T:coarsening}
Let $\vare>0$ and let $v \in {\mathcal A}^{\eta,t}_G$ and $w \in V$ be so that
\[
\|v-w \| \le \vare.
\]
Let $N=N(\vare)$ be the smallest integer such that the best $N$-term
approximation $w_N$ of $w$ satisfies
\[
\|w - w_N\| \le 2\vare.
\]
Then, {$\|v-w_N\|\le3\vare$ and}
\[
N \le \frac{\omega_d}{\ \eta^{d/t}} 
\left(\log\frac{\Vert v \Vert_{{\mathcal A}^{\eta,t}_G}}{\vare}\right)^{d/t}\!\!\!+1 \;.
\]
\end{theorem}
\begin{proof}
Let $\Lambda_\vare$ be the set of indices corresponding to the best
approximation of {$v$} with accuracy $\vare$. So $\Lambda_\vare$ is a
minimal set with properties
\[
\|v - P_{\Lambda_\vare} v\|\le \vare,
\qquad
|\Lambda_\vare| \le \frac{\omega_d}{\ \eta^{d/t}} 
\left(\log\frac{\Vert v \Vert_{{\mathcal A}^{\eta,t}_G}}{\vare}\right)^{d/t}\!\!\! +1 \;.
\]
If $z=w-v$, then
\begin{align*}
\|w - P_{\Lambda_\vare}w \| & \le \|(v+z) - P_{\Lambda_\vare}(v+z)\|
= \|(v- P_{\Lambda_\vare} v) + (z-P_{\Lambda_\vare} z )\|
\\
& \le \|v- P_{\Lambda_\vare} v\| + \|z-P_{\Lambda_\vare} z\|
 \le \vare + \|z\| \le 2\vare \;,
\end{align*}
because $I-P_{\Lambda_\vare}$ is the projector onto $V_{\mathbb{Z}^d \setminus \Lambda_\vare}$.
Since $N$ is the cardinality of the smallest set satisfying the above
relation, we deduce that $N \le |\Lambda_\vare|$. This concludes the proof.
\end{proof}

\section{Optimality properties of adaptive algorithms: algebraic case}\label{sec:complexity}

The rest of the paper will be devoted to investigating complexity issues for the sequence of approximations 
$u_n=u_{\Lambda_n}$ generated by any of the adaptive algorithms presented in Sect. \ref{sec:plain-adapt-alg}. 
In particular, we wish to estimate the cardinality
of each $\Lambda_n$ and check whether its growth is ``optimal'' with respect to the sparsity class 
${\mathcal A}_\phi$ of the exact solution, in the sense that $|\Lambda_n|$ is comparable to 
the cardinality of the index set of the best approximation of $u$ yielding the same error
$\Vert u - u_n \Vert$.  

The algebraic case will be dealt with in the present section, whereas the exponential case will be analyzed
in the next one. The two cases differ in that no coarsening is needed for optimality in the former case,
whereas we will prove optimality in the latter case only for the algorithms that incorporate a coarsening step.
The reason of such a difference can be {attributed}, on the one hand, to the slower growth of the activated degrees
of freedom in the exponential case as opposed to the algebraic case and, on the other hand, to the discrepancy 
in the sparsity classes of the residuals and the solution in the
exponential case, {discussed in Sect. \ref{subsec:spars-res-exp}.}

\subsection{ADFOUR with moderate D\"orfler marking}\label{subsec:moderate-adfour-alg}

The approach followed in the sequel, which has been proposed in \cite{Gantumur-Stevenson:2007} 
in the wavelet framework and adopted in \cite{Stevenson:2007, Nochetto-et-al:2008} 
in the finite-element framework, allows us to prove the optimality of the algorithm in the algebraic case, provided
D\"orfler marking is not too aggressive.

The two following  lemmas will be useful in the subsequent analysis.

\begin{lemma}[{localized a posteriori upper bound}]\label{lem:optim1}
Let $\Lambda \subset \Lambda_* \subset \mathbb{Z}^d$ be nonempty subsets of indices. Let $u_\Lambda \in V_\Lambda$
and $u_{\Lambda_*} \in V_{\Lambda_*}$ be the Galerkin approximations of Problem (\ref{eq:four.1}). Then
$$
\tvert u_{\Lambda_*} - u_{\Lambda} \tvert^2 \leq \frac1{\alpha_*} 
\sum_{k \in \Lambda_* \setminus \Lambda} |\hat{R}_k(u_\Lambda)|^2 
= \frac1{\alpha_*} \eta^2(u_\Lambda, \Lambda_*) \;.
$$
\end{lemma}
\begin{proof}
One has
$$
\tvert u_{\Lambda_*} - u_{\Lambda} \tvert^2 =a( u_{\Lambda_*} - u_{\Lambda}, u_{\Lambda_*} - u_{\Lambda})
=(f, u_{\Lambda_*} - u_{\Lambda})-a(u_\Lambda, u_{\Lambda_*} - u_{\Lambda})=
\sum_{k \in \Lambda_*}\hat{r}_k(u_\Lambda)
{(\hat{u}_{\Lambda_*} - \hat{u}_{\Lambda})_k}
$$
{because $\Lambda_*$ is the support of $u_{\Lambda_*} - u_{\Lambda}$.}
The {asserted} result follows immediately by the Cauchy-Schwarz inequality, 
{upon} recalling that $\hat{r}_k(u_\Lambda)=0$
for all $k \in \Lambda$.
\end{proof}

\begin{lemma}[{D\"orfler property}]\label{lem:optim2}
Let $\Lambda \subset \Lambda_* \subset \mathbb{Z}^d$ be nonempty subsets of indices. Let $u_\Lambda \in V_\Lambda$
and $u_{\Lambda_*} \in V_{\Lambda_*}$ be the Galerkin approximations of Problem (\ref{eq:four.1}). 
{Let} the marking parameter $\theta$ satisfies $\theta \in (0,\theta_*)$, where
$\theta_*=\sqrt{\frac{\alpha_*}{\alpha^*}}$,
{and} set $\mu_\theta=1-\frac{\alpha^*}{\alpha_*}\theta^2>0$.  If 
$$
\tvert u- u_{\Lambda_*}  \tvert^2 \leq \mu \tvert u - u_{\Lambda} \tvert^2 \;,
$$
for some $\mu \in (0, \mu_\theta]$, then $\Lambda^*$ fulfils D\"orfer's condition, i.e.,
$$
\eta(u_{\Lambda}, \Lambda^*) \geq \theta \eta(u_{\Lambda}) \;.
$$
\end{lemma}
\proof
{Since $u - u_{\Lambda_*} \perp u_{\Lambda} - u_{\Lambda_*}$ in
  the energy norm because of Pythagoras, 
the assumption yields}
$$
\tvert u - u_{\Lambda} \tvert^2 = \tvert u - u_{\Lambda_*} \tvert^2
+ \tvert u_{\Lambda_*} - u_{\Lambda} \tvert^2 \leq \mu \tvert u - u_{\Lambda} \tvert^2 
+ \tvert u_{\Lambda_*} - u_{\Lambda} \tvert^2 \; .
$$
{Invoking the lower bound in (\ref{eq:four.2.3}) gives} 
$$
\tvert u_{\Lambda_*} - u_{\Lambda} \tvert^2 \geq (1-\mu) \tvert u - u_{\Lambda} \tvert^2 \geq
(1-\mu)\frac1{\alpha^*} \eta^2(u_{\Lambda}) \; ,
$$
{whence applying Lemma \ref{lem:optim1} implies}
$$
\eta^2(u_\Lambda, \Lambda_*) \geq (1-\mu)\frac{\alpha_*}{\alpha^*} \eta^2(u_{\Lambda}) \geq
(1-\mu_\theta)\frac{\alpha_*}{\alpha^*} \eta^2(u_{\Lambda}) = \theta^2 \eta^2(u_{\Lambda}).
$$
{This concludes the proof.}\endproof

We are ready to estimate the growth of degrees of freedom 
{generated by the algorithm {\bf ADFOUR} of}
Sect. \ref{sec:defADFOUR}. For the moment, we place ourselves in the abstract framework of 
Sect. \ref{sec:abstract-nl}, only the final result {being
  specifically for} the algebraic case.

\begin{proposition}[{cardinality of $\partial\Lambda_n$}]\label{prop:optim1}
Let $\theta$ satisfy the condition stated in Lemma \ref{lem:optim2}, and 
let $\mu \in (0,\mu_\theta]$ be fixed.
 Let $\{\Lambda_n, \, u_n \}_{n\geq 0}$ be the sequence generated by the adaptive algorithm 
{\bf ADFOUR}, {and set} $\varepsilon_n^2=\mu \tvert u - u_n \tvert^2$. If the solution $u$ belongs to 
the sparsity class ${\mathcal A}_\phi$, then
\begin{equation}
|\partial \Lambda_n|=|\Lambda_{n+1}|-|\Lambda_n| \leq  \kappa \, 
\phi^{-1}\left(\frac{\varepsilon_n}{\Vert u \Vert_{{\mathcal
        A}_\phi}}\right) \;, 
\qquad{\forall\ n\ge0\;,}
\end{equation}
{where $\kappa>1$ is the constant in \eqref{eq:nl.gen.2}.}
\end{proposition}
\begin{proof}
{Let $\varepsilon=\varepsilon_n$ and make use of
  (\ref{eq:nl.gen.2}) for $u\in{\mathcal A}_\phi$}:
 there exists $\Lambda_{\varepsilon}$ 
and $w_{\varepsilon} \in V_{\Lambda_{\varepsilon}}$ such that
$$
\tvert u - w_{\varepsilon} \tvert^2 \leq \varepsilon^2 \qquad \text{and} \qquad 
|\Lambda_{\varepsilon}| \leq  \kappa \, \phi^{-1}\left(\frac{\varepsilon}{\Vert u \Vert_{{\mathcal A}_\phi}}\right).
$$
Let $\Lambda_*=\Lambda_n \cup \Lambda_\epsilon$ be the overlay of the two index sets, and let 
$u_* \in V_{\Lambda_*}$ be the Galerkin approximation of Problem (\ref{eq:four.1}). Then, since 
$V_{\Lambda_\epsilon} \subseteq V_{\Lambda_*}$, we have
$$
\tvert u - u_* \tvert^2 \leq \tvert u - w_{\varepsilon} \tvert^2 \leq \mu \tvert u - u_n \tvert^2 \;.
$$
Thus, we are entitled to apply Lemma \ref{lem:optim2} to $\Lambda_n$ and $\Lambda_*$, yielding
$$
\eta(u_n, \Lambda^*) \geq \theta \eta(u_n) \;.
$$
By the minimality property of the cardinality of $\Lambda_{n+1}$ among all sets satisfying D\"orfler {property}
for $u_n$ (Assumption \ref{ass:minimality}), we deduce that
$|\Lambda_{n+1}|\leq |\Lambda_*| \leq |\Lambda_n| + |\Lambda_\epsilon|$, i.e.,
\begin{equation}\label{eq:optim11}
|\Lambda_{n+1}| - |\Lambda_n| \leq |\Lambda_\epsilon| \;,
\end{equation}
whence the result.
\end{proof}

\begin{corollary}[{cardinality of $\Lambda_n$: general case}]\label{cor:optim1}
{Let the assumptions of Proposition \ref{prop:optim1} be valid
  and $\rho = \sqrt{1-\frac{\alpha_*}{\alpha^*}\theta^2}$
be given by (\ref{eq:def_rhotheta}). Then}
\begin{equation}\label{eq:optim12}
|\Lambda_n| \leq \kappa \, \sum_{k=0}^{n-1} \phi^{-1}\left(\rho^{k-n} \, 
\frac{\varepsilon_n}{\Vert u \Vert_{{\mathcal A}_\phi}}\right)  \;, 
\qquad {\forall\ n\ge0 \;.}
\end{equation}
\end{corollary}
\begin{proof}
Recalling that $|\Lambda_0|=0$, the previous proposition yields
$$
|\Lambda_n| = \sum_{k=0}^{n-1} |\partial\Lambda_k|  \leq \kappa \, \sum_{k=0}^{n-1} \phi^{-1}\left( 
\frac{\varepsilon_k}{\Vert u \Vert_{{\mathcal A}_\phi}}\right)  \;. 
$$
On the other hand, by Theorem \ref{teo:four1} one has
\begin{equation}\label{eq:optim12bis}
{\varepsilon_n = \sqrt{\mu} \tvert u-u_n \tvert 
\leq \sqrt{\mu} \rho^{n-k} \tvert u-u_k \tvert
= \rho^{n-k} \varepsilon_k \qquad\forall\ 0\le k \le n-1 \; ,} 
\end{equation}
and we conclude recalling the monotonicity of $\phi$.
\end{proof}

At this point, we assume to be in the algebraic case, i.e. $u \in {\mathcal A}_B^s$ for some $s>0$. Then, 
(\ref{eq:optim12}) reads
$$
|\Lambda_n| \leq \kappa \, \mu^{-d/2s} 
\tvert u-u_n \tvert^{-d/s} \Vert u \Vert_{{\mathcal A}_B^s}^{d/s}
{ ~\sum_{k=0}^{n-1} \left(\rho^{d/s}\right)^{n-k}\;, 
\qquad\forall\ n\ge0 \;.}
$$ 
Summing-up the geometric series and using (\ref{eq:four.1bis}), we arrive at the following result.

\begin{theorem}[{cardinality of $\Lambda_n$: algebraic case}]\label{teo:optim2}
Under the assumptions of Proposition \ref{prop:optim1}, the growth of the active degrees of freedom
produced by {\bf ADFOUR} in the algebraic case is estimated as follows:
$$
{|\Lambda_n| \le C_* \, \Vert u-u_n \Vert^{-d/s} \Vert u \Vert_{{\mathcal A}_B^s}^{d/s}\;, 
\qquad\forall\ n\ge0} \;,
$$
where the constant $C_*$ depends only on $\alpha_*$, $\mu$ and $\rho$. \endproof
\end{theorem}

This result is ``optimal'' in that the number of active degrees of
freedom is {governed}, up to a 
multiplicative constant, by the same law 
{(\ref{eq:nl.gen.2})-\eqref{eq:nlb.200}} as for the best approximation 
of $u$. The optimality of this result is related to the ``sufficiently fast'' growth of the 
active degrees of freedom: the increment of degrees of freedom at each interation may be comparable to 
the total number of previously activated degrees of freedom (geometric growth).

\subsection{{A-ADFOUR:} Aggressive ADFOUR}
%
We now examine Algorithm {\bf A-ADFOUR}, defined in Sect. \ref{subsec:aggress}, which allows the choice of the
parameter $\theta$ as close to 1 as desired. Such a feature is in the spirit of high regularity, or
equivalently a large value of $s$ for $u\in {\mathcal A}_B^s$. This is a
novel approach which combines the contraction property in Theorem \ref{teo:four3} and the key property of uniform
boundedness of the residuals stated in Proposition \ref{prop:unif-bound-res-alg}.

\begin{theorem}[{cardinality of $\Lambda_n$ for {\bf A-ADFOUR}}]\label{T:aggressive-ADFOUR}
Let the assumptions of Property \ref{prop:inverse.alg} and Theorem \ref{teo:four3} be fulfilled, 
and let $u \in {\mathcal A}_B^s$ for some $s>0$.
Then, the growth of the active degrees of freedom produced by {\bf A-ADFOUR} is estimated as follows:
$$
{|\Lambda_n| \le C_* \, J^d \, \Vert u-u_n \Vert^{-d/s} \Vert u \Vert_{{\mathcal A}_B^s}^{d/s}\;, 
\qquad\forall\ n\ge0\;.}
$$
Here, $J$ is the ($\theta$-dependent) input parameter of {\bf ENRICH},
whereas the constant {$C_*$ is independent of $\theta$.}
\end{theorem}

\begin{proof}
At each iteration $n$, the set $\widetilde{\partial\Lambda}_n$ selected by {\bf D\"ORFLER} is minimal, hence by 
(\ref{eq:four.2.5.5bis}), (\ref{eq:nl.gen.2o}) and (\ref{eq:nlb.201}), one has
$$
|\widetilde{\partial\Lambda}_n| \leq 
\left(\sqrt{1-\theta^2} \, \Vert r_n \Vert\right)^{-d/s} \Vert r_n \Vert_{{\mathcal A}_B^s}^{d/s} \ + \ 1 \;.
$$
Using (\ref{eq:four.2.1bis}) and Proposition \ref{prop:unif-bound-res-alg}, this bound becomes
$$
|\widetilde{\partial\Lambda}_n| \lsim
\left(\sqrt{1-\theta^2} \, \tvert u-u_n \tvert\right)^{-d/s} \Vert u \Vert_{{\mathcal A}_B^s}^{d/s} \;.
$$
On the other hand, estimate (\ref{eq:estim-enrich}) for the procedure {\bf ENRICH} yields
$$
|\partial\Lambda_n| \lsim J^d \,
\left(\sqrt{1-\theta^2} \, \tvert u-u_n \tvert\right)^{-d/s} \Vert u \Vert_{{\mathcal A}_B^s}^{d/s} \;.
$$
Now, as in the proof of Corollary \ref{cor:optim1},
\begin{equation}\label{eq:optim14}
|\Lambda_n| \lsim J^d \, (1-\theta^2)^{-d/s} 
\left( \sum_{k=0}^{n-1} \tvert u-u_k \tvert^{-d/s}\right) \Vert u \Vert_{{\mathcal A}_B^s}^{d/s} \;.
\end{equation}
The contraction property of Theorem \ref{teo:four3} yields for $0 \leq k \leq n-1$
$$
\tvert u-u_n \tvert \le \bar{\rho}^{n-k} \tvert u-u_k \tvert \;,
$$
with $\bar{\rho}=C_0\sqrt{1-\theta^2}<1$ (see \ref{eq:aggr3}); thus,
$$
\sum_{k=0}^{n-1} \tvert u-u_k \tvert^{-d/s}
\le \tvert u-u_n \tvert^{-d/s}
\sum_{k=0}^{n-1} \bar{\rho}^{\frac{d}{s}(n-k)}\lsim
\bar{\rho}^{\frac{d}{s}} \tvert u-u_n \tvert^{-d/s} \lsim (1-\theta^2)^{d/s} \Vert u-u_n \Vert^{-d/s} \;.
$$
Substituting into (\ref{eq:optim14}), the powers of $1-\theta^2$
cancel out, and the {asserted estimate follows.}
\end{proof}

\section{Optimality properties of adaptive algorithms: exponential case}\label{sec:adfour-coarse}

From now on, let us assume that $u \in {\mathcal A}_G^{\eta,t}$ for some $\eta>0$ and $t \in (0,d]$. 
{Let us first} observe that none of the arguments that led to the complexity estimates of the 
previous section can be extended to the present situation.

For {\bf ADFOUR} with moderate D\"orfler marking, Corollary \ref{cor:optim1} in which $\phi^{-1}$ is replaced by
its logarithmic expression yields a bound for $|\Lambda_n|$ which is at least $n$ times larger than the optimal bound
$$
|\Lambda_n^{\rm best}| \leq \kappa \, \frac{\omega_d}{\eta^{d/t}}
\left( \log \frac{\Vert u \Vert_{{\mathcal A}_G^{\eta,t}}}{\varepsilon_n} \right)^{d/t}
$$
for the given accuracy $\varepsilon_n$ (see the proof of Proposition \ref{prop:opt-inner-loop} for more
details, in a similar situation). Manifestedly, the first cause of non-optimality is the crude bound (\ref{eq:optim11}), 
which in this case is no longer absorbed by the summation of a geometric series as in the algebraic case.

On the other hand, for {\bf A-ADFOUR} a sharp estimate of the increment $|\widetilde{\partial\Lambda}_n|$ is indeed 
used in the proof of Theorem \ref{T:aggressive-ADFOUR}, but this involves the sparsity class of the residual, which in the
exponential case may be different from that of the solution, as discussed in Sect. \ref{subsec:spars-res-exp}.

Incorporating a coarsening step in the algorithms allows us to avoid, at least in part, these drawbacks. For these reasons,
herafter we investigate the optimality properties of the two algorithms with coarsening presented in 
Sect. \ref{sec:plain-adapt-alg}

\subsection{{C-ADFOUR:} ADFOUR with coarsening}
Let us {now} discuss the complexity of Algorithm {\bf C-ADFOUR}, defined in Sect. \ref{subsec:coarse-adfour}. 
The following optimal result holds.

\begin{theorem}[{cardinality of $\Lambda_n$ for {\bf C-ADFOUR}}]\label{teo:optim.adcoars.1} 
Assume that the solution $u$ belongs to ${\mathcal A}^{\eta,t}_G$, for some $\eta >0$ and
$t \in (0,d]$. Then, there exists a constant $C>1$ such that the cardinality of the set 
$\Lambda_n$ of the active degrees of freedom produced by {\bf
  C-ADFOUR} satisfies the bound
\begin{equation}\label{Lambda-n}
|\Lambda_n| \leq \kappa \frac{\omega_d}{\eta^{d/t}}
\left( \log  \frac{\Vert u \Vert_{{\mathcal A}^{\eta,t}_G}}{\Vert u-u_n \Vert} + \log C \right)^{d/t} \;, 
\qquad\forall\ n\ge0.
\end{equation}
\end{theorem}
\begin{proof} Since each Galerkin approximation $u_{n+1}$ comes just after a call 
{$\Lambda_{n+1}:= {\bf COARSE} (u_{n,k+1}, \vare_n)$
with threshold 
$\varepsilon_n=\alpha_*^{-1/2}\|r_{n,k+1}\|\ge\|u-u_{n,k+1}\|$}, 
Theorem \ref{T:coarsening} yields
\[
|\Lambda_{n+1}| \leq \frac{\omega_d}{\ \eta^{d/t}} 
\left(\log\frac{\Vert u \Vert_{{\mathcal A}^{\eta,t}_G}}{\varepsilon_n}\right)^{d/t}\!\!\!+1 \;.
\]
On the other hand, (\ref{eq:four.1bis}) and Property \ref{prop:cons-coarse} yield
\begin{equation}\label{u-un+1}
{\Vert u-u_{n+1} \Vert \le \alpha_*^{-1/2} \tvert u-u_{n+1} \tvert
\leq 3 (\alpha^*/\alpha_*)^{1/2}\varepsilon_n.}
\end{equation}
{Since $n\ge-1$, this} gives the result, up to a shift in the
index.  \end{proof}

\medskip
Next, we investigate the optimality of each inner loop. We already know from 
Theorem \ref{teo:adfour-coarse-convergence} that the number $K_n$ of inner iterations is
bounded independently of $n$. So, we just estimate the growth of degrees of freedom when going from $k$ to $k+1$.
We only consider the case of a moderate D\"orfler marking, i.e., we 
{subject} $\theta$ to the condition stated in Lemma
\ref{lem:optim2} (since the case of $\theta$ close to 1 will be covered in the next subsection). The following result
holds.

\begin{proposition}[{cardinality of $\Lambda_{n,k}$ for {\bf C-ADFOUR}}]\label{prop:opt-inner-loop} 
Assume that $u \in {\mathcal A}^{\eta,t}_G$ for some $\eta >0$ and
$t \in (0,d]$, and that the marking parameter satisfies $\theta \in (0,\theta_*)$, where
$\theta_*=\sqrt{\frac{\alpha_*}{\alpha^*}}$. Then, there exist constants $C>1$ and $\bar{\eta} \in (0,\eta]$
such that, for all $n \geq 0$ and all $k=1, \dots , K_n$, one has
$$
|\Lambda_{n,k}| \leq \kappa \frac{\omega_d}{\bar{\eta}^{d/t}}
\left( \log  \frac{\Vert u \Vert_{{\mathcal A}^{\eta,t}_G}}{\Vert u-u_{n+1} \Vert} + \log C \right)^{d/t} \;. 
$$
\end{proposition}
\begin{proof} Each inner loop of {\bf C-ADFOUR} can be viewed as a truncated version of {\bf ADFOUR}; hence, the 
analysis of this algorithm given in Sect. \ref{subsec:moderate-adfour-alg} can be adapted to the exponential case.
In particular, for each increment $\partial\Lambda_{n,j}$ of degrees of freedom, Proposition \ref{prop:optim1} 
gives 
$$
|\partial\Lambda_{n,j}| \le \frac{\omega_d}{\ \eta^{d/t}} 
\left(\log\frac{\Vert u \Vert_{{\mathcal
        A}^{\eta,t}_G}}{\varepsilon_{n,j}}\right)^{d/t}\!\!\!+1 \; ,
{\qquad\forall\ 0\le j \le K_n.}
$$
Since,
{$\varepsilon_{n,K_n} \leq \rho^{K_n-j}\varepsilon_{n,j}$ 
by (\ref{eq:optim12bis}), it follows that}
$$
|\partial\Lambda_{n,j}| \le \frac{\omega_d}{\ \eta^{d/t}} 
\left(\log\frac{\Vert u \Vert_{{\mathcal A}^{\eta,t}_G}}{\varepsilon_{n,K_n}} + (K_n-j)|\log \rho| \right)^{d/t}\!\!\!+1 \;.
$$ 
Thus, recalling that $t \leq d$ by assumption, we have
\begin{eqnarray*}
|\Lambda_{n,k}|^{t/d} &\leq& |\Lambda_{n}|^{t/d} + \sum_{j=0}^{k-1} |\partial\Lambda_{n,j}|^{t/d} \\
&\leq& |\Lambda_{n}|^{t/d} + \kappa \frac{\omega_d^{t/d}}{\eta}
\left(k \log\frac{\Vert u \Vert_{{\mathcal A}^{\eta,t}_G}}{\varepsilon_{n,K_n}} + O(K_n^2) |\log \rho| \right) \;.
\end{eqnarray*}
{Combining  \eqref{eq:adgev.2}, \eqref{Lambda-n},
  and \eqref{u-un+1} with
$k \leq K_n \lsim 1$, we conclude the assertion with
$\bar\eta \le \eta/(1+K_n)$.}
\end{proof}

We remark that the previous result provides a complexity bound, relative to the sparsity class ${\mathcal A}^{\eta,t}_G$
of the solution, which is optimal with respect to the index $t$, 
{but} suboptimal with respect to the index $\bar{\eta}< \eta$.

\subsection{{PC-ADFOUR:} Predictor/Corrector ADFOUR}
At last, we discuss the optimality of Algorithm {\bf PC-ADFOUR}, presented in the second part of
Sect. \ref{subsec:coarse-adfour}.

\begin{theorem}[{cardinality of {\bf PC-ADFOUR}}]\label{teo:pc-adfour}
Suppose that $u \in {\mathcal A}^{\eta,t}_G$, for some $\eta >0$ and $t \in (0,d]$. 
Then, there exists a constant $C>1$ such that the cardinality of the set 
$\Lambda_n$ of the active degrees of freedom produced by {\bf PC-ADFOUR} satisfies the bound 
$$
|\Lambda_n| \leq \kappa \frac{\omega_d}{\eta^{d/t}}
\left( \log  \frac{\Vert u \Vert_{{\mathcal A}^{\eta,t}_G}}{\Vert u-u_n \Vert} + \log C \right)^{d/t} \;, 
{\qquad\forall\ n\ge0.}
$$
If, in addition, the assumptions of Proposition \ref{prop:unif-bound-res-exp} are satisfied, then 
the cardinality of the intermediate sets $\widehat\Lambda_{n+1}$ activated in the predictor step
can be estimated as
$$
|\widehat\Lambda_{n+1}| \leq | \Lambda_n | +
\kappa J^d \frac{\omega_d^2}{\bar{\eta}^{d/\bar{t}}}
\left( \log  \frac{\Vert u \Vert_{{\mathcal A}^{\eta,t}_G}}{\Vert u-u_n \Vert} 
+ \big|\log\sqrt{1-\theta^2}\,\big| + \log C \right)^{d/\bar{t}} \;, 
{\qquad\forall\ n\ge0\;,}
$$
where $J$ is the input parameter of {\bf ENRICH}, and $\bar{\eta}\leq \eta$, $\bar{t}\leq t$ 
are the parameters which occur in the thesis of Proposition \ref{prop:unif-bound-res-exp}.
\end{theorem}
\begin{proof} The proof of the first bound is the same as 
{that} of Theorem \ref{teo:optim.adcoars.1}.
Concerning the second bound, {we invoke Proposition
  \ref{prop:unif-bound-res-exp} to write
$r_n \in \mathcal{A}_G^{\bar\eta,\bar t}$ and  recall that 
$\|r_n-P_{\widetilde{\partial\Lambda_n}} r_n\|\le (1-\theta^2)^{1/2}\|r_n\|$
for each iteration $n$. This, combined with the minimality of
the set $\widetilde{\partial\Lambda}_n$ selected 
by {\bf D\"ORFLER}, yields}
$$
|\widetilde{\partial\Lambda}_n| \leq 
 \frac{\omega_d}{\bar{\eta}^{d/\bar{t}}}
\left( \log  \frac{\Vert r_n \Vert_{{\mathcal A}^{\bar{\eta},\bar{t}}_G}}{\sqrt{1-\theta^2}\Vert r_n \Vert} 
\right)^{d/\bar{t}} \!\! + 1\;.
$$
Estimate (\ref{eq:estim-enrich}) for {\bf ENRICH} yields
$$
|{\partial\Lambda}_n| \leq \kappa J^d \, \frac{\omega_d^2}{\bar{\eta}^{d/\bar{t}}}
\left( \log  \frac{\Vert r_n \Vert_{{\mathcal A}^{\bar{\eta},\bar{t}}_G}}{\sqrt{1-\theta^2}\Vert r_n \Vert} 
\right)^{d/\bar{t}} \;.
$$
Using {(\ref{eq:four.2.1})} and Proposition
\ref{prop:unif-bound-res-exp}, {this time to replace $r_n$ by $u$
and $u-u_n$,}
we obtain the {desired} result.  \end{proof}

We observe that in the case $\bar\eta<\eta$ and $\bar t< t$, 
the cardinalities $|\widehat\Lambda_{n+1}|$ and $| \Lambda_n|$ are not bounded by comparable quantities.
This looks like a non-optimal result, yet it appears to be intimately related to the fact that in general the residuals
belongs to a worse sparsity class than the solution. 

\section*{Acknowledgements}
We wish to thank Dario Bini for providing us with Property 
\ref{prop:inverse.matrix-estimate}, 
and Paolo Tilli for insightful discussions on the exponential classes.\\
The first and the third author have been partially supported by the Italian research fund PRIN 2008 
``Analisi e sviluppo di metodi numerici avanzati per EDP''. The second author has been partially supported by NSF grants DMS-0807811 and DMS-1109325.


\end{document}